\documentclass[10pt]{article}
\usepackage{newinutile}
\usepackage{cite}

\usepackage{latexsym,amsmath,amsfonts,amsthm,amssymb,bbm,amsbsy}
\usepackage{epsfig,graphics,color,calc,graphicx,pict2e}
\usepackage{comment} %Per commentare molte righe di codice latex
\usepackage{multirow}
\usepackage[T1]{fontenc}
\usepackage[utf8]{inputenc}
\usepackage[ruled]{algorithm2e}
\usepackage{float} %Necessario per usare {figure} con multicolumn: usa
\usepackage{mathrsfs}
\usepackage{tikz}
\usetikzlibrary{shapes,arrows}
\usetikzlibrary{intersections,matrix, positioning}

\usepackage{subfigure} 
\newtheorem{theorem}{Theorem}

\newtheorem{lemma}{Lemma}

\def\theequation{\arabic{section}.\arabic{equation}}
\def\thefigure{\arabic{figure}}
\newcommand{\newatop}[2]{\genfrac{}{}{0pt}{}{#1}{#2}}
\newcommand{\puno}{{\pmb{+1}}}
\newcommand{\zero}{{\pmb{0}}}
\newcommand{\muno}{{\pmb{-1}}}

\newlength{\pecettawidth}
\begin{document}
%%% Font: commenta la riga che segue se vuoi i font standard
%\fontfamily{ppl}\selectfont
%\logo
\title{\normalsize\Large\bfseries Particle transport based study of nucleation in a ferromagnetic three-state spin system with conservative dynamics}

\author{Emilio N.M.~Cirillo}
\affiliation{Dipartimento di Scienze di Base e Applicate per l'Ingegneria, 
             Sapienza Universit\`a di Roma, 
             via A.\ Scarpa 16, I--00161, Roma, Italy.}
\email{emilio.cirillo@uniroma1.it}

\author{Vanessa~Jacquier}
\affiliation{Institute of Mathematics,
University of Utrecht, Budapestlaan 6, 3584 CD Utrecht, The~Netherlands.}
\email{v.jacquier@uu.nl} 

\author{Cristian~Spitoni}
\affiliation{Institute of Mathematics,
University of Utrecht, Budapestlaan 6, 3584 CD Utrecht, The~Netherlands.}
\email{C.Spitoni@uu.nl}

\date{\empty} %Per togliere la data

\maketitle

\begin{abstract}
We pose the problem of metastability for a three--state spin system with conservative 
dynamics. We consider the Blume--Capel model with the Kawasaki dynamics, 
we prove that, in a particular region of the parameter plane, the metastable
state is the unique homogeneous minus state, and we estimate the exit 
time. 
To achieve our goal we have to solve several variational problems 
in the configuration space 
which result to be particularly involved, due to complicated structure 
of the trajectories. 
They key ingredient is the control 
of the energy differences between the configurations crossed when a spin is transported
from the boundary to an internal site  of the 
lattice through a completely arbitrary mixture of the three--state spin species.
To master these mechanisms we have introduced a new approach based on 
the transport of spins along nearest neighbor connected regions of the 
lattice with constant spin configuration. This novel approach goes beyond the Blume--Capel model
and can be used for the study of more general multi--state spin models.
\end{abstract}

%\pacs{}


%\ams{}

\preprint{Bozza: \today}  

%\vfill\eject
%\noindent
%\textbf{MSC2000:} 82B28; 82B44; 60K35.

\section{Introduction}
\label{s:int}
\par\noindent
Metastable states are very common in nature and are observed 
in several physical systems. 
Their mathematical description is a challenging task 
that, in the last decades, has given rise to different approaches and
has been attacked from different angles.
For a systematic discussion of the literature related 
to the so--called \emph{pathwise} and \emph{potential--theoretic} 
approaches
we refer to the monografies \cite{olivieri2005large, bovier2016metastability}
and to the review \cite{cjs2022}.
For the more recent 
\emph{trace method} we refer, for instance, 
to \cite{beltran2015martingale}.

The pertinent literature is characterized by the alternation of 
abstract studies \cite{bet2021effect, manzo2004essential}, in which general theories are proposed or 
refined, and applied ones in which the behavior of special
models is discussed. The latter, beside their interest due 
to specific novel applications, see, e.g., \cite{alonso1996three,den2003droplet} for the 3D Ising model, \cite{apollonio2022metastability, baldassarri2023metastability} for applications to non-square lattices, and \cite{baldassarri2023ising} for the Ising model on a family of finite
networks, play an important role in 
suggesting possible progresses of the general theories.  

This is a paper of the second type in which we investigate
the metastable behavior of the Blume--Capel lattice spin 
model \cite{blume1966theory,capel1966possibility}
under the 
Kawasaki dynamics \cite{kawasaki1966diffusion}
in finite volume and in the zero temperature limit.
The model is characterized by the fact that
the lattice spin variables 
can assume the three states 
minus, zero, and plus, and the Hamiltonian is ferromagnetic.
Moreover, 
the Kawasaki dynamics, reversible with respect to the 
Gibbs measure, allows in the bulk 
only the swaps between neighboring spins. 
This is a particularly challenging problem 
since it combines the difficulties due to the three state character 
of the spins and those due to the conservative character of the dynamics.

Before entering into the details of our work, we review 
briefly the papers in which the metastable behavior of the
Blume--Capel model with the Glauber 
non--conservative dynamics has been investigated.
The first results were published 
in \cite{cirillo1996metastability} 
in finite volume and in the zero-temperature limit. 
In \cite{cirillo2013relaxation,landim2016metastability,cirillo2017sum} 
a peculiar version of the model, for which there exist two non--degenerate 
in energy metastable states, is studied. 
The infinite volume regime is considered in 
\cite{manzo2001dynamical,landim2019metastability}. 
We mention, also, that in the interesting 
paper \cite{FGRN1994}, this model has been studied
at finite temperature with non--rigorous methods, such as 
Monte Carlo simulations and transfer matrix approaches.
Moreover, in different regimes, thanks to its great ductility,
the Blume--Capel model
and its generalizations have been recently used to study the 
pattern formation in three--components mixtures of 
polymers and solvent \cite{lcm2024NARWA,mkcmsmc2022pre,lmcm2023physicad}.
In all these studies the model is defined on 
a torus, namely, 
periodic boundary conditions are considered. 

In our recent work \cite{cirillo2024homogeneous} we 
attack this model
with zero boundary conditions and discuss the metastability 
scenario on the parameter plane, comparing it with what was already 
known for periodic boundary conditions. 
In particular, we prove that, 
depending on the choice of the parameters, the system can exhibit 
both homogeneous and heterogeneous nucleation triggered by the boundary. 
Beside its intrinsic interest,
\cite{cirillo2024homogeneous} is a sort of 
prequel of the present paper in which 
we study the Kawasaki case
where zero boundary conditions are natural
for modelling the exchange of particles 
between the system and the reservoir.
%where, as we shall 
%explain below, 
%in order to treat the Kawasaki dynamics, 
%we are forced to consider 
%zero boundary conditions. 

The rigorous study of metastability for a Kawasaki dynamics has been a
   challenging problem from the very beginning: the number of particles is
   indeed conserved in the interior of a finite  box, so that during the
   nucleation particles must travel between the droplet and the boundary of
   the box, which causes several mathematical complications.  A first
   simplified local 2D model for a Ising lattice gas was introduced
   in \cite{hollander2000metastability}, where only the particles inside a finite box
   %$\bar{\Lambda}_{0}$ 
   were considered for the interaction, while the
   particles outside evolved  via non-interacting random walks. The
   removal of the interaction outside the finite box
   %$\bar{\Lambda}_{0}$ 
   allowed the
   mathematical control of the gas of particles, and it was consistent with
   the physical picture of an ideal gas approximation of a low density gas in
   the limit in which the temperature tends to zero.
   %$\beta \rightarrow+\infty$.

   The crucial point was indeed that at low density, say proportional to the 
   exponential of minus the inverse of the temperature, 
   %$e^{-\beta \Delta}$ 
   the
   gas outside the finite box
   %$\bar{\Lambda}_{0}$ 
   could be treated as a reservoir
   that creates particles with rate equal to the small density
   %$e^{-\beta\Delta}$ 
   at sites on the
   interior boundary of the finite box
   %outside $\bar{\Lambda}_{0}$  
   and annihilates
   particles with rate one 
   %$1$ 
   at the external boundary sites.
   %on the exterior boundary of the finite box itself.
   %$\bar{\Lambda}_{0}$.

   The local dynamics results in \cite{hollander2000metastability} for the Ising lattice gas were extended to the 3D case
   in \cite{den2003droplet}, with the sharp asymptotics  given in \cite{bovier2006sharp},
   and to the anisotropic Ising lattice gas in \cite{nardi2005anisotropy,baldassarri2022,baldassarri2021metastability,baldassarri2022critical}.

  Moreover, the case of volumes growing moderately fast as the temperature
   decreases was first studied in \cite{gaudilliere2009ideal} by using the pathwise approach,
   then in \cite{bovier2010homogeneous} with the use of the potential theoretic approach, and finally
   in \cite{gl2015annals} with the trace method. All these results have been derived for two
   spin classical lattice gases with Kawasaki dynamics.

To our knowledge, 
\cite{den2012metastability} is
the sole paper in which the metastable behavior of
a three state system with a conservative dynamics has been approached.
In that paper a swap
dynamics is considered for a model with three state site variables
(zero, one, and two) in which 
direct swaps between the states one and two are not allowed.
Moreover, the Hamiltonian is not ferromagnetic in the sense that
it promotes the interface between states one and two, with respect to all 
other possible bond configurations, which are all alike;
thus the single interface one--two is favored with respect 
to all the others. 
This yields in a chessboard--like stable 
state with the sites of the even and odd sublattices 
occupied, respectively, by state one and two or viceversa.

In the present paper, in contrast with 
\cite{den2012metastability}, 
we approach  
the Blume--Capel model with the Kawaski dynamics 
allowing swaps between spins of any value.
Moreover, the ferromagnetic interaction 
favors the three type of bonds between
sites sharing the same spin, whereas
the bonds between different spins are 
not alike, but pay different energy costs, indeed, the 
minus--plus bond costs four times the zero--minus and the zero--plus bonds.
This fact complicates the study of the model:
due to the interface energy cost hierarchy, 
plus structures cannot grow freely inside minus regions, but 
must always be protected by a thin layer of zeros. 
In other words, 
in order to 
perform changes of the system configuration in which a plus 
structure grows inside a minus background, before substituting 
bulk pluses with minuses, it is necessary to 
let zeros populate the bulk to avoid direct interfaces between 
minus and plus regions. Controlling these 
mechanisms becomes 
utterly difficult in the case of conservative 
dynamics, 
since the spins cannot be 
locally created, as it is the case for non--conservative ones,
but they must first enter the lattice at its 
boundary and then be transported 
through the bulk to the spot 
where they are needed.

The complications that arise are due to the necessity to 
control energy differences along path of configurations 
corresponding to the transport of one spin (either minus, zero, or plus)  
along the lattice through a completely general mixture 
of the three spin species.  
In this paper we manage these mechanisms by means 
of new techniques 
based on the idea that 
spins are transported through the lattice along 
regions made of nearest neighbor connected sites 
with constant spin value. It is important to stress 
that the energy cost associated to the transport of 
spins along the lattice is controlled in Lemma \ref{thm:grafo_pesato} 
which is 
susceptible of being 
extended to more general multi--spin models, 
such as a generalized Potts model with different 
interface energy costs, see Lemma \ref{thm:grafo_pesato_canale_Potts}.
Thus, the results discussed in this paper can open the way 
to new studies of conservative dynamics for general multi--state 
spin systems.

Coming back to this work, we focus on
a region of the parameter plane 
where we are able to prove the uniqueness of the metastable state 
(the homogeneous minus configuration) and 
to study the transition to the unique 
stable configuration (the homogeneous plus state), 
that is to say, to the absolute minimum of the energy.
This is achieved in the framework of the pathwise approach as 
formulated in \cite{manzo2004essential} by solving 
two model dependent problems:
call the \emph{stability level} of a configuration 
the minimal energy barrier that has to be overcome
by a path connecting such a configuration to any other 
at lower energy, then 
compute the stability level of the homogeneous minus configuration 
and show that the stability level of any other configuration 
is strictly smaller.
These two problems are addressed in the literature respectively as the 
optimal path and the recurrence problem. 

Due to the very intricated structure of the trajectories of our 
three--state conservative model, we solve a weaker version of 
the optimal path problem, that is to say, we do not compute exactly
the stability level of 
the homogeneous minus configuration, but we show that it  
belongs to a closed interval. As for the 
recurrence property, we show that the stability level of all 
other configurations is smaller than the infimum of this interval. 
This allows us to prove that the homogeneous minus configuration 
is the metastable state and to prove an estimate for the exit time 
from it with an uncertainty related to the width of the interval. 
We remark that, due to the uncertainty on the value of the 
stability level of the homogeneous minus configuration, it is natural 
to use the pathwise approach, since the other methods need a 
more detailed control of the energy landscape, i.e., of the minimizers
of the Dirichlet form associated to the dynamics.

The paper is organized as follows. 
In Section~\ref{s:mrs} we introduce the model and state our main results.
Section~\ref{s:pro-th} is devoted to the proof of the theorems which are based 
on the key Lemma \ref{thm:grafo_pesato} about the particle transport. 
The proofs of the more technical lemmas are reported in 
Section~\ref{s:pro-le}.
In Section~\ref{s:con}
we summarize 
our conclusions, propose a generalization of our 
key lemma, and discuss its potential applications to general 
multi--state spin systems. 

\section{Model and results}
\label{s:mrs}
\par\noindent
In this section we introduce the Blume--Capel model with the Kawasaki dynamics
and state our main results.

\subsection{The lattice}
\label{s:lat}
\par\noindent
We consider the set $\mathbb{Z}^2$ embedded in $\mathbb{R}^2$ 
and call \emph{sites} its elements.
We will denote by $\{e_1,e_2\}$ the canonical base of $\mathbb{R}^2$ 
and we will often address to the direction $1$ (resp.\ $2$) as the 
horizontal (resp.\ vertical) direction. 
Given two sites $x,x'\in\mathbb{Z}^2$ we let $|x-x'|$ be their 
Euclidian distance. 
Given $x\in\mathbb{Z}^2$, we say that 
$x'\in\mathbb{Z}^2$ is a \emph{nearest neighbor} 
of $x$ if and only if $|x-x'|=1$. 
Pairs of neighboring sites will be called \emph{bonds}.
A set $X\subset\mathbb{Z}^2$ is \emph{connected} if and 
only if for any $x\neq x'\in X$ there exists a sequence 
$x_1,x_2,\dots,x_n$ of sites of $X$ such that 
$x_1=x$, 
$x_n=x'$, and 
$x_k$ and $x_{k+1}$ are nearest neighbors for any $k=1,\dots,n-1$.

Given $X\subset \mathbb{Z}^2$ we call 
\emph{internal boundary} $\partial^-X$ of $X$
the set of sites in $X$ having a nearest neighbor outside 
$X$. 
The \emph{interior} or \emph{bulk} 
of $X$ is the set $X\setminus\partial^-X$, 
namely, the set of sites of $X$ having four nearest neighbors 
inside $X$. 
We call 
\emph{external boundary} $\partial^+X$ of $X$
the set of sites in $\mathbb{Z}^2\setminus X$ having 
a nearest neighbor inside $X$. 

A \emph{column}, resp.\ a \emph{row}, of $\mathbb{Z}^2$ is a 
sequence of $L$ connected sites of $\Lambda$ such that the line 
joining them is parallel to the vertical, resp.\ horizontal, axis.

A set $R\subset\mathbb{Z}^2$ is called a \emph{rectangle} 
(resp.~\emph{square}) of $\mathbb{Z}^2$ if the union of the 
closed unit squares of $\mathbb{R}^2$ centered in
the sites of $R$ with sides parallel 
to the axes of $\mathbb{Z}^2$ is a rectangle (resp.~a square) 
of $\mathbb{R}^2$.
The \emph{sides} of a rectangle of $\mathbb{Z}^2$ 
are the four maximal connected 
subsets of its internal boundary (note that they lie on straight 
lines parallel to the axes of $\mathbb{Z}^2$).
The \emph{length} of one side of a rectangle of $\mathbb{Z}^2$
is the number of sites 
belonging to the side itself. 
A \emph{quasi--square} is a rectangle with side lengths equal to $n$ and $n+1$,
with $n$ an integer greater than or equal to one.

We equip $\mathbb{Z}^2$ with a directed graph structure by considering 
the set of directed edges $(x \to y)$ where $\{x,y\}\subset \mathbb{Z}^2$ 
such that $|x-y|=1$. 

\subsection{The configuration space}
\label{model_configuration}
\par\noindent
Let $\Lambda:=\{0,...,L+1\}^2 \subset \mathbb{Z}^2$
be a finite square with fixed side length $L+2$
and denote its interior by 
$\Lambda_0:= \Lambda \setminus \partial^- \Lambda=\{1,...,L\}^2$.
Remark that $\partial^+\Lambda_0$ is not equal to 
$\partial^-\Lambda$, since the four corner sites of $\Lambda$
belong to 
$\partial^-\Lambda$, but not to 
$\partial^+\Lambda_0$.
We denote by
$E\subset\Lambda\times\Lambda$ 
the collection of all the directed edges $(x\to y)$ 
such that $x,y\in\Lambda$
and 
$E_0:=\{(x \to y) \in E | x,y\in\Lambda_0\}$ 
the set of oriented edges with both vertices in $\Lambda_0$. 

With each $x \in \Lambda$ we associate a spin variable 
$\sigma(x)\in\{-1,0,+1\}$. Moreover, we let 
$\mathcal{X}:=\{{-1,0,+1}\}^{\Lambda}$ be the 
\emph{configuration space} and 
$\hat{\mathcal{X}}:=\{\sigma\in\mathcal{X}\,|
               \,\sigma(x)=0\, \forall x\in\partial^-\Lambda \}$ 
be the set of configurations with spin $0$ associated to the 
sites in the internal boundary of $\Lambda$.

Given a configuration $\sigma\in\mathcal{X}$, if $\sigma(x)=0$ 
we say that the site is \emph{empty}, otherwise, 
if $\sigma(x)=+1$ (resp.\ $\sigma(x)=-1$),  
we say that it is \emph{occupied} by a particle 
with spin plus (resp.\ minus). 

Given a configuration $\sigma\in\mathcal{X}$ and a set 
$A \subseteq \Lambda_0$, we denote $\sigma_{|A}$ by $\sigma_A$. 
Finally, given $s\in\{-1,0,+1\}$ and $A \subseteq \Lambda_0$, we 
denote by $s_A$ the homogeneous configuration $\sigma \in \{-1,0,+1\}^A$ 
with spin equal to $s$ for any site in $A$. 
The symbols $\muno$, $\zero$, $\puno$ denote, respectively, 
the configurations with spins zero in $\partial^+ \Lambda_0$ 
and spins 
$-1$, $0$, and $+1$  
in $\Lambda_0$, namely the homogeneous configuration in 
the interior of $\Lambda$.

\subsection{Hamiltonian of the model and assumptions on its parameters}
\label{model_BCK}
\par\noindent
The Blume--capel Hamiltonian is defined in terms of three real parameters
$J>0$, $\lambda$, and $h$, respectively called \emph{coupling constant},
\emph{chemical potential}, and \emph{magnetic field}.
It is also convenient to set 
\begin{equation}
\label{eq:parameters_delta}
\Delta^p=4J-(\lambda+h)
\;\;\textrm{ and }\;\;
\Delta^m=4J-(\lambda-h)
.
\end{equation}
For any $s \in \{-1,+1\}$, we define 
\begin{equation}
    \Delta_s= \frac{1-s}{2} \Delta^m + \frac{1+s}{2} \Delta^p. 
\end{equation}
We now introduce a particular version of the 
Hamiltonian of the Blume--Capel model which is well suited to treat 
the particle exchange between the bulk and the boundary:
\begin{equation}
\label{def:hamiltonian_BCK}
H(\sigma)= H_0(\sigma)+\Delta^p{ \sum_{i\in \partial^+\Lambda_0}}
    {\mathbf{1}}_{\{\sigma(i)=+1\}}
+\Delta^m{\sum_{i\in \partial^+ \Lambda_0}}
    {\mathbf{1}}_{\{\sigma(i)=-1\}} 
\end{equation}
where
\begin{equation}
\label{def:hamiltonian_Glauber}
H_0(\sigma)
=
\frac{J}{2}\sum_{\newatop{i,j\in\Lambda_0:}{|i-j|=1}}[\sigma(i)-\sigma(j)]^2
+J\sum_{i\in\partial^-\Lambda_0}
 \sum_{\newatop{j\in\mathbb{Z}^2\setminus\Lambda_0:}{|i-j|=1}}
[\sigma(i)]^2
-\lambda\sum_{i\in\Lambda_0}\sigma(i)^2
-h\sum_{i\in\Lambda_0}\sigma(i) .
\end{equation}
Note that if in a configuration $\sigma$ a spin zero surrounded by 
zeros is changed to plus or minus the value of the energy 
$H_0$ increases precisely of $\Delta^p$ and $\Delta^m$, respectively.
Thus, these two quantities can be interpreted as the energy 
cost of a plus and a minus spin in the sea of zeros. 

We note that if $\sigma$ is a configuration 
such that the external boundary of $\Lambda_0$ is filled 
with zero spins, then $H(\sigma)=H_0(\sigma)$.
It is interesting to remark that
the function $H_0$ in  
\eqref{def:hamiltonian_Glauber} is precisely the Hamiltonian 
of the Blume--Capel model that has been used in 
\cite{cirillo2024homogeneous}
to study the Blume--Capel model with Glauber dynamics in the case of 
zero boundary conditions.
Here, we consider the Hamiltonian \eqref{def:hamiltonian_BCK} 
to take into account the energy cost of particles that must be created in 
the boundary $\partial^+\Lambda_0$ to run the Kawasaki dynamics
mimicking the presence of an external gas with fixed density of 
plus and minus spins. 

The equilibrium state of the Blume--Capel model with zero boundary 
condition on $\Lambda$ and at temperature $1/\beta$, with $\beta$ 
a positive real number, is described by the Gibbs measure 
\begin{equation}
\label{def:gibbs_measure}
    \mu_{\beta}(\sigma)=\frac{e^{-\beta H(\sigma)}}{Z_{\beta}},
\end{equation}
where $Z_{\beta}:=\sum_{\eta \in \mathcal{X}} e^{-\beta H(\eta)}$
is the partition function.

In the sequel we shall discuss the metastable behavior of 
the Blume--Capel model, but, to do so, 
we will rely on the following assumptions 
on the parameters\footnote{With the notation $0<a\ll b$ we mean
$0<a< cb$ for some positive constant $c>1$ that we are not 
interested to compute exactly.} (see, also, Figure~\ref{f:param}):

\begin{align}
\label{mod007a}
i)
&
\,\lambda>0 
\;\textup{ and }\;
\frac{\lambda}{2}<h<\min\Big\{\lambda,\frac{1}{3}J-\lambda\Big\}
\\ 
\label{mod007b}
ii)
&
\,
\lambda<\frac{4J(1-h/\lambda)}{h/\lambda(1+h/\lambda)}
 \Bigg[-\frac{2(1-2h/\lambda)}{h/\lambda}
       +\sqrt{\frac{64}{1+h/\lambda}-\frac{22}{h/\lambda}-6}\Bigg]^{-1}
\;\;\textup{ for }\;\;
\frac{h}{\lambda}>\frac{1}{3}(9-4\sqrt{3})
\\
\label{mod007c}
iii)
&
\,
L>\Big(\frac{2J}{\lambda-h}\Big)^3
\;\textup{ and }\; 
\frac{2J}{\lambda+h}, \,\frac{2J}{\lambda-h},\,\frac{2J+\lambda-h}{2h}
\;\;\textup{ are not integers}.
\end{align}

Condition \eqref{mod007a} implies that 
$\Delta^p$ and $\Delta^m$ are both positive and
will be also used in the proof of Theorem~\ref{thm:out_gammaBCK}.
Conditions \eqref{mod007b}--\eqref{mod007c}
are more technical, in particular \eqref{mod007b}
is related to the uncertainty on the computation
of the stability level of the configuration 
$\muno$, as mentioned already in the 
introduction (see, also, 
Theorem~\ref{thm:vmuno} and the way in which it is used 
in Section~\ref{sec:Identification}).
Note that the above conditions are not completely independent 
one from each other, indeed, in the case 
$\xi$ much larger than $(9-4\sqrt{3})/3$ the condition \eqref{mod007b}
implies the requirements assumed in \eqref{mod007a} on $h$.

\begin{figure}[t]
        \centering
        \includegraphics[scale=0.35]{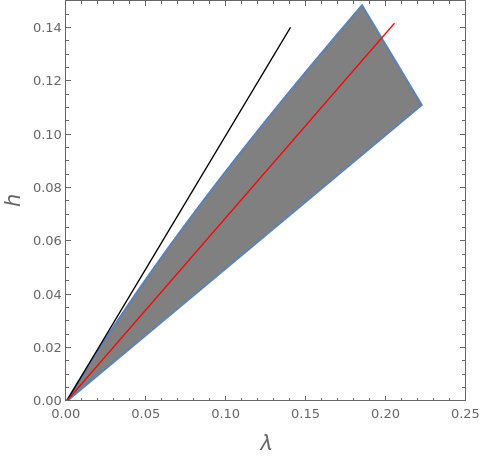}
        \caption{Region of the parameters space compatible with
assumptions \eqref{mod007a}--\eqref{mod007c} for $J=1$. The red and 
the black lines have, respectively, slope $(9-4\sqrt{3})/3$ and $1$.}
        \label{f:param}
    \end{figure}

We close this section by remarking that a simple computation yields that
the Hamiltonian \eqref{def:hamiltonian_BCK} can be rewritten as follows
\begin{equation}
\label{def:hamiltonian_BCK00}
H(\sigma)  =
-J\!\!\!\!
\sum_{(x \to y)\in E_0}
\!\!\!\!\!
    {\mathbf{1}}_{\{\sigma({x})\sigma({y})=+1\}}
+J\!\!\!\!
 \sum_{(x \to y)\in E_0} 
\!\!
    {\mathbf{1}}_{\{\sigma({x})\sigma({y})=-1\}} 
 +\Delta^p
\!\!\!\!
\!\!
{ \sum_{x\in\Lambda_0\cup\partial^+\Lambda_0}}
\!\!\!\!
    {\mathbf{1}}_{\{\sigma(x)=+1\}}
+\Delta^m
\!\!\!\!
\!\!
{\sum_{x\in\Lambda\cup\partial^+\Lambda_0}}
\!\!\!\!
    {\mathbf{1}}_{\{\sigma(x)=-1\}} 
.
\end{equation}
This is an interesting expression pointing out that
the interaction takes effect only inside $\Lambda_0$ and that the 
binding energy associated to a positive bond 
is $-J<0$ (respectively, $J>0$ for a negative bond).

\subsection{Energy landscape}
\label{s:lan}
\par\noindent
The energy difference (\emph{energy cost}) associated with each possible 
swap between two particles of different type plays a crucial role 
in the proof of several results.

Given $\sigma\in\mathcal{X}$, 
we well consider the following transformations:
we denote by $\sigma^{(x,y)}$ 
the configuration obtained by swapping the values of the spins 
at sites $x$ and $y$ in $\sigma$, 
by $\sigma^{(x;0)}$ the configuration obtained from $\sigma$ by 
replacing the value of the spin at site $x$ with $0$,
and 
by $\sigma^{(x;s)}$ the configuration obtained from $\sigma$ by replacing 
the value of the spin at site $x$ with $s \in \{-1,+1\}$,
more precisely, we set
\begin{equation}
\label{sigma_scambio}
\sigma^{(x,y)}(z)= \left\{
\begin{array}{ll}
\!\sigma(z) & \!\!\textrm{if } z \neq x,y\\
\!\sigma(y) & \!\!\textrm{if } z = x \\
\!\sigma(x) & \!\!\textrm{if } z = y, \\
\end{array}
\right.
\;\;
\sigma^{(x;0)}(z)= \left\{
\begin{array}{ll}
\!\sigma(z) & \!\!\textrm{if } z \neq x\\
\!  0 & \!\!\textrm{if } z = x,
\end{array}
\right.
\;\;
\sigma^{(x;s)}(z)= \left\{
\begin{array}{ll}
\!\sigma(z) & \!\!\textrm{if } z \neq x\\
\!   s & \!\!\textrm{if } z = x. 
\end{array}
\right.
\end{equation}

In order to express conveniently the energy differences, 
we introduce the following energy cost functions. 
Let $\sigma \in \mathcal{X}$ and $x\in \Lambda$, we set 
\begin{equation}\label{def:number_interfaces}
    D_x(\sigma):=\sum_{\substack{y \in \Lambda \\ |x-y|=1}} 
\textbf{1}_{\{\sigma(y)\sigma(x)=+1\}} \,  
    -\sum_{\substack{y \in \Lambda \\ |x-y|=1}} 
\textbf{1}_{\{\sigma(y)\sigma(x)=-1\}}
\;\;\textup { if } x\in\Lambda_0
\end{equation}
and 
$D_x(\sigma)=0$
if $x \in \partial^+\Lambda_0$, 
which makes sense 
since, 
according to \eqref{def:hamiltonian_BCK},
in $\partial^-\Lambda$
there is no interaction. 

We note that for $x\in\Lambda_0$, 
if $\sigma(x)=0$, then $D_x(\sigma)=0$, otherwise $D_x(\sigma)$ is 
the difference between the number of nearest neighbours of $x$ with 
spin equal to $\sigma(x)$ and the number of nearest neighbors with 
spin equal to $-\sigma(x)$. 

Furthermore, it is easy to see that 
the energy cost of the swap between the spins in $x$ and $y$ in $\Lambda$ 
is given by
\begin{equation}
\label{def:cost_swap}
    H(\sigma^{(x,y)})-H(\sigma)=2J \Big [ [D_x(\sigma)+D_y(\sigma)]-[D_x(\sigma^{(x,y)})+D_y(\sigma^{(x,y)})] \Big ].
\end{equation}

As we will see below, the homogeneous states $\muno$, $\zero$, and $\puno$ 
will be of basic importance in our study. We, thus, compute their 
energy. Recalling \eqref{def:hamiltonian_BCK}, it follows
\begin{equation}\label{H_0_+1_-1}
    H(\zero)=0, \;
    H(\muno)=%-J|E_0|+\Delta^m|\Lambda_0|=
    -L^2 (\lambda-h)+4JL, \;
    H(\puno)=%-J|E_0|+\Delta^p|\Lambda_0|=
    -L^2 (\lambda+h)+4JL,
\end{equation}
where we have reported only the terms proportional to $L^2$ 
omitting the terms proportional to $L$ and smaller. 

The \emph{ground states} of the system (or of the Hamiltonian)
are the configurations where 
the Hamiltonian 
\eqref{def:hamiltonian_BCK}
attains its absolute
minimum. We let $\mathcal{X}^\mathrm{s}$ be the set of ground states.

\begin{lemma}
\label{t:lan00-5}
Assume condition \eqref{mod007a} is satisfied, 
then 
the homogeneous state $\puno$
is the sole ground state of the 
system, namely $\mathcal{X}^\mathrm{s}=\{\muno\}$.
\end{lemma}

We say that a configuration $\eta\in\mathcal{X}$ 
is a \emph{local minimum} of the Hamiltonian if and only if 
for any $\eta'\neq\eta$ communicating with $\eta$ we have 
$H(\eta')>H(\eta)$.

\begin{lemma}
\label{t:lan000}
Assume \eqref{mod007a} and \eqref{mod007b} 
are satisfied.
The homogeneous states $\zero$ and $\muno$ are local minima of the system. 
\end{lemma}

Based on the above lemma, we can expect that the homogeneous states $\muno$ and $\zero$ are potential metastable states.
 
\subsection{Local Kawasaki dynamics}
\label{Kdynamics}
\par\noindent
We consider a lattice dynamics 
similar to the local Kawasaki dynamics defined 
in \cite{hollander2000metastability} for one type of particles. 
In the bulk spins cannot be freely modified,
as it is the case for the Glauber dynamics, but 
only swaps of the spin values between neighboring sites are allowed. 
The total amount of pluses, minuses, and zeroes is not kept constant since
we let particle to enter the system at the boudary. 
Indeed, in the region $\partial^+\Lambda_0$, which is contained in 
the internal boundary $\partial^-\Lambda$ of 
$\Lambda$, the value of a spin can be changed from zero to plus or 
minus, and viceversa, 
mimicking in this way a swap with the exterior of $\Lambda$, 
namely, $\mathbb{Z}^2\setminus\Lambda$, 
which plays the role of an infinite reservoir. 

We denote by $E^+$ the set obtained by adding 
to $E$ the oriented pairs of neighboring sites such that 
one and only one is inside $\Lambda$.
In other words, 
\begin{equation}
E^+= E\cup\{(x \to y) \, | \, x\in\Lambda, y \in\partial^+\Lambda\} 
      \cup\{(x \to y) \, | \, x\in\partial^+\Lambda, y \in \Lambda\}.
\end{equation}
We say that $\sigma$ and $\eta$ are \emph{communicating configurations}, 
and we denote by $\sigma\sim\eta$, 
if there exists an edge $(x\to y)\in E^+$ such that $\eta$ may 
be obtained from $\sigma$ in any one of these ways:

\begin{itemize}
\item[--] for $(x \to y) \in E^+$, $\eta=\sigma$;
\item[--] for $(x \to y)\in E$, $\eta=\sigma^{(x,y)}$ is the configuration obtained from $\sigma$ by exchanging the spin values in the sites $x$ and $y$;
\item[--] for $(x \to y) \in E^+$ with $y\in \partial^+ \Lambda$, $\eta=\sigma^{(x;0)}$ represents the fact that a spin inside the internal-boundary $\partial^-\Lambda$ is set to zero; % annihilation of one particle along the border;
\item[--] for $(x \to y) \in E^+$ with $x\in \partial^+ \Lambda$ and $\sigma(y)=0$, $\eta=\sigma^{(y;s)}$ represents the fact that a spin inside the internal-boundary $\partial^-\Lambda$ is set to $s \in \{-1,+1\}$. %creation of one particle $s \in \{-1,+1\}$ along the border.
\end{itemize}

The dynamics is modelled by the discrete time Markov chain $\sigma_t\in\mathcal{X}$, with $t\geq 0$, performing jumps among communicating configurations. Precisely, at time $t\geq 1$ choose at random an edge $(x \to y)\in E^+$ with uniform probability, then 
\begin{itemize}
\item[i)]
if $(x \to y)\in E$ then $\sigma_{t+1}=\sigma_t^{(x,y)}$ 
with probability $e^{-\beta[H(\sigma_t^{(x,y)})-H(\sigma_t)]_+}$, otherwise $\sigma_{t+1}=\sigma_t$; 
%where $[\cdot]_+$ is the positive part;
\item[ii)]
if $(x \to y)\not\in E$ and $x\in\Lambda$,
then $\sigma_{t+1}(x)=0$ and
$\sigma_{t+1}(z)=\sigma_t(z)$ for all $z\in\Lambda\setminus\{x\}$ with probability $1/2$, otherwise $\sigma_{t+1}=\sigma_t$;
\item[iii)]
if $(x \to y)\not\in E$ and $y\in\Lambda$ is such that $\sigma(y)=\pm 1$,
then $\sigma_{t+1}=\sigma_t$. Otherwise if $y\in\Lambda$ is such that $\sigma(y)=0$ we set 
\begin{equation}
\sigma_{t+1}(y)=
    \begin{cases}
    +1 & \, \text{ with probability } \frac{1}{2}e^{-\beta\Delta^p}, \\
    -1 & \, \text{ with probability } \frac{1}{2}e^{-\beta\Delta^m}, \\
    \sigma_{t}(y) & \, \text{ otherwise. }
    \end{cases}
\end{equation}
\end{itemize}
Moreover $\sigma_{t+1}(z)=\sigma_t(z)$ for all $z\in\Lambda\setminus\{y\}$.

Case i) corresponds to the swap of spins between neighboring sites in 
$\Lambda_0$. Cases ii) and iii) correspond to the exchange of spins at 
the boundary $\partial^-\Lambda$. In particular point ii) states that a 
spin $s\in\{-1,+1\}$ inside $\partial^-\Lambda$ is replaced by a spin 
zero and point iii) states that a spin $s\in\{-1,+1\}$ 
outside $\Lambda$ replaces a spin in $\partial^-\Lambda$.

We denote by $p_\beta(\sigma,\eta)$ the transition probability
associated with this Markov chain.
 
\begin{lemma}\label{lemma:reversibility}
The Markov chain defined above is reversible with respect to the Gibbs measure 
\eqref{def:gibbs_measure}, i.e., it satisfies
the detailed balance condition 
\begin{equation}\label{eq:detailed_balance}
\mu_\beta(\sigma)p_\beta(\sigma,\eta)
=\mu_\beta(\eta)p_\beta(\eta,\sigma)
\end{equation}
for $\sigma,\eta\in\mathcal{X}$.
\end{lemma}

\subsection{Paths, energy costs, metastable states}
\label{s:pat}
\par\noindent
A sequence of configurations 
$\underline{\omega})=(\omega_1,\omega_2,\dots,\omega_n)$ such that 
$\omega_i$ and $\omega_{i+1}$ are communicating for any $i=1,2,\dots,n-1$ 
is called a \emph{path of length} $n$.
A path $(\omega_1,\dots,\omega_n)$ is called \emph{downhill}
(resp.~\emph{uphill}) if and only if 
$H(\omega_i)\ge H(\omega_{i+1})$ 
(resp.~$H(\omega_i)\le H(\omega_{i+1})$) 
for any $i=1,2,\dots,n-1$.
In particular, a path $(\omega_1,\dots,\omega_n)$ is called \emph{two-steps downhill}
 if and only if $H(\omega_i)\ge H(\omega_{i+2}) \ge H(\omega_{i+1})$ 
for any $i=1,2,\dots,n-2$.
Given two configurations $\eta,\eta'\in\mathcal{X}$, the set of 
paths with first configuration $\eta$ and last configurations $\eta'$ 
 is denoted by $\Omega(\eta,\eta')$.

Given a path $\underline{\omega}=(\omega_1,\dots,\omega_n)$, its \emph{height} 
$\Phi(\underline{\omega})$ is the maximal energy reached by the 
configurations of the path, more precisely,
\begin{equation}
\label{pat000}
\Phi(\underline{\omega})
:=\max_{i=1,\dots,n}H(\omega_i)
\;.
\end{equation}
Given two configurations $\eta,\eta'$, the \emph{communication height}
between $\eta$ and $\eta'$ is defined as 
\begin{equation}
\label{pat005}
\Phi(\eta,\eta')
:=
\min_{\underline{\omega}\in\Omega(\eta,\eta')}\Phi(\underline{\omega})
\;.
\end{equation}
Any path $\underline{\omega}\in\Omega(\eta,\eta')$ such that 
$\Phi(\underline{\omega})=\Phi(\eta,\eta')$ is called \emph{optimal} 
for $\eta$ and $\eta'$.

Let $\mathcal{I}_{\sigma}$ be the set of configurations with energy strictly lower than $H(\eta)$. The \emph{stability level} of
a configuration $\eta \in \mathcal{X}$ is

\begin{equation}
V_{\eta}:=\Phi(\sigma,\mathcal{I}_{\eta})-H(\eta),
\end{equation}
If $\mathcal{I}_{\sigma}$ is empty, then we define $V_{\sigma}=\infty$. 
Given a real number $a$, we let $\mathcal{X}_a$ be the set of the 
configuration with stability level larger than $a$.

The \emph{metastable states} are defined as the configurations, different from 
the ground states,
such that their stability level is maximal.

\subsection{Main results}
\label{section:main_results}
\par\noindent
In this section, we present the main results of the model for 
the region of the parameter space specified by the 
conditions
\eqref{mod007a}--\eqref{mod007c} 
that are assumed to be satisfied.
We denote by $\mathbb{P}_\sigma$ and $\mathbb{E}_\sigma$ the 
probability measure induced by the Markov chain started 
at $\sigma$ and the associated expectation.

In the next theorems we will 
compute 
estimates for the maximal stability level of all the configurations
of the model and we will 
identify the metastable state.
Some key estimates will be given in terms of the energies of 
two peculiar configurations $\sigma_c$ and $\sigma_s$, 
see Figure \ref{fig:critical_configuration}.
The precise description of these configurations is given in the 
caption of the figure where 
the definition of the critical length
\begin{equation}
\label{ellec}
l_c= \Big\lfloor\frac{2J+\lambda-h}{2h} \Big\rfloor+1
\end{equation}\label{eq:energy_saddle}
will be used. The energies of the two configuration are given by
\begin{align}
    &H(\sigma_c)=H(\muno)+2J(2l_c-1)-(\lambda+h)l_c(l_c-1)
     +(\lambda-h)(l_c(l_c-1)+(2l_c-1)) \notag \\
    &H(\sigma_s)= H(\sigma_c)+6J-2(\lambda+h)+2(\lambda-h). 
\end{align}

\begin{figure}[t]
\begin{center}
    \includegraphics[scale=0.4]{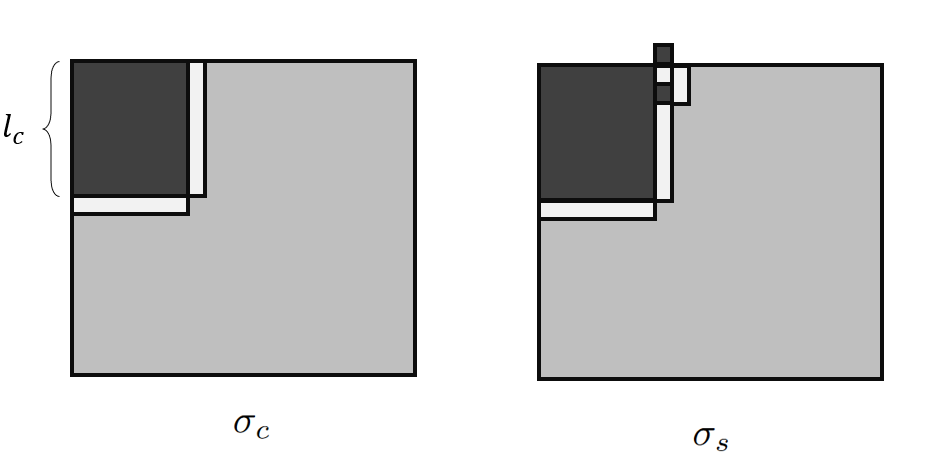}
    \caption{In the picture, white, dark gray and light gray represent the zero, plus and minus regions respectively. On the left, the configuration $\sigma_c$. 
    On the right, the configuration $\sigma_s$ that contains two more pluses and two less minuses than $\sigma_c$. We note that the cluster of pluses can be attached in one of the four corners of $\Lambda$ and the protuberance can be attached along one of the two sides of the cluster of pluses.}
\label{fig:critical_configuration}
\end{center}
    \end{figure}

The first result is an estimate of the stability level of 
all the configurations of 
$\mathcal{X}$ different from $\{\muno,\puno\}$.
This result suggests that 
$\muno$ is the unique metastable state in the region of the parameter
plane under consideration. 

\begin{theorem}
\label{thm:out_gammaBCK}
Let $\eta \in \mathcal{X}$ be a configuration such that 
$\eta \not \in \{\muno,\puno\}$, then 
\begin{equation}
\label{vzero}
V_\eta
\le
6J-2(\lambda+h)+
\frac{4J^2}{\lambda+h} 
=:V^*
.
\end{equation}
\end{theorem}

An immediate consequence of Theorem \ref{thm:out_gammaBCK} is the
recurrence of the system to the set $\{\muno, \puno\}$. 
Indeed, 
by applying \cite[Theorem 3.1]{manzo2004essential} 
we have 
that for any $\epsilon>0$ 
the function
\begin{align}
\label{eq:recurrence_BCK}
\beta\to
\sup_{\sigma \in \mathcal{X}} 
\mathbb{P}_{\sigma}(\tau_{\{\muno,\puno\}}> e^{\beta(V^*+\epsilon)})
\end{align}
is super--exponentially small\footnote{We say that a 
function $x\mapsto f(x)$ is super exponentially small 
for $x\to\infty$ 
if $\lim_{x\to\infty}(\log f(x))/x=-\infty.$}
for $\beta\to\infty$,
where $\tau_{\{\muno,\puno\}}$ is the first hitting time to 
$\{\muno,\puno\}$ 
of the chain started at $\sigma$.

The equation \eqref{eq:recurrence_BCK} implies that the system 
reaches with high probability either the state $\muno$ (which is a 
local minimizer of the Hamiltonian) or the  ground state in a 
time shorter than 
$e^{\beta (V^*+\epsilon)}$, uniformly in the starting configuration 
$\sigma$ for any $\epsilon >0$. In other words we can say that 
the dynamics speeded up by a time factor of order $e^{\beta V^*}$ 
reaches with high probability $\{ \muno, \puno \}$. 

\begin{theorem}[Stability level of $\muno$]
\label{thm:vmuno} 
The stability level $V_\muno$ of $\muno$ is such that 
$$H(\sigma_c)+\Delta^p-H(\muno) \leq V_\muno \leq H(\sigma_s)-H(\muno).$$
\end{theorem}

\noindent
We remark that the upper bound of $V_{\muno}$ can be estimated as follows:
\begin{align} 
& H(\sigma_s)-H(\muno)=2J(2l_c+1)+\lambda (2l_c-1)-h(2l_c^2+3) < \frac{2J^2}{h}+\frac{2J \lambda}{h}+4J+\frac{\lambda ^2}{2h}-\frac{7}{2}h -\lambda.
%& H(\sigma_c)+\Delta^p-H(\muno)= H(\sigma_s) -6J+2h+\Delta^p-H(\muno)=H(\sigma_s)-2J-\lambda+h-H(\muno)
\end{align}
Moreover, by using the explicit expression of the energy of the two configurations $\sigma_s$ and $\sigma_c$, see the equation \eqref{eq:energy_saddle} below, the difference between the upper and the lower bound is
\begin{align}
   [H(\sigma_s)-H(\muno)]-[H(\sigma_c)+\Delta^p-H(\muno)]=2J+\lambda-h.
\end{align}
Finally, we are now able to identify the metastable states of the system. Indeed, since the lower bound in Theorem \ref{thm:vmuno} is strictly greater than $V^*$ and since all the other configurations have stability level smaller than $V^*$, as stated in Theorem \ref{thm:out_gammaBCK}, we have that $\muno$ is the configuration with maximal stability level.

\begin{theorem}[Identification of the metastable state]
\label{thm:Identification} 
The unique metastable state is $\muno$.
\end{theorem}

Finally, we recall that the knowledge of the stability 
level of the metastable state allows to 
give the 
the asymptotic behavior as $\beta \to \infty$ 
of the transition time of the system started at the metastable 
state.

\begin{theorem}[Asymptotic behavior of the transition time]
\label{thm:transition_time_BCG}
For any $\epsilon>0$, we have
\begin{equation}
\label{innominata}
    \lim_{\beta \to \infty} 
\mathbb{P}_{\muno}(e^{\beta(H(\sigma_c)+\Delta^p-H(\muno)-\epsilon)}< 
             \tau_{\puno}<e^{\beta(H(\sigma_s)-H(\muno)+\epsilon)}) =1, 
\end{equation}
where $\tau_\puno$ is the first hitting time to $\puno$ of the chain 
started at $\muno$.
\end{theorem}

\section{Proof of main results}\label{s:pro-th}
\par\noindent
In this section we collect the proofs of all the results stated in 
Section~\ref{s:mrs}.

\subsection{Proof of Lemma~\ref{t:lan00-5}}
\par\noindent
Let $\sigma$ and $\eta$ be two communicating configurations. If $\sigma=\eta$ the statement is trivial. We consider then the case $\sigma\neq \eta$. If there exists $(x \to y) \in E$ such that $\eta=\sigma^{(x,y)}$, then the statement follows by standard Metropolis computations. On the other hand, suppose that $\eta$ and $\sigma$ differ for the value of the spin in a single site $x\in\partial^-\Lambda$. Several cases have to be considered: 
\begin{itemize}
\item[a)] $\sigma(x)=+1$ and $\eta(x)=0$, we have
\begin{equation}\label{eq:mu_sigma_p_sigma_eta}
\mu_\beta(\sigma)p_\beta(\sigma,\eta)
=
\frac{e^{-\beta H(\sigma)}}{Z_\beta}
\frac{1}{2|E^+|}
\end{equation}
since the case ii) in the definition of the dynamics must be 
considered, and 
\begin{equation}\label{eq:mu_eta_p_eta_sigma}
\mu_\beta(\eta) p_\beta(\eta,\sigma)
=
\frac{e^{-\beta H(\eta)}}{Z_\beta}
\frac{1}{2|E^+|}e^{-\beta \Delta^p}
\end{equation}
since the case iii) in the definition of the dynamics must be considered. By \eqref{def:hamiltonian_BCK}, we obtain $H(\sigma)=H(\eta)+\Delta^p$ and we get that \eqref{eq:mu_sigma_p_sigma_eta} and \eqref{eq:mu_eta_p_eta_sigma} are equal.

\item[b)] $\sigma(x)=-1$ and $\eta(x)=0$, the proof is similar to that 
of case a).
\item[c)] $\sigma(x)=0$ and $\eta(x)=+1$, same as case a).
\item[d)] $\sigma(x)=0$ and $\eta(x)=-1$, same as case b).

%\item[e)] $\sigma(x)=+1$ and $\eta(x)=-1$, we have
%\begin{equation} \mu_\beta(\sigma)p_\beta(\sigma,\eta)=\frac{e^{-\beta H(\sigma)}}{Z_\beta}\frac{1}{|E^+|}e^{-\beta \Delta^m}\end{equation}
%since the case $iii)$ in the definition of the dynamics must be considered, and 
%\begin{equation} \mu_\beta(\eta)p_\beta(\eta,\sigma)=\frac{e^{-\beta H(\eta)}}{Z_\beta}\frac{1}{|E^+|}e^{-\beta \Delta^p} \end{equation}
%since the case $iii)$ in the definition of the dynamics must be considered and a minus spin has to enter. The statement is valid in this case since, by \eqref{def:hamiltonian_BCK}, $H(\sigma)+\Delta^m=H(\eta)+\Delta^p$.
%\item[f)] $\sigma(x)=-1$ and $\eta(x)=+1$, same as case e).
%INSERIRE GLI ALTRI DUE CASI? CREDO DI NO
\end{itemize}
We note that if $\sigma(x)=\pm 1$ and $\eta(x)=\mp 1$ with $x \in \partial^-\Lambda$, then $\sigma$ and $\eta$ are not communicating configurations. 
%\qed

\subsection{Proof of Lemma~\ref{t:lan00-5}}
\par\noindent
The fact that $\puno$ is the ground state of the Hamiltonian 
is achieved as in the 
proof of \cite[Lemma~2.2]{cirillo2024homogeneous}
and using that $\Delta^p,\Delta^m>0$.
%\qed

\subsection{Proof of Lemma~\ref{t:lan000}}
\par\noindent
We observe that when a particle moves in $\Lambda$, the energy of the system increases. In particular, the energy cost to get a plus (resp. a minus) in $\Lambda$ is $\Delta^p$ (resp. $\Delta^m$). This implies that the state $\zero$ is a local minumun. Moreover, also the state $\muno$ is a local minimum of the Hamiltonian, since the other possible moves have a positive energy cost, indeed a minus may move from $\Lambda_0$ in $\partial^+ \Lambda_0$ with an energy cost greater than (exit from the boundary) or equal to (exit from the corner) $4J$. 
%\qed

\subsection{Auxiliary lemmas}\label{sec:auxiliary}
\label{s:auxlem}
\par\noindent
In this section we collect some auxiliary lemmas that wil be used in 
the proof of the following theorems.
The proofs of these lemmas are in Section \ref{s:pro-le}. 

First of all, given a configuration $\eta \in \mathcal{X}$,
we consider the set $\mathcal{C}(\eta) \subseteq \mathbb{R}^2$
defined as the union of the closed unitary squares centered at
sites of $\Lambda_0$ with the boundary parallel to the axes of
$\mathbb{Z}^2$ and such that the spin in $\eta$ associated with the site is 
plus. 
The maximal connected components $C_1,\dots, C_m$, 
with $m \in \mathbb{N}$, of $\mathcal{C}(\eta)$ are 
called \emph{clusters of pluses}. 
We define in the same way the \emph{clusters of minuses} 
and the \emph{clusters of zeros}. 
The boundary of each cluster is made of straight lines and corners
on the dual lattice,
that can be \emph{convex corners} or \emph{concave corners} 
following the usual $\mathbb{R}^2$ definitions. 
Moreover, the portion of the boundary of the cluster
delimited by two subsequent convex corners is
called a \emph{convex side} of the cluster,
otherwise, if at least one of the two corners is concave, 
it is called a \emph{concave side}.
We observe that each cluster has at least one convex side,
since $\Lambda_0$ is finite (and, recall, it is not a torus).

Moreover, given a configuration $\sigma\in\mathcal{X}$, 
a \emph{zero-carpet} of $\sigma$ is a connected set 
$X\subset\Lambda$ 
such that $\sigma(x)=0$ for 
each $x\in X$. In a similar way we define the \emph{s-carpet} 
of $\sigma$ with $s \in \{-1,+1\}$.

The following lemma is a key result of this paper. Indeed, 
here we show that 
if a particle, indifferently a plus or a minus, is transported along a 
zero carpet, then it possible to reduce 
the computation 
of the difference of energy 
between any two configurations visitec along the path to 
the examination of nine possible cases. 
Moreover, the height of this path in the configuration space
equals the height  
computed along the corresponding path on the 
graph by summing the weights of each transition.
We use this property to prove a bound for the height of the 
path in the configuration space which does not depend on its length.

\begin{lemma}[Weighted graph for the energy cost in a carpet]
\label{thm:grafo_pesato}
Let $s \in \{-1,0,+1\}$.
Consider a configuration 
$\sigma \in \mathcal{X}$ such that there exists 
a $s$-carpet $X$ of $\sigma$ and two nearest neighboring sites
$x \in \Lambda\setminus X$ and $x'\in X$ such that 
$\sigma(x)=r \neq s$.
Assume, also, that 
there exists $y \neq x'$ a nearest neighbor of $x$ such 
that $\sigma(y)=s$. 
Let $X'$ be the subset of $X$ obtained by collecting $x'$ together 
with all the sites of $X$ having at least two 
neighboring sites in $X$.
Then the following holds:
\begin{enumerate}
\item
for any $v\in X'\cup\{x\}$,
let
$\sigma^v=\sigma^{(x,v)}$ 
and recall that $\sigma^{(x,x)}$ is equal to $\sigma$.
Then, for any pair of neighboring 
sites $v,w\in X'\cup\{x\}$ 
\begin{align}
H(\sigma^v)-H(\sigma^w)
\in
\begin{cases}
\{-8J,-6J,-4J,-2J,0,2J,4J,6J,8J\} \qquad & \text{ if } s=0, \notag \\
\{-8J,-6J,-4J,-2J,0,2J,4J,6J,8J\} \qquad & \text{ if } s \neq 0 \text{ and } r=0, \notag \\
\{-16J,-12J,-8J,-4,0,4J,8J,12J,16J\} \qquad & \text{ if } s \neq 0 \text{ and } r=-s.
\end{cases}
\end{align}
can be computed as specified in the Figure~\ref{figure:graph_carpet}.
\item
For any $v\in X'$ 
\begin{align}
    \Phi(\sigma, \sigma^v)-H(\sigma) \leq 
    \begin{cases}
      8J\qquad & \text{ if } s= 0, \notag \\
      8J\qquad & \text{ if } s \neq 0 \text{ and } r=0, \notag \\
     16J\qquad & \text{ if } s \neq 0 \text{ and } r=-s. 
    \end{cases} 
\end{align}
\end{enumerate}
\end{lemma}

\begin{figure}[t]
\centering
\raisebox{-0.5\height}{\includegraphics[width=0.32\textwidth]{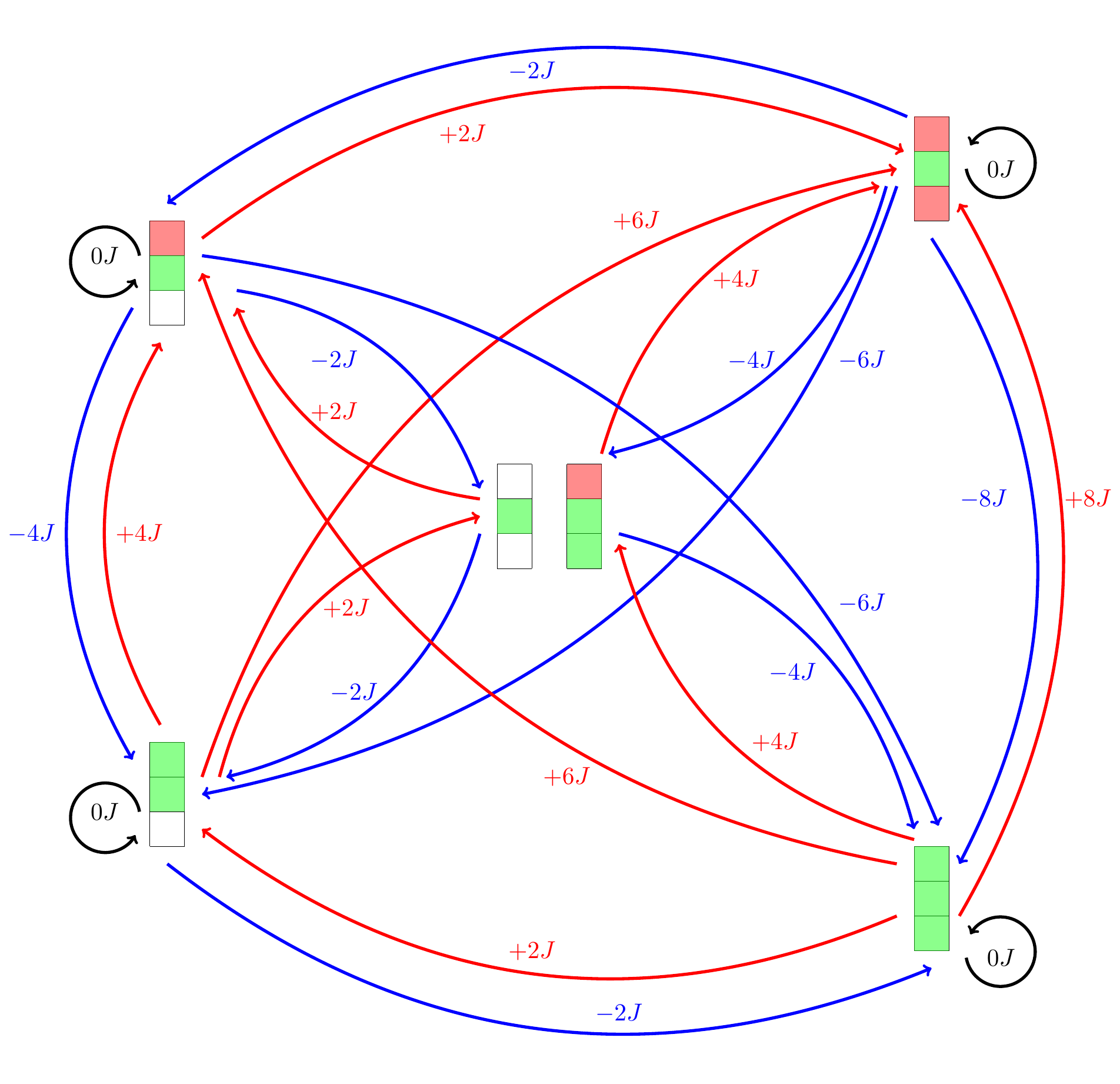}}
    \,\,
    \raisebox{-0.5\height}{\includegraphics[width=0.32\textwidth]{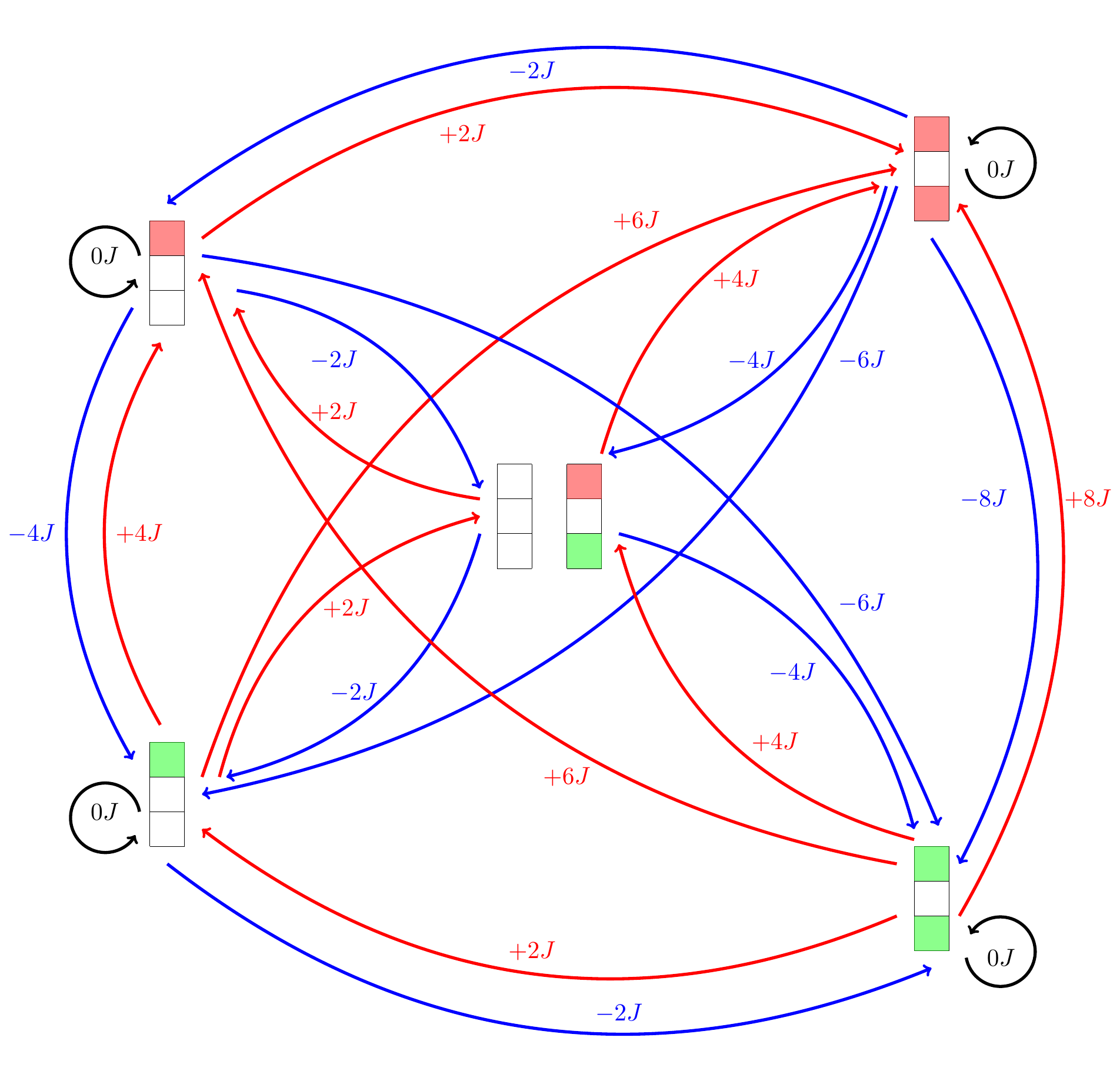}}
    \,\,
\raisebox{-0.5\height}{\includegraphics[width=0.32\textwidth]{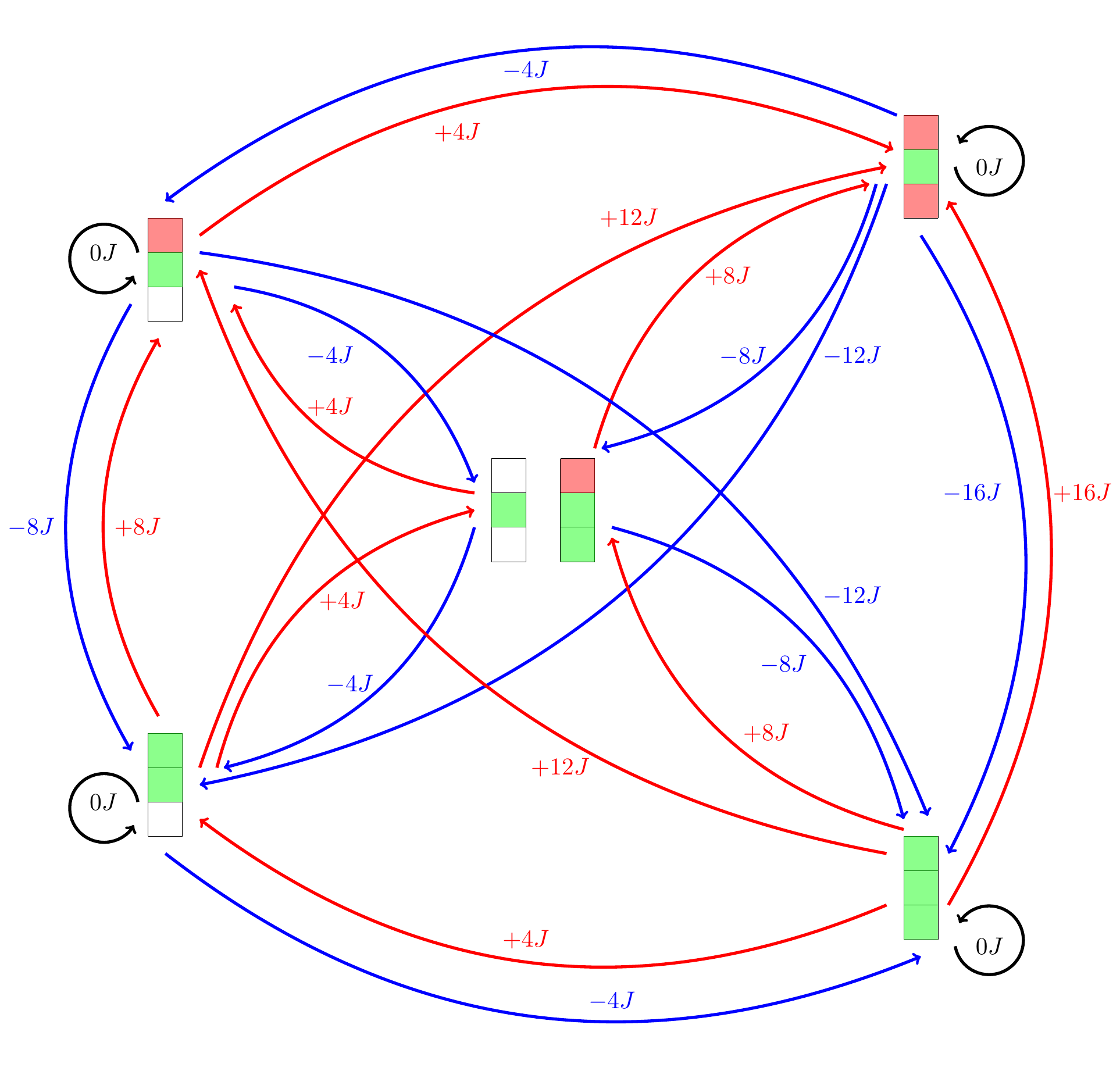}}
\caption{In each vertex of the graph 
the three squares represent three sites such that the one at the 
center is the site $x$ and it is nereast neighbor of the other two.
The local configuration is reported with the colors 
red, green and white representing, respectively, 
the spin values $s \neq 0$, $-s$, and 0. 
The other two 
sites neighboring the middle square and not reported in the picture 
are occupied by a zero spin in the left panel, and by a $s$ spin in the center and right panel. 
The vertex at the center of the graph corresponds to any 
of the two reported configurations.
Any transition 
from $\sigma^v$ to $\sigma^w$ for 
$v,w\in X'\cup\{x\}$ 
is realized via the swap of the particle
between two sites with local configuration represented by 
one of the vertices of the graph.
For each possible swap the energy difference is reported in the picture.
}
\label{figure:graph_carpet}
\end{figure}

In the next lemma, we show that if a particle is transported 
inside or outside $\Lambda_0$ through a carpet, 
than Lemma~\ref{thm:grafo_pesato}
can be used to bound the 
height of the path by $18J$.

\begin{lemma}[Transport through a zero-carpet]\label{lem:transport_zero}
Consider a configuration 
$\sigma \in \mathcal{X}$ such that there exists 
a zero-carpet $X$ of $\sigma$ such that 
$X\cap\partial^-\Lambda_0\neq\emptyset$. 
\begin{itemize}
\item[(i)] Given $x \in X$,
let $\eta:=\sigma^{(x;s)}$ with $s \in \{-1,+1\}$, then
\begin{equation}
\Phi(\sigma, \eta)-H(\sigma) %\leq \Delta_s+6J 
\leq 18J \qquad \text{and} \qquad 
H(\eta)=H(\sigma)+\Delta_s-2D_x(\eta)J,
\end{equation}
where $D_x(\eta)$ is defined in \ref{def:number_interfaces}.
\item[(ii)] Given $x \in \Lambda_0$ a nearest neighbor of $X$ such that $\sigma(x)=s \neq 0$.
Let $\eta:=\sigma^{(x;0)}$, then
\begin{equation}
\Phi(\sigma, \eta)-H(\sigma) %\leq 4J+2 D_y(\sigma)J 
\leq 14J \qquad 
\text{and} 
\qquad H(\eta)=H(\sigma)+2D_x(\sigma)J-\Delta_s.
\end{equation}
\end{itemize}
\end{lemma}

In words, in $(i)$ we state that to transport a plus or a minus through a zero-carpet from outside to a site $x$ with spin zero, has an energy cost smaller than or equal to $18 J$ and the energy of the final configuration depends on the local configuration around $x$. While, in $(ii)$, we state that a plus or a minus in the site $x$ follows a zero-carpet and then it leaves $\Lambda$ with an energy cost smaller than or equal to $10 J$; also in this case the energy of the final configuration depends on the local configuration around $x$. 

\begin{lemma}[Transport through a $s$-carpet]
\label{lemma:transport_scarpet}
Consider a configuration 
$\sigma \in \mathcal{X}$ such that there exists 
a $s$-carpet $X$ of $\sigma$ such that 
$X\cap\partial^-\Lambda_0\neq\emptyset$ and 
$y \in \Lambda$ a nearest neighbor of $X$. %such that $\sigma(y)\neq s$
\begin{itemize}
\item[(i)] If $\sigma(y) \neq 0$, %\sigma(y)=-s$
 then let $\eta:=\sigma^{(y;0)}$ and we have
\begin{align}\label{eq:from_scarpet_to_boundary1}
\Phi(\sigma, \eta)-H(\sigma) \leq 54J 
\qquad 
\text{and} 
\qquad 
H(\eta)=H(\sigma)+2D_y(\sigma)J-\Delta_{s},
%H(\eta)=H(\sigma)+2D_y(\sigma)J-\Delta_{-s},
\end{align}
where $D_y(\sigma)$ is defined in \ref{def:number_interfaces}.

\item[(ii)] If $\sigma(y)=0$, then let $\eta:=\sigma^{(y;s)}$ and we have
\begin{align}\label{eq:from_scarpet_to_boundary2}
\Phi(\sigma, \eta)-H(\sigma) \leq 35J %12J 
\qquad 
\text{and} 
\qquad 
H(\eta)=H(\sigma)+\Delta_s-2D_y(\eta)J.
\end{align}
\end{itemize}
\end{lemma}

In words, in $(i)$ a particle with spin $-s$ follows a $s$-carpet from its position $y$ to $\partial^+ \Lambda_0$ and then it is replaced by a zero. This zero runs through the $s$-carpet until it reaches the site $y$. This path has an energy cost smaller than or equal to $38 J$ and the energy of the final configuration depends only on the local configuration around $y$.
While, $(ii)$ considers the motion of the zero initially in $x$ along the $s$-carpet, until it reaches $\partial^- \Lambda_0$. Hence, this zero is exchanged with a particle with spin $s$ created in $\partial^+\Lambda_0$. The energy cost of this path is smaller than or equal to $54 J$ and the energy of the final configuration depends only on the local configuration around $y$.

We will now state some lemmas in which we use the notion of 
carpet to estimate the stability level of some specified 
configurations. 
Before stating the lemmas we need a new definition.
%An $s$--carpet $X$ of $\sigma$ is called a \emph{wide} $s$--\emph{carpet} if $X\setminus\partial^-X$ is an $s$--carpet.

\begin{lemma}\label{lem:bond_pm}
Let $\sigma$ be a configuration that contains a bond of type $(+,-)$. 
Assume that 
there exists a zero-carpet of $\sigma$ at distance one 
from the bond and intersecting $\partial^-\Lambda_0$ 
or the bond is at distance one from $\partial^+\Lambda_0$, 
then $V_\sigma \leq 14J$.
\end{lemma}

\begin{lemma}\label{lem:bond_pm_different_phase}
Let $\sigma$ be a configuration that contains a bond of type $(+,-)$ and 
let $s \in \{-1,+1\}$. 
If there exists a $s$-carpet of $\sigma$
at distance one from $\partial^+\Lambda_0$ 
and from 
the site of the bond 
with spin $-s$, then $V_\sigma \leq 54J$.
\end{lemma}

%\begin{lemma}\label{lem:distance_2}
%Let $x\in\Lambda_0$ and let $x_i$ be the nearest neighbours 
%of $x$, with $i=1,...,4$. Let $s \in \{-1,+1 \}$ and $\sigma$ be 
%a configuration such that $\sigma(x)=0$, 
%$\sigma(x_1)=\sigma(x_2)=s$, $\sigma(x_3)=0$ and $\sigma(x_4) \neq -s$. 
%If there exists a zero-carpet of $\sigma$ at distance 
%one from $x$ and intersecting $\partial^+\Lambda_0$,
%then $V_\sigma<8J$.
%\end{lemma}

%{\color{blue}
%\begin{lemma}\label{lem:distance_2_scarpet}
%Let $x\in\Lambda_0$ and let $x_i$ be the nearest neighbours 
%of $x$, with $i=1,...,4$. Let $s \in \{-1,+1 \}$ and $\sigma$ be 
%a configuration such that $\sigma(x)=0$, 
%$\sigma(x_1)=\sigma(x_2)=s$, $\sigma(x_3)=0$ and $\sigma(x_4) \neq -s$. 
%If there exists a $s$-carpet of $\sigma$ at distance 
%one from $x$ and intersecting $\partial^+\Lambda_0$,
%then $V_\sigma<...J$.
%\end{lemma}
%}

%{\color{blue}
%\begin{lemma}\label{lem:distance_2_scarpet_part}
%Let $x\in\Lambda_0$ and let $x_i$ be the nearest neighbours 
%of $x$, with $i=1,...,4$. Let $s \in \{-1,+1 \}$ and $\sigma$ be 
%a configuration such that $\sigma(x)=0$ and 
%$\sigma(x_1)=\sigma(x_2)=\sigma(x_3)=s$, $\sigma(x_3)=0$. 
%If there exists a $s$-carpet of $\sigma$ at distance 
%one from $x$ and intersecting $\partial^+\Lambda_0$,
%then $V_\sigma<...J$.
%\end{lemma}
%}

\begin{lemma}\label{lem:p_distance_2_0carpet}
Let $x \in \Lambda$ and let $x_i$ be the nearest neighbours of $x$, 
with $i=1,\dots,4$. Let $\sigma$ be a configuration such 
that $\sigma(x)=0$, $\sigma(x_1)=\sigma(x_2)=+1$, $\sigma(x_3)=0$, 
and $\sigma(x_4) \neq -1$ or $\sigma(x_4) =-1$  with at most two 
nearest neighbors equal to $-1$. 
If there exists a zero-carpet of $\sigma$ at distance 
one from $x$ and intersecting $\partial^+\Lambda_0$,
then $V_\sigma< 18J$. 
\end{lemma}

%{\color{blue}
%\begin{lemma}\label{lem:p_distance_2_scarpet}
%Let $x \in \Lambda$ and let $x_i$ be the nearest neighbours of $x$, 
%with $i=1,\dots,4$. Let $\sigma$ be a configuration such 
%that $\sigma(x)=0$, $\sigma(x_1)=\sigma(x_2)=+1$, $\sigma(x_3)=0$, 
%and $\sigma(x_4) \neq -1$ or $\sigma(x_4) =-1$  with at most two 
%nearest neighbors equal to $-1$. 
%If there exists a plus-carpet of $\sigma$ at distance 
%one from $x$ and intersecting $\partial^+\Lambda_0$,
%then $V_\sigma< ...J$. 
%\end{lemma}
%}

\begin{lemma}\label{lem:shrink_rectangle}
Let $\sigma$ be a configuration that contains a cluster of pluses 
(resp.\ minuses) with at least a convex side with 
length $l \leq \lfloor \frac{2J}{\lambda+h} \rfloor$ (resp.\ $l \leq \lfloor \frac{2J}{\lambda-h} \rfloor$). 
Assume that 
%the cluster of pluses is at distance one from $\partial^+\Lambda_0$ or 
there exists a zero-carpet of $\sigma$ at distance one from one 
of the two corner sites of the convex side of the cluster and intersecting 
$\partial^+\Lambda_0$. 
Then $V_\sigma<16J$.
\end{lemma}

\begin{lemma}\label{lem:grow_rectangle}
Let $\sigma$ be a configuration that contains a cluster of pluses 
(resp.\ minuses) with at least a convex side with 
length $l \geq \lfloor\frac{2J}{\lambda+h}\rfloor +1$ (resp.\ $l \geq \lfloor\frac{2J}{\lambda-h} \rfloor +1$) 
at distance strictly greater than two from a minus spin (resp.\ plus spin). 
Assume that  
%the cluster of pluses (resp. the cluster of minuses) is at distance one from $\partial^+\Lambda_0$ or 
there exists a zero-carpet of $\sigma$ at distance one from one 
of the two corner sites of the convex side of the cluster and intersecting 
$\partial^+\Lambda_0$. 
Then $V_\sigma<22J$.
\end{lemma}

We note that in the previous two lemmas, if the cluster is at distance one from $\partial^+ \Lambda_0$, then there exists a zero carpet, at distance one the two corner sites of the convex side, composed by only sites in $\partial^+ \Lambda_0$.

\begin{lemma}\label{lem:flag}
Let $\sigma$ be a configuration that contains a \emph{flag}-shaped structure namely a structure made of a strip of pluses, a strip of zeros and a strip of minuses 
with equal lengths as in the left or right panel of Figure \ref{fig:flag}.
Assume that there exists a minus-carpet of $\sigma$ at distance one from the strip of zeros, 
then $V_\sigma < 58 J$.
\end{lemma}

\begin{figure}[t]
        \centering
        \includegraphics[scale=0.5]{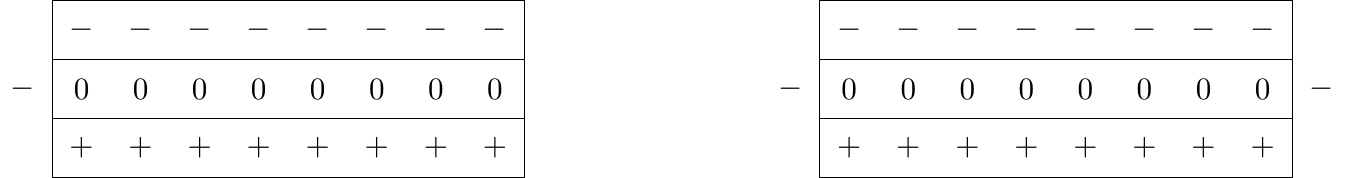}
        \caption{Representation of the flag-shaped structure present in the configuration considered in Lemma \ref{lem:flag}. We note that the strips can be vertical or horizontal.}
        \label{fig:flag}
    \end{figure}

\begin{lemma}\label{lemma:path_0_+}
We have that
$V_\zero \leq 6J-2(\lambda+h)+4J^2/(\lambda+h)$.
\end{lemma}

\subsection{Proof of Theorem~\ref{thm:out_gammaBCK}}
\par\noindent
In order to prove Theorem~\ref{thm:out_gammaBCK} 
we start showing that the stability level of 
each configuration $\eta \not \in \{\muno,\puno\}$ 
is smaller than or equal to $V^*$. 
First of all, we note that if $\eta$ contains some plus or minus spin in $\partial^+ \Lambda_0$, then $V_\eta=0$ indeed the particle with plus (resp. minus) spin leaves $\Lambda$ and the energy decreases by $\Delta^p$ (resp. $\Delta^m)$. 
Thus, from now on we assume $\eta_{\partial^+ \Lambda_0}=0_{\partial^+ \Lambda_0}$.
To prove that $V_\eta < V^*$, we partition the set of all configurations according to 
the value of the spin in the upper-left corner, that we denote by 
$x_0=(1,L)$, and we use 
the auxiliary lemmas \ref{lem:bond_pm}-\ref{lemma:path_0_+}. 
In the following the columns are ordered from left to right and 
the row from top to bottom. 

\paragraph{Case $\eta(x_0)=0$.} Let $R$ be the maximal rectangle 
with the upper-left corner in the $x_0$ and such that 
$\eta_R=0_R$. We note that if $R \equiv \Lambda_0$, i.e. 
$\eta \equiv \zero$, then we conclude by using Lemma \ref{lemma:path_0_+}. 
Thus, assume $R \not \equiv \Lambda_0$ and let $l_v,l_h$ be, respectovely, 
the vertical and horizontal side lengths of $R$. 

If $l_v,l_h \geq \lfloor \frac{2J}{\lambda+h} \rfloor+1$, then we reduce the energy 
of $\eta$ with an energy cost strictly smaller than $V^*$ by using 
the same path described in the proof of Lemma \ref{lemma:path_0_+}. 

Thus, assume that at least one of the two lengths (for instance $l_h$) 
is smaller than or equal to $\lfloor \frac{2J}{\lambda+h} \rfloor+1$. 
Since, $R$ is maximal it follows that 
there exists $x_2\in\{L-l_v+1,\dots,L\}$ such that, 
denoted $x=(l_h+1, x_2)$,
we have
$\eta(x)\neq0$ and 
$\eta(x+je_2)=0$ for $j=1,\dots,L-x_2$.  
We let, also,
$y=x-(m-1)e_2$, with $m\geq 1$ integer, be such 
that $\eta(x-je_2)=\eta(x)$ for every $0\le j \leq m-1$ 
and 
$\eta(y-e_2)\neq \eta(y)$.
We distinguish two cases, see Figure \ref{fig:2cases_P}:

\begin{figure}[t]
        \centering
        \includegraphics[scale=0.35]{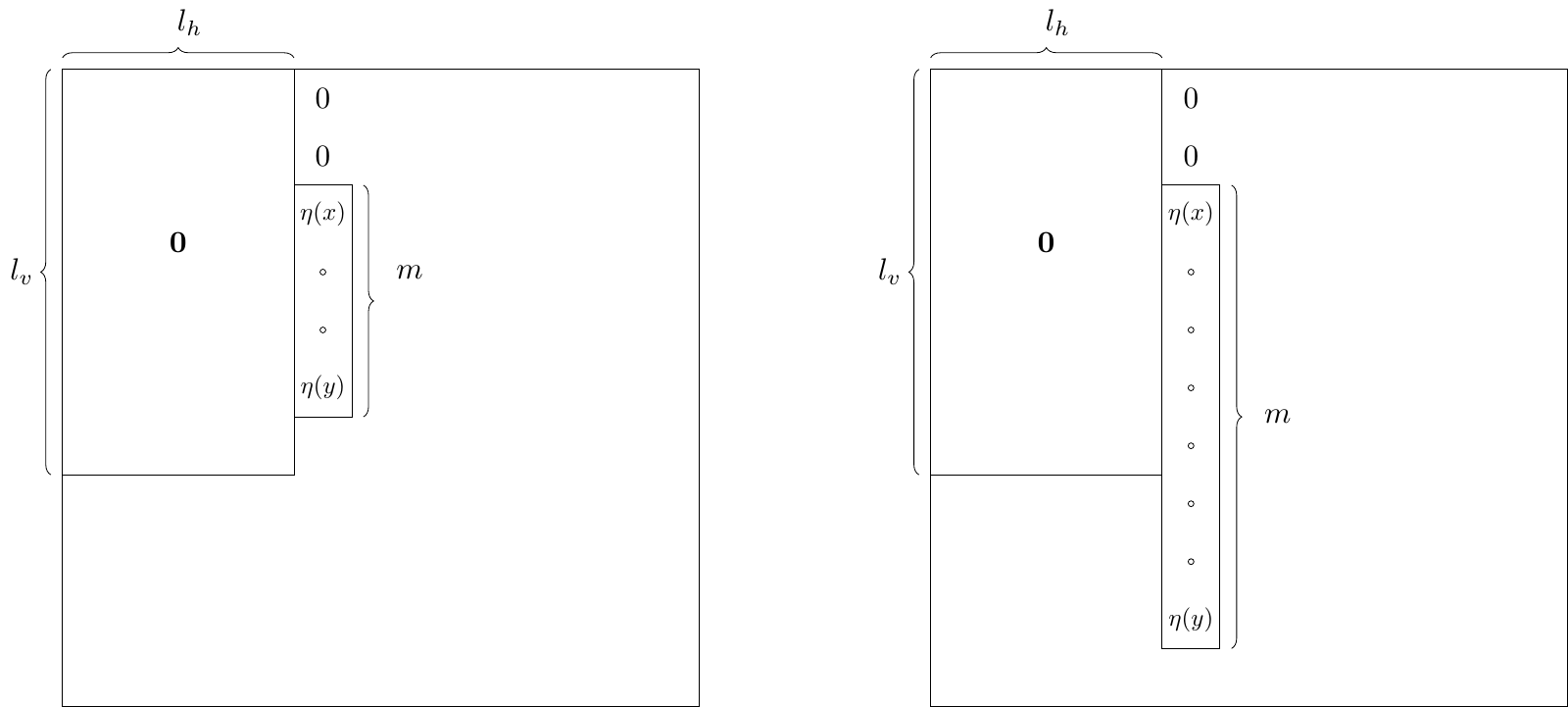}
        \caption{Representation of the configurations used in 
items (i) and (ii) of Case $\eta(x_0)=0$ 
in the proof of the Theorem~\ref{thm:out_gammaBCK}, 
respectively, on the left and on the right.}
        \label{fig:2cases_P}
    \end{figure}
    
\begin{itemize}
\item[(i)] Case $x_2-m \geq L-l_v+1$: 
if $m \leq \lfloor \frac{2J}{\lambda \pm h} \rfloor$, 
we consider the configuration in which the spins in 
$x-je_2$, with $j=0,\dots,m-1$, are changed to $0$, and note that,
by Lemma \ref{lem:shrink_rectangle}, such a configuration
has 
energy smaller that the initial one and the 
communication height is smaller than $10J$ which, in turn,
is smaller than $V^*$.
On the other hand, 
if $m \geq \lfloor \frac{2J}{\lambda \pm h} \rfloor+1$, 
we consider the configuration in which the zeroes at sites 
$(l_h,x_2-j)$ with $j=0,\dots,m-1$ are changed to $\eta(x)$ 
and the theorem follows from 
Lemma~\ref{lem:grow_rectangle}.

\item[(ii)] Case $x_2-m \leq L-l_v$, 
and assume that 
there exists a site $z=(l_h,z_2)$ for 
some $x_2-(m-1)-1\leq z_2 \leq L-l_v$ 
such that $\eta(z) \neq 0$ and 
$\eta(l_h,a)=0$ for $a\geq z_2+1$,
see Figure \ref{fig:origin_zero}.
We distinguish two cases: 
\begin{itemize}
\item[(ii.a)] 
$\eta(z)=-\eta(x)$. We conclude by applying Lemma \ref{lem:bond_pm}.
\item[(ii.b)] 
$\eta(z)=\eta(x)$. We consider the site $w=z+e_2$ and assume 
that $\eta(w-e_1) \neq -\eta(x)$.
If necessary, we add a particle with spin $\eta(x)$ in $\partial^+\Lambda_0$, 
by paying either $\Delta^p$ or $\Delta^m$. Then we 
transport the particle trough the zero-carpet to $w$, by paying $8J$, 
see Lemma~\ref{thm:grafo_pesato}.
By direct inspection we also get that the difference 
of energy between the final and the initial configurations 
is smaller than or equal to 
$-(\lambda \pm h)$. 
On the other hand, in the case $\eta(w-e_1) = -\eta(x)$, we conclude 
by applying Lemma \ref{lem:bond_pm}.

\end{itemize} 

\begin{figure}[!htb]
        \centering
        \includegraphics[scale=0.32]{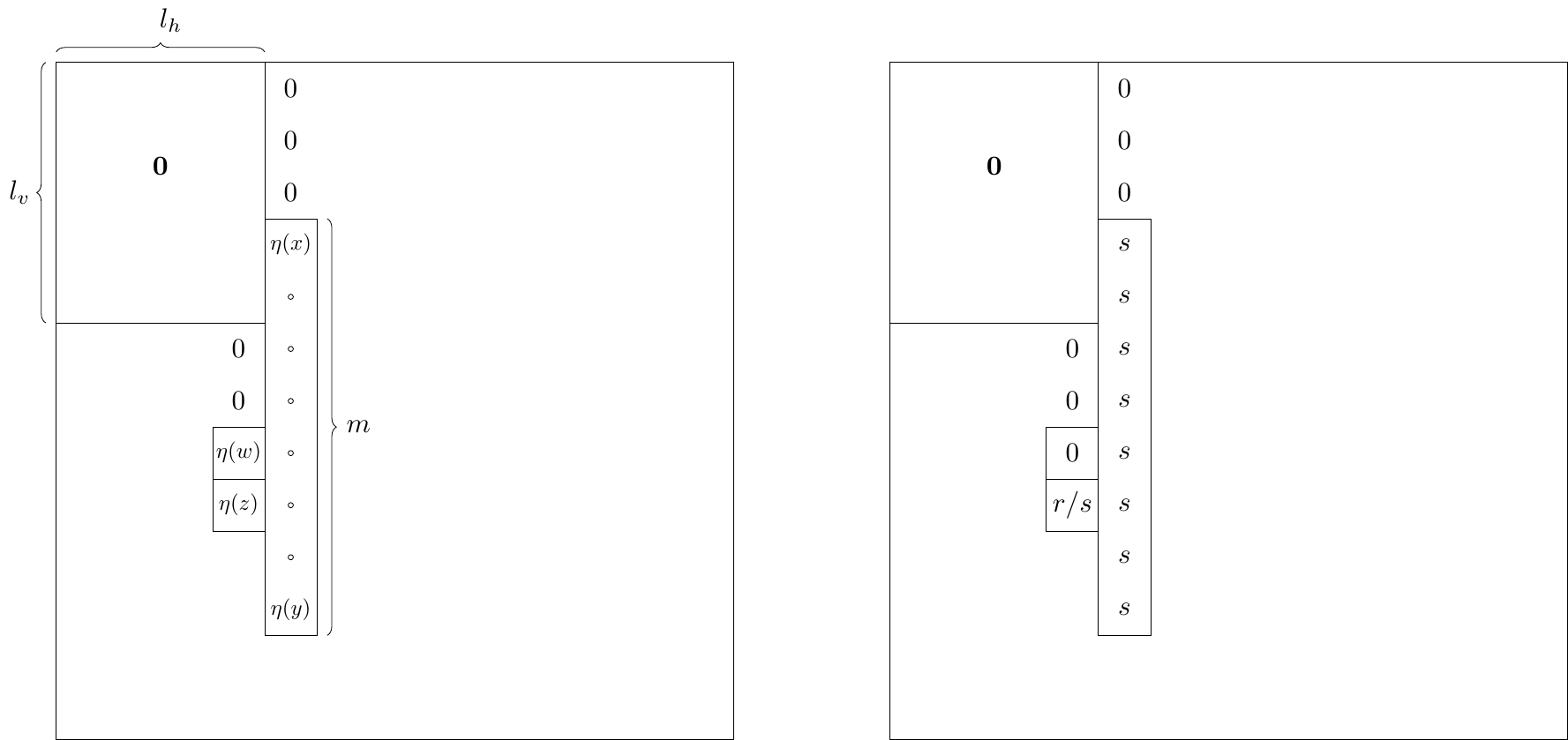}
        \caption{Representation of the configurations in items (ii.a) 
and (ii.b) 
of Case $\eta(x_0)=0$
in the proof of the Theorem~\ref{thm:out_gammaBCK}, respectively,
on the left and on the right.}
        \label{fig:origin_zero}
\end{figure}

\item[(iii)] Case $x_2-m \leq L-l_v$, 
and assume that 
all the spins associated with the sites $(l_h,a)$ with $a\ge x_2-(m-1)-1$ 
are zero. 

In the case 
$m \leq \lfloor \frac{2J}{\lambda \pm h} \rfloor$, we proceed exactly as in case (i) above.

\begin{figure}[!htb]
        \centering
        \includegraphics[scale=0.3]{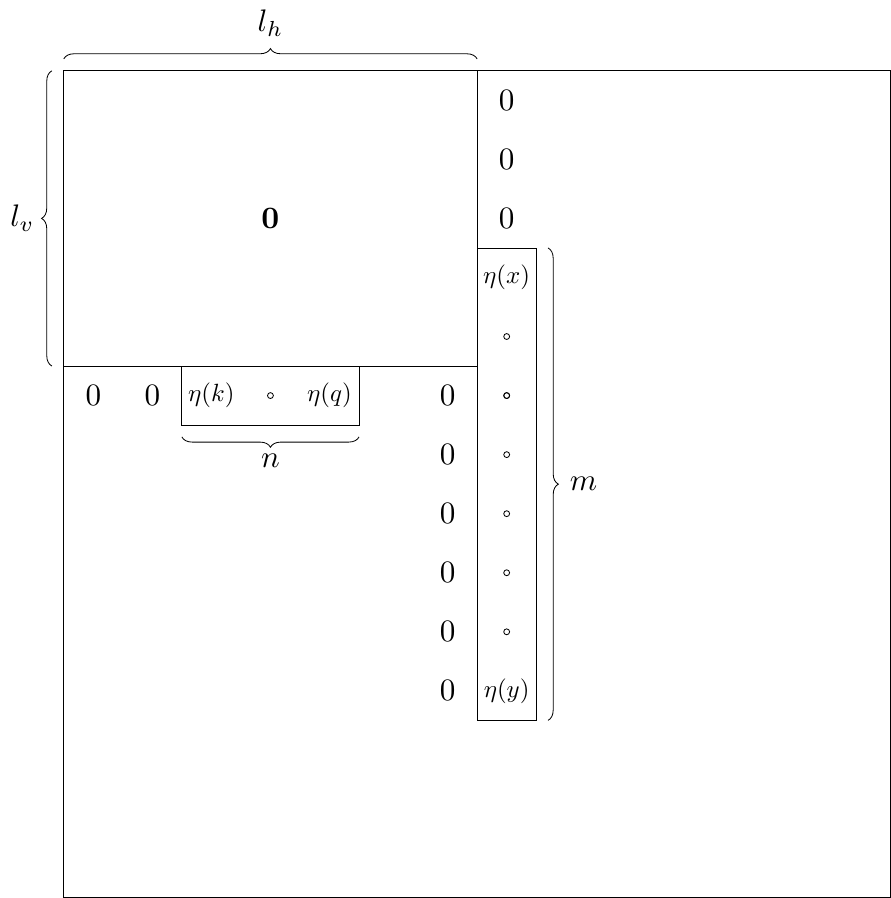} \qquad \includegraphics[scale=0.3]{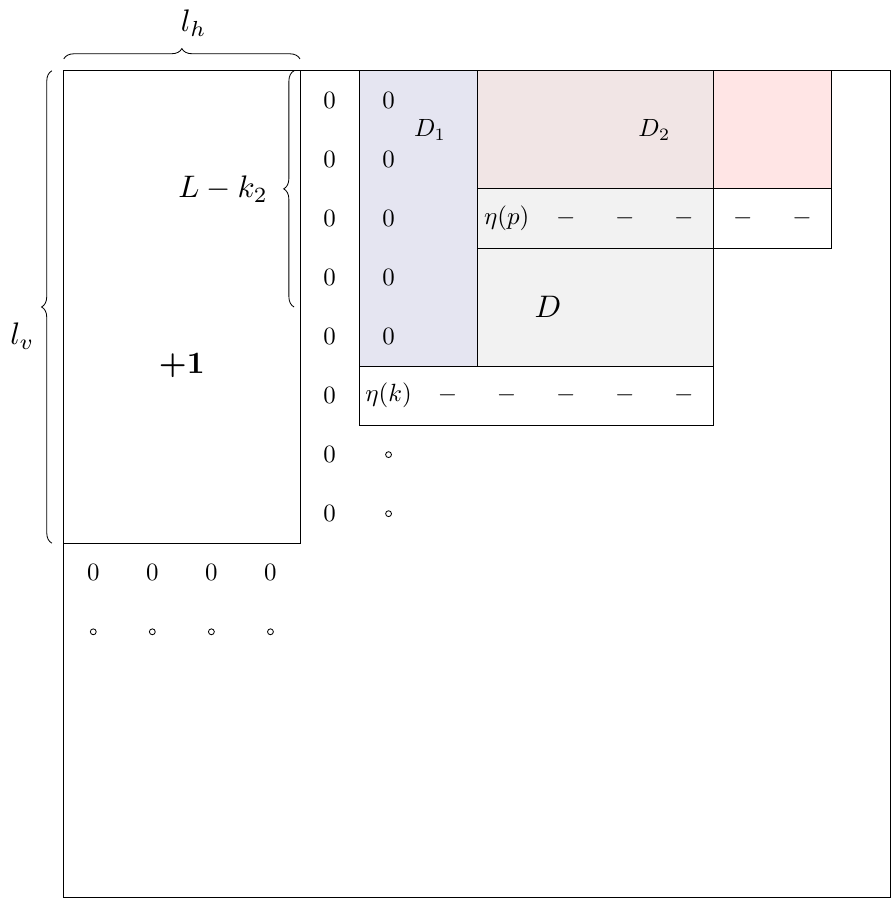}
        \caption{Left: Representation of the configuration used in case (iii)
of Case $\eta(x_0)=0$ in the proof of the Theorem~\ref{thm:out_gammaBCK}.
Right: An example of configuration $\eta$ of point (ii)
in Case $\eta(x_0)=+1$.}
        \label{fig:K_Ktilde}
    \end{figure}

On the other hand, 
in the case $m \geq \lfloor \frac{2J}{\lambda \pm h} \rfloor+1$, we have to look at the 
spins in the column $l_h-1$. 
If all the sites $(l_h-1,j)$ with $j\geq x_2-(m-1)$ have spin zero, 
then we get the proof again as in case (i) above. 
Otherwise, namely, if at least one of the spins associated 
with the sites 
$(l_h-1,j)$ with $j\geq x_2-(m-1)$ is not zero,
then we have to look at 
the sites along the row $L-l_v$. 
Since $R$ is maximal, there exists a site 
$k=(k_1,L-l_v)$ such that $\eta(k)\neq0$ and $\eta(j,L-l_v)=0$ for 
$j=1,\dots,k_1-1$. 
Let $q=k+n e_1$, with $n\geq0$ integer, be such that 
$\eta(k+je_1)=\eta(k)$ for $0\leq j \leq n$ and $\eta(k)\neq\eta(q+e_1)$, 
see the left panel in Figure \ref{fig:K_Ktilde}. 
If $n \leq \lfloor \frac{2J}{\lambda \pm h} \rfloor$ we conclude by 
using Lemma~\ref{lem:shrink_rectangle} and if $n \geq \lfloor \frac{2J}{\lambda \pm h} \rfloor+1$ 
we conclude by using Lemma \ref{lem:grow_rectangle}.

\end{itemize}

\paragraph{Case $\eta(x_0)=+1$.} Let $R$ be the maximal rectangle 
with the upper-left corner in the $x_0$ and such that 
$\eta_R=+1_R$. We note that $R \not \equiv \Lambda_0$, otherwise 
$\eta \equiv \puno$. 

Thus, let $l_v,l_h$ be, respectively, 
the vertical and horizontal side lengths of $R$. 
Since $R$ is maximal, it follows that 
there exists a spin different from plus at distance one from $R$.

In case it is a minus, 
we conclude by 
applying Lemma \ref{lem:bond_pm}, 
if the associated site is in $\partial^- \Lambda_0$, 
otherwise we conclude by applying Lemma \ref{lem:bond_pm_different_phase}.

Now, we consider the case in which all the spins at distance one 
from $R$ are pluses and zeros. 
Without loss of generality, we can assume that there 
exists a site on the vertical part of the external boundary of $R$
with associates spin equal to zero. Thus, we 
let $x_2\in\{L-l_v+1,\dots,L\}$ be such that, 
denoted $x=(l_h+1, x_2)$,
we have $\eta(x)=0$
and $\eta(x+je_2)=+1$ for $j=1,\dots,L-x_2$. 

We note that if $x \not \in \partial^- \Lambda_0$, 
see the left panel in Figure~\ref{fig:case_plus_simply},
then we apply Lemma~\ref{lemma:transport_scarpet} or Lemma~\ref{lem:p_distance_2_0carpet} and we conclude. 

\begin{figure}[t]
        \centering
        \includegraphics[scale=0.4]{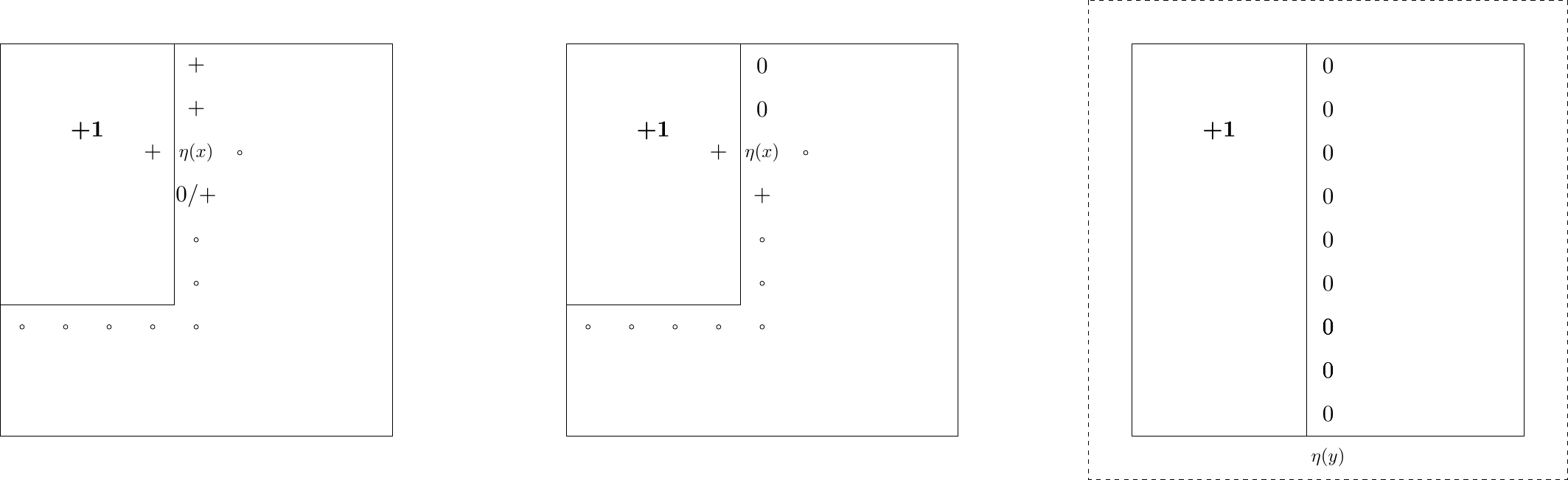}
        \caption{Representation of the configuration used in the 
case $\eta(x_0)=+1$ in the proof of Theorem~\ref{thm:out_gammaBCK}.}
        \label{fig:case_plus_simply}
    \end{figure}

We are thus left with the case 
$x\in \partial^- \Lambda_0$, i.e., $x_2=L$.
In such a case 
we let $y=x-(m-1)e_2$, with $m\geq 1$ integer, be such 
that $\eta(x-je_2)=\eta(x)$ for every $0\le j \leq m-1$ 
and either $y-e_2\in\partial^+\Lambda_0$ or $\eta(y-e_2)\neq \eta(y)$,
see the center and right panel in Figure~\ref{fig:case_plus_simply}.
We distinguish two cases:

\begin{itemize}
\item[(i)] Case $m \leq l_v-1$: 
we can conclude since 
$\eta$ satisfies the assumption 
of Lemma \ref{lem:transport_zero}-(i), indeed, 
$\eta(y)=0$, $\eta(y-e_1)=+1$ (inside $R$), 
$\eta(y-e_2)=+1$ (otherwise there would be a minus in $\partial^+R$), 
and if $\eta(y+e_1)$ were minus, then it would not have 
three neighboring minuses because in such a case there would be 
a direct $(+,-)$ interface in the sites $y-e_2$ and $y+e_1-e_2$.
See the center panel in Figure~\ref{fig:case_plus_simply}.

\item[(ii)] Case $m \geq l_v$. 
If $l_v \leq \lfloor \frac{2J}{\lambda + h} \rfloor$, 
then we reduce the energy 
of $\eta$ with an energy cost strictly smaller than $V^*$ by using 
Lemma~\ref{lem:shrink_rectangle}. 
If $l_v \geq \lfloor \frac{2J+\lambda-h}{2h} \rfloor+1$, 
we conclude by using Lemma~\ref{lem:flag}. 

We are thus left with the case 
$\lfloor \frac{2J+\lambda-h}{2h} 
  \rfloor \geq l_v \geq \lfloor \frac{2J}{\lambda + h} \rfloor+1$ 
with which we deal by looking at the value of the spins 
at distance $\sqrt{2}$ and $2$ from the side of $R$ with length $l_h$.

If one of these spins is plus we apply 
Lemma~\ref{lem:transport_zero}-(i)
and if they are all zeros we conclude by applying 
Lemma~\ref{lem:grow_rectangle}.
We are left with the case in which 
these spins are either zero or minus and 
at least one of them is minus, 
see Figure~\ref{fig:origin_plus_norectangle}. 
We distinguish two cases:

\begin{figure}[t]
        \centering
        \includegraphics[scale=0.4]{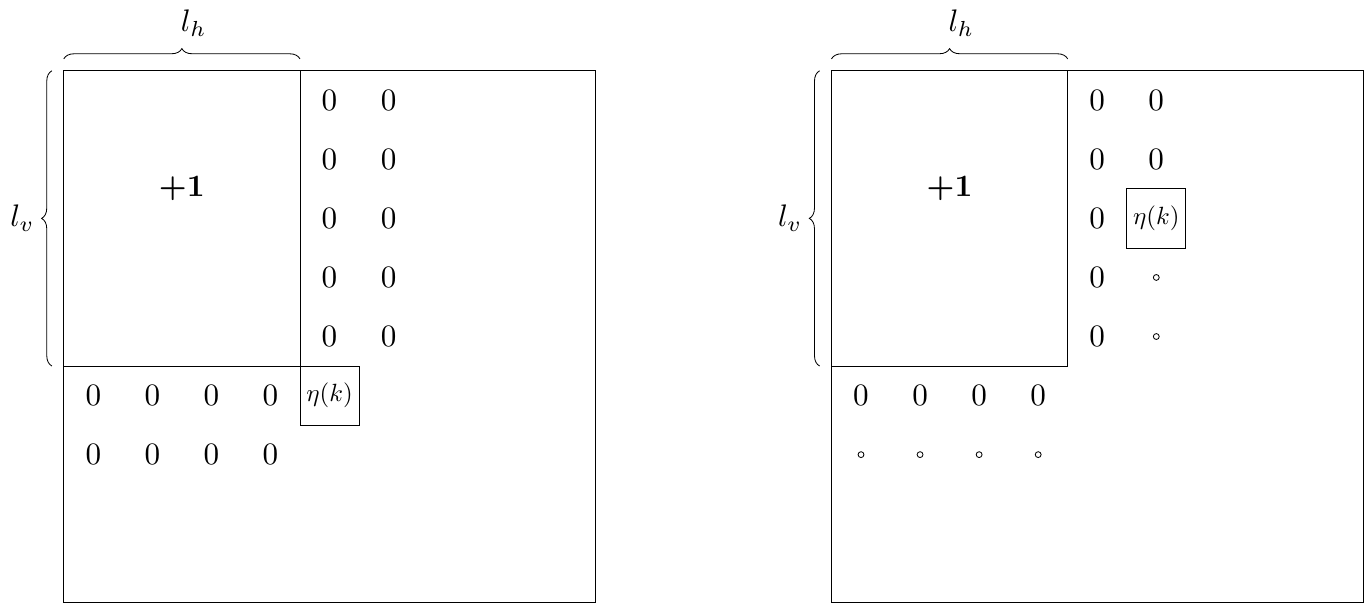}
        \caption{Representation of the configuration used in the 
case $\eta(x_0)=+1$ in the proof of Theorem~\ref{thm:out_gammaBCK}.}
        \label{fig:origin_plus_norectangle}
    \end{figure}

\begin{itemize}
\item[(ii.a)] all these spins are zeros excepted the 
one at distance $\sqrt{2}$ which has spin minus,
see the left panel in Figure~\ref{fig:origin_plus_norectangle}. 

We note that if this minus has three minus nearest neighbours, 
then one of them is at distance one from $R$ and so we conclude by 
applying Lemma \ref{lem:bond_pm_different_phase}. 

Thus, we assume that the minus has at most two nearest neighbors 
with minus spin.
Let $\eta'$ be the configuration obtained from $\eta$ by replacing the 
minus at distance $\sqrt{2}$ with a zero, 
and let $\eta''$ be the configuration obtained from $\eta'$ 
by replacing with plus 
the zeros at distance one from the side of the rectangle 
with length $l_v$, 
see Figure~\ref{fig:origin_plus_caseA}. 
We consider the following path which starts from $\eta$ and ends in $\eta''$ 
crossing $\eta'$: the minus is transported along the (vertical in 
the figure) column adjacent to $R$ and, after it reaches 
the external boundary $\partial^+\Lambda_0$, it is removed. 
Then, one after the other, $l_v$ pluses are created in 
$\partial^+\Lambda_0$ and transported along the same column to 
their final location as reported in 
panel of Figure~\ref{fig:origin_plus_caseA}.
We have that $H(\eta')=H(\eta)+(\lambda-h)$ and $H(\eta'')
=H(\eta')+2J-l_v(\lambda+h)$. Then, we obtain
$H(\eta'') \leq H(\eta)+(\lambda-h)+2J
-\Big( \frac{2J}{\lambda+h}+1\Big ) (\lambda+h)<H(\eta)$.
Moreover, by direct inspection, we have that the height along the 
considered path is smaller than or equal to $6J$.

\begin{figure}[t]
\centering
\includegraphics[scale=0.4]{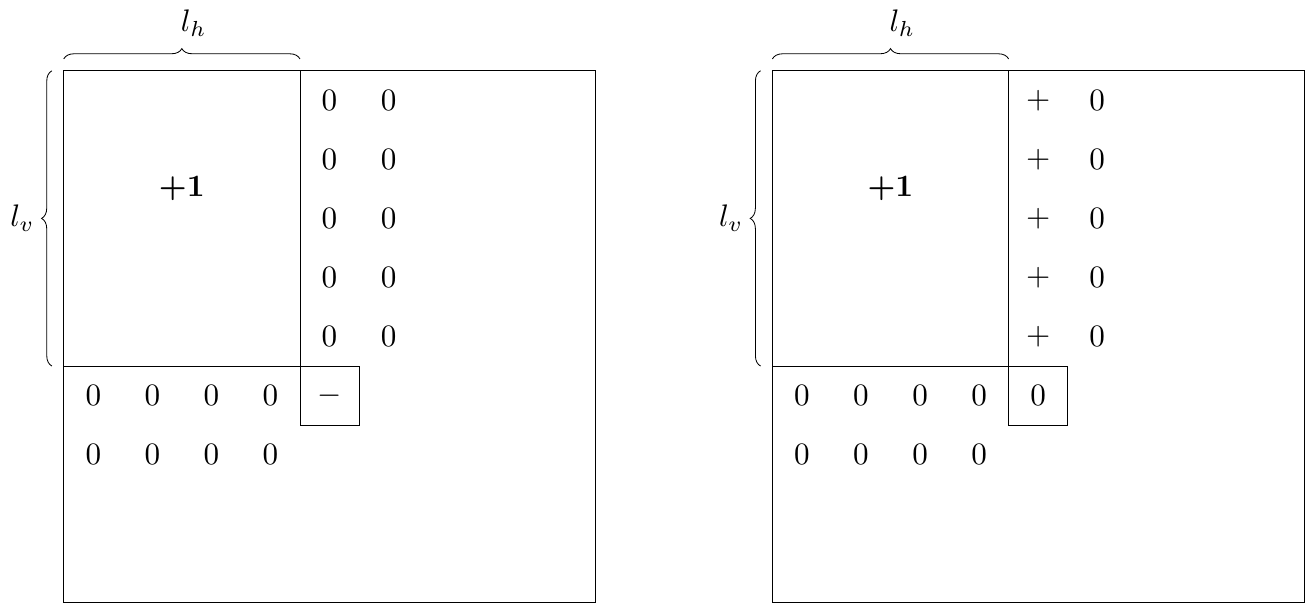}
\caption{On the left the initial configuration $\eta$ of point (i)
in Case $\eta(x_0)=+1$. On the right the final configuration $\eta''$.}
\label{fig:origin_plus_caseA}
\end{figure}

\item[(ii.b)] There is at least a minus spin 
at distance $2$ from the side of $R$ with length $l_v$.

We denote by $k=(l_h+2,k_2)$ the site with minus spin and at minimal distance 
from $\partial^+\Lambda_0$, 
see the right panel in Figure \ref{fig:origin_plus_norectangle}. 

If $L-k_2 \geq \lfloor \frac{2J}{\lambda+h} \rfloor+1$, then we 
replace with pluses the first (from the top) $L-k_2$ zeroes 
in the vertical column adjacent to $R$, 
by applying Lemma \ref{lem:grow_rectangle} we prove the theorem.

On the contrary, suppose, now, 
that $L-k_2 \leq \lfloor \frac{2J}{\lambda+h} \rfloor$. 

We consider the rectangular region $D$ of $\Lambda_0$ with the upper-left corner in $(l_h+2, L)$ and the down-right corner in $k+ne_1$ where $n \in \{l_h+2,...,L\}$ is such that $\eta(k+je_1)=\eta(k)$ for every $0\leq j \leq n$ 
and $\eta(k+ne_1)\neq \eta(k)$, see the right panel in Figure \ref{fig:K_Ktilde}.
Notice that the horizontal strip from $k$ to $k+ne_1$ contained in $D$
is a strip of minuses by construction. 
Let $p \in D$ be the site with the minus spin according to the lexicographic order. We note that $p$ can be $k$ by construction. 
We consider the rectangular region $D_1 \subset D$ with the upper-left corner in $(l_h+1, L)$ and the down-left corner in 
$(p_1-1,k_1+1)$. We observe that by construction in $D_1$ there are not minus spins, see the right panel in Figure \ref{fig:K_Ktilde}.
If there is a plus spin at distance one from the strip of minuses containing $k$ 
in the rectangular region $D_1$, then we conclude by applying Lemma \ref{lem:bond_pm}.
Thus, we are left with the case in which there are only zeros at distance one from this strip.
Assume, first, that in $D_1$ there is a cluster of pluses,
then it necessarily has a convex side length smaller 
than $\lfloor \frac{2J}{\lambda+h} \rfloor$. Then, we conclude by applying 
Lemma~\ref{lem:grow_rectangle}. 
Thus, in $D_1$ there are only zero spins. 
and we consider the rectangular region $D_2$ 
with the upper-left corner in $(p_1,L)$ and the down left corner in $p+ae_1$
where $a \in \{l_h+2,...,L\}$ is such that $\eta(p+ae_1)=\eta(k)$ for every $0\leq j \leq a$ 
and $\eta(p+ae_1)\neq \eta(p)$, see the right panel in Figure \ref{fig:K_Ktilde}. 

Assume $D_2 \neq \emptyset$.
Notice that the horizontal strip from $p$ to $p+ae_1$ is a strip of minuses by construction. 
If there is a plus spin at distance one from this strip, then we conclude by applying Lemma \ref{lem:bond_pm}.
We observe that by construction in $D_2$ there are not minus spins and we
assume, first, that in $D_2$ there is a cluster of pluses.
This cluster necessarily has a convex side length smaller 
than $\lfloor \frac{2J}{\lambda+h} \rfloor$ by construction,
thus we conclude by applying Lemma~\ref{lem:grow_rectangle}. 
Thus, in $D_2$ there are only zeros and we consider the length $a$ of the strip of minuses 
containing $p$.

If $a\leq \lfloor \frac{2J}{\lambda-h} \rfloor$
we conclude by applying Lemma \ref{lem:shrink_rectangle}. 

Otherwise if $a \geq \lfloor \frac{2J}{\lambda-h} \rfloor +1$, we distinguish two cases:
\begin{itemize}
    \item[(ii.b.1)] The site at distance $\sqrt{2}$ from the strip of minus has spin different from plus. In this case, we apply Lemma \ref{lem:grow_rectangle} and we prove the Theorem.
    \item[(ii.b.2)] The site at distance $\sqrt{2}$ from the strip of minus has plus spin. We denote by $z$ this site and we note that
    if $\eta(z-e_2)=+1$, then there is a direct $(+,-)$ interface in the sites $z-e_2$ and $z-e_1-e_2$. Thus, we conclude by applying Lemma \ref{lem:transport_zero}-(ii) to the site $z-e_1-e_2$.
    Otherwise if $\eta(z-e_2)=0$, then we can conclude by applying Lemma \ref{lem:shrink_rectangle} to the cluster of pluses containing $z$. 
\end{itemize}

If $D_2= \emptyset$, i.e. $p \equiv k$, then we proceed as above by applying the same procedure to $D_1$.

\end{itemize}
\end{itemize}

\paragraph{Case $\eta(x_0)=-1$.} Let $R$ be the maximal rectangle 
with the upper-left corner in the $x_0$ and such that 
$\eta_R=-1_R$. We note that $R \not \equiv \Lambda_0$, otherwise 
$\eta \equiv \muno$. 

Thus, let $l_v,l_h$ be, respectively, 
the vertical and horizontal side lengths of $R$. 
Since $R$ is maximal, it follows that 
there exists a spin different from minus at distance one from $R$.

In case it is a plus, 
we conclude by 
applying Lemma \ref{lem:bond_pm}
if the associated site is in $\partial^- \Lambda_0$, 
otherwise we conclude by applying Lemma \ref{lem:bond_pm_different_phase}.

Then, we consider the case in which all the spins at distance one 
from $R$ are minuses and zeros. 
Without loss of generality, we can assume that there 
exists a site on the vertical part of the external boundary of $R$
with an associated spin equal to zero. Thus, we 
let $x_2\in\{L-l_v+1,\dots,L\}$ be such that, 
denoted $x=(l_h+1, x_2)$,
we have $\eta(x)=0$
and $\eta(x+je_2)=-1$ for $j=1,\dots,L-x_2$. 

We distinguish two cases:
\begin{itemize}
\item[(i)] Case $x_2 \not = L$. 
We consider the cluster of minuses $C$ containing $R$. 
If $\partial^+ C$ contains a plus spin, then we conclude by using Lemma \ref{lem:bond_pm_different_phase}.
Thus, suppose that $\partial^+ C$ contains only zero spins.
Either $\eta \equiv \muno$ or we denote by $p=(p_1,p_2)$ with $p_1, p_2 \in \{1,...,L \}$ the first site such that $\eta(p)=0$
in lexicographic order.
If $p_2=L$, then we proceed as in case (ii) below assuming $p \equiv x$. 
Otherwise if $p_2=L-j$ with $L-1 \geq j>0$, then all spins with vertical coordinate $L, L-1, ..., L-(j-1)$ are minus and we look at sites $p+e_1$ and $p-e_2$.
If $\eta(p+e_1)=-1$ or $\eta(p-e_2)=-1$, then we conclude 
by applying Lemma~\ref{lemma:transport_scarpet}. 
If $\eta(p+e_1)=\eta(p-e_2)=0$, then we prove the statement 
by using Lemma \ref{lemma:transport_scarpet}. 
We are thus left with the case $\eta(p+e_1)=0$ and $\eta(p-e_2)=+1$, indeed we recall that $p+e_1 \in \partial^+ \Lambda_0$ or $p+e_1 \in \partial^+ C$, so $\eta(p+e_1) \neq +1$.
We consider the strip of pluses containing $p-e_2$ and we observe that if there exists a site at distance one from this strip with vertical coordinate $p_2+1$ with spin different from zero, then we can apply Lemma \ref{lem:bond_pm_different_phase} and we conclude. 
Thus, all the sites at distance one from this strip with vertical coordinate $p_2+1$ are zero and we prove the statement by using Lemma \ref{lem:flag}. 

\item[(ii)] Case $x_2 = L$.
In such a case 
we let $y=x-(m-1)e_2$, with $m\geq 1$ integer, be such 
that $\eta(x-je_2)=\eta(x)$ for every $0\le j \leq m-1$ 
and either $y-e_2\in\partial^+\Lambda_0$ or $\eta(y-e_2)\neq \eta(y)$.
We distinguish two cases:

\begin{itemize}
\item[(ii.a)] Case $m \leq l_v-1$. 
Assume first $\eta(y+e_1) \neq +1$, then
we can conclude since 
$\eta$ satisfies the assumption 
of Lemma \ref{lem:transport_zero}-(i), indeed, 
$\eta(y)=0$, $\eta(y-e_1)=-1$ (inside $R$), 
$\eta(y-e_2)=-1$ (otherwise there would be a plus in $\partial^+R$).

Now, suppose that $\eta(y+e_1) = +1$ and let $w=y+e_1+ne_2$ with $n$ be such 
that $\eta(y+e_1+je_2)=\eta(y+e_1)$ for every $0\le j \leq m-1$ 
and either $w+e_2\in\partial^+\Lambda_0$ or $\eta(w+e_2)\neq \eta(y+e_1)$.

If $L-(y_2+ \tilde j)-1\geq \lfloor \frac{2J}{\lambda-h} \rfloor+1$, then we apply
Lemma \ref{lem:grow_rectangle} to the part of the side of $R$ with length $L-(y_2+ \tilde j)-1$ 
and we prove the theorem.

On the contrary, if $L-(y_2+ \tilde j)-1 \leq \lfloor \frac{2J}{\lambda-h} \rfloor$, 
we consider the rectangle of $\Lambda_0$ with the upper-left corner 
in $(l_h+2, L)$ with side lengths $L-(y_2+ \tilde j)$ and $l_s$, 
where $l_s$ is the horizontal length of the strip of pluses
containing $\eta(w)$.
If inside of this rectangle there is a minus spin at distance one from this strip of pluses, 
then we conclude by applying Lemma \ref{lem:bond_pm}.
Thus, assume that inside the rectangle 
there are only zeros and pluses at distance one from this strip and 
assume, first, that
in the rectangular region there is a cluster of minuses,
then it necessarily has a convex side length smaller 
than $L-(y_2+ \tilde j)-1 \leq \lfloor \frac{2J}{\lambda-h} \rfloor$. Then, we conclude by applying 
Lemma~\ref{lem:grow_rectangle}. 

Now, assume that 
in this rectangular region there is a mixture of pluses and zeros. 
If the cluster of pluses containing $\eta(w)$ has a concave side or 
there are some plus spins at distance two from it, then we 
conclude by applying Lemma~\ref{lem:transport_zero}-(i). 
In the other case, we consider the length $n$. 

If $l_s \leq \lfloor \frac{2J}{\lambda+h} \rfloor$
we conclude by applying Lemma \ref{lem:shrink_rectangle}. 

Otherwise if $l_s \geq \lfloor \frac{2J}{\lambda+h} \rfloor +1$, we distinguish two cases:
\begin{itemize}
    \item[(ii.a.1)] The site $z$ at distance $\sqrt{2}$ from the strip of pluses and at distance one from the rectangle $S$ has spin different from minus. In this case, we apply Lemma \ref{lem:grow_rectangle} and we prove the Theorem. 
    \item[(ii.a.2)] The site $z$ at distance $\sqrt{2}$ from the strip of pluses and at distance one from the rectangle $S$ has minus spin. We conclude as in case $\eta(x_0)=+1$  (ii.a). 
\end{itemize}

\item[(ii.b)] Case $m \geq l_v$. 
If $l_v \leq \lfloor \frac{2J}{\lambda - h} \rfloor$, 
then we reduce the energy 
of $\eta$ with an energy cost strictly smaller than $V^*$ by using 
Lemma~\ref{lem:shrink_rectangle}. 
We are thus left with the case 
$l_v \geq \lfloor \frac{2J}{\lambda - h} \rfloor+1$ 
with which we deal by looking at the value of the spins 
at distance $\sqrt{2}$ and $2$ from the vertical side of $R$.

Assume first that the spin at distance $\sqrt{2}$ is different from plus.
If all the sites at distance 2 have zero spin, then we apply 
Lemma~\ref{lem:grow_rectangle}. 
If one of the spins at distance 2 is minus we apply 
Lemma~\ref{lem:transport_zero}-(i).

Now, we suppose that the spin at distance $\sqrt{2}$ is plus.
If there is a minus spin in a site $d=(l_h+2, d_2)$ such that $d_2 \neq L-l_v-1$, then we apply Lemma \ref{lem:transport_zero}-(i). 
Otherwise, if $d_2=L-l_v-1$ then we analyze the spins at the right and the left of the site $d-e_1-e_2$, i.e. the site of the plus spin at distance $\sqrt{2}$ from the rectangle.
If one of them is plus, then we apply Lemma \ref{lem:transport_zero} at the minus at site $d$ or $d-2e_1$. 
Otherwise, we apply the Lemma \ref{lem:transport_zero} at the plus in  $d-e_1-e_2$,. 

We are left with the case in which 
spins are either zero or plus at distance 2 and $\sqrt{2}$ and 
at least one of them is plus.
We distinguish two cases:

\begin{itemize}
\item[(ii.b.1)] all these spins are zeros excepted the 
one at distance $\sqrt{2}$ which has spin plus. We denote by $q=(l_h+1, L-l_v)$ the corresponding site.
If $l_v-1 \geq \lfloor \frac{2J}{\lambda-h} \rfloor+1$, then we 
apply Lemma \ref{lem:grow_rectangle} and we prove the theorem.

On the contrary, suppose  
that $l_v-1 \leq \lfloor \frac{2J}{\lambda-h} \rfloor$. 

We consider the rectangular region $D$ of $\Lambda_0$ with the upper-left corner in $(l_h+1, L)$ and the down-right corner in $q+ne_1$ where $n \in \{l_h+2,...,L\}$ is such that $\eta(q+je_1)=\eta(q)$ for every $0\leq j \leq n$ 
and $\eta(q+ne_1)\neq \eta(q)$.
Notice that the horizontal strip from $q$ to $n$ contained in $D$
is a strip of pluses by construction. 
Let $p \in D$ be the site with the first plus spin according to the lexicographic order. We note that $p$ can be $q$ by construction. We consider the rectangular region $D_1 \subset D$ with the upper-left corner in $(l_h+1, L)$ and the down-left corner in 
$(p_1-1, q_1+1)$. We observe that by construction in $D_1$ there are not plus spins.
If there is a minus spin at distance one from the strip of pluses containing $q$ 
in the rectangular region $D_1$, then we conclude by applying Lemma \ref{lem:bond_pm}.
Thus, we are left with the case in which there are only zeros at distance one from this strip.
Assume, first, that in $D_1$ there is a cluster of minuses,
then it necessarily has a convex side length smaller 
than $l_v-1 \leq \lfloor \frac{2J}{\lambda-h} \rfloor$. Then, we conclude by applying 
Lemma~\ref{lem:grow_rectangle}. 
Thus, in $D_1$ there are only zero spins 
and we consider the rectangular region $D_2$ 
with the upper-left corner in $(p_1,L)$ and the down left corner in $p+ae_1$
where $a \in \{p_1,...,L\}$ is such that $\eta(p+ae_1)=\eta(p)$ for every $0\leq j \leq a$ 
and $\eta(p+ae_1)\neq \eta(p)$.
Assume $D_2 \neq \emptyset$.
Notice that the horizontal strip from $p$ to $p+ae_1$ is a strip of pluses by construction. 
If there is a minus spin at distance one from this strip, 
then we conclude by applying Lemma \ref{lem:bond_pm}.

We observe that by construction in $D_2$ there are not plus spins and we
assume, first, that in $D_2$ there is a cluster of minuses.
This cluster necessarily has a convex side length smaller 
than $l_v-1 \leq \lfloor \frac{2J}{\lambda-h} \rfloor$ by construction,
thus we conclude by applying Lemma~\ref{lem:grow_rectangle}. 
 
Thus, in $D_2$ there are only zeros 
and we consider the length $a$ of the strip of pluses 
containing $p$.
If $a \leq \lfloor \frac{2J}{\lambda+h} \rfloor $, then we conclude by applying Lemma~\ref{lem:grow_rectangle}.
If $a \geq \lfloor \frac{2J}{\lambda+h} \rfloor +1$, we distinguish two cases:
\begin{itemize}
    \item[(ii.b.1.I)] The site at distance $\sqrt{2}$ from the strip of pluses has spin different from minus. In this case, we apply Lemma \ref{lem:grow_rectangle} and we prove the Theorem.
    \item[(ii.b.1.II)] The site at distance $\sqrt{2}$ from the strip of pluses has minus spin. We conclude by building two configurations $\eta'$ and $\eta''$ and by proceeding as in case $\eta(x_0)=+1$  (ii.a). 
\end{itemize}

If $D_2= \emptyset$, i.e. $p \equiv q$, then we proceed as above by applying the same procedure to $D_1$.

\item[(ii.b.2)] 
We denote by $k=(l_h+2,k_2)$ the site with plus spin and at minimal distance 
from $\partial^+\Lambda_0$ and we proceed in the same manner of case $\eta(x_0)=+1$ (ii.b). 
The only difference from this proof is that in case (ii.b.2) we have to apply Lemma \ref{lem:shrink_rectangle} to the cluster of pluses containing the strip of length $a$ instead of the strip containing $z$.
\end{itemize}
\end{itemize}
\end{itemize}

\subsection{Proof of Theorem~\ref{thm:vmuno}}
\label{sec:vmuno}
\par\noindent
We provide separately the upper and the lower bound to $V_\muno$, 
respectively, in Sections~\ref{sec:upper} and \ref{sec:lower}.

\begin{figure}[t]
\begin{center}
    \includegraphics[scale=0.5]{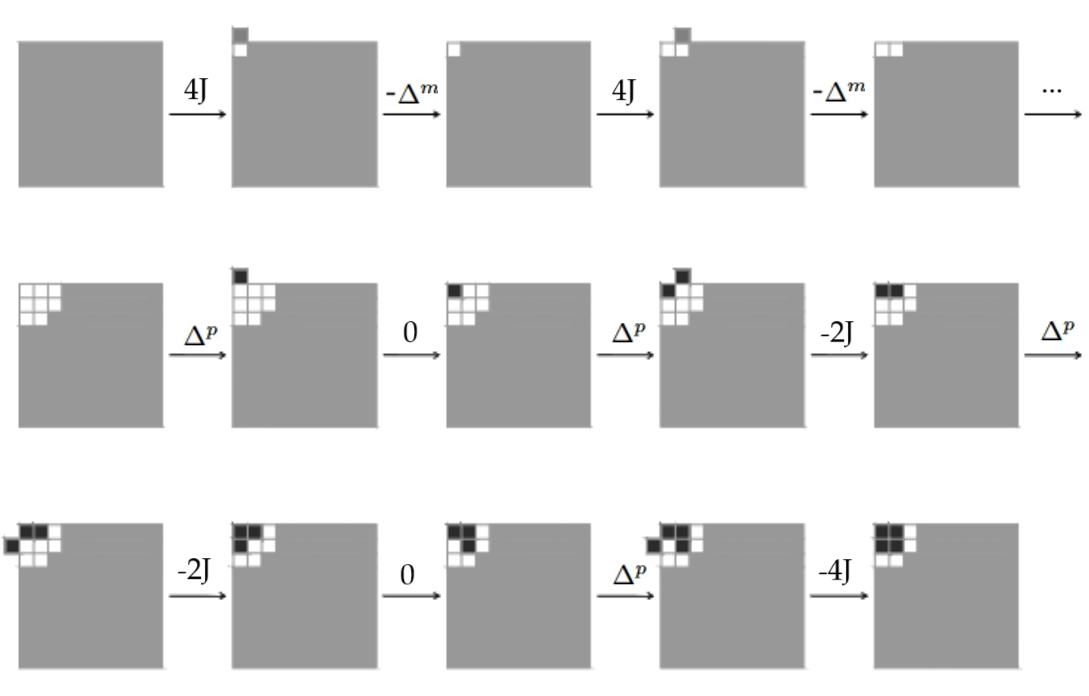}
    \caption{The first part of the reference path. White, light gray and dark grey squares represent zero, minus and plus spins respectively.}\label{fig:ref_path_1}
\end{center}
    \end{figure}

\subsubsection{Upper bound for $V_{\muno}$}
\label{sec:upper}
\par\noindent
In this section, we construct a path from $\muno$ to $\puno$, 
the so-called \emph{reference path}, and we found an upper bound for 
$V_\muno$. 
Starting from $\sigma_0 = \muno$, the system follows the path in Figure \ref{fig:ref_path_1}
% In particular, starting from $\sigma_0 \equiv \muno$, we consider one of the four minuses located at one of corners of $\Lambda_0$. Without loss of generality, we consider the minus in the upper-left corner and we move it to $\partial^+ \Lambda_0$. The energy cost of this swap is $4J$, indeed two bonds between two minuses are replaced by two bonds between a minus and a zero. Then, we get the minus out of $\Lambda$ and the energy decreases by $\Delta^m$. In a similar way, we replace all the minuses at distance smaller then or equal to $\sqrt{5}$ from the zero in the corner, see Figure \ref{fig:ref_path_1}. The energy of the system increases by $8(\lambda-h)$. Next, we create a square $2 \times 2$ of pluses. In particular, we get a plus in $\partial^+\Lambda_0$ with an energy cost equal to $\Delta^p$ and we move it to the upper-left corner. Then, another plus enters in $\partial^+\Lambda_0$ with the same energy cost and we put it at distance one from the first one. In this way, the energy decreases by $2J$, indeed a bond between a plus and a zero is replaced by a bond between two pluses. We add the third plus in the same way and then we move it to the site at distance one from the second plus, see Figure \ref{fig:ref_path_1}. The fourth plus enters in $\partial^+\Lambda_0$ and it is placed in $\Lambda_0$ so that a square of pluses is formed. This swap has an energy cost equal to $4J$, indeed two bonds between a plus and a zero are replaced by two bonds between two pluses. 
until it reaches the configuration with a frame in the corner as in the bottom right panel in Figure \ref{fig:ref_path_1}. We denote such a configuration by $\sigma_{2,2}$ and we continue the path towards $\puno$ with the mechanism described in the figures \ref{fig:ref_path_2}-\ref{fig:ref_path_3} for a general side length $l$ of the frame
(we denote the corresponding configuration with $\sigma_{l,l}$).
%
% We obtained the configuration with a corner frame as in the bottom right picture in Figure \ref{fig:ref_path_1}. We denote such a configuration by $\sigma_{2,2}$ and we continue the growing path with the mechanism described below for a general side length $l$ of the frame, see Figure \ref{fig:ref_path_2}. More precisely, $\sigma_{l,l}$ is a configuration in which a plus square of side length $l$ is placed at the corner of $\Lambda_0$ and separated from the sea of minuses by two strips of zeros, see the upper left picture in Figure \ref{fig:ref_path_2}. This configuration, called corner frame, has already appeared in the paper \cite{cirillo2024homogeneous}.
%
In particular, starting from $\sigma_{l,l}$, the path reaches $\sigma_{l,l+1}$ by crossing the configuration $\eta$, i.e. the configuration in the 10-th panel in Figure \ref{fig:ref_path_2}, and we have 
%
% Starting from $\sigma_{l,l}$, consider one of the two minuses at distance one from $\partial^+\Lambda_0$ and from the strip of zeros. It is replaced by a zero in $\partial^+\Lambda_0$ with an energy cost equal to $4J$, see the second picture in Figure \ref{fig:ref_path_2}. Then, the minus in $\partial^+ \Lambda_0$ leaves $\Lambda$ and the energy of the system decreases by $\Delta^m$. Next, the minus at distance one from the new zero and from the strip of zeros is exchanged by its nearest neighbour zero in the frame, see the third picture in Figure \ref{fig:ref_path_2}. The energy cost of this swap is $2J$, since a bond between two spins minus is broken. Next, this minus is replaced by a zero in $\partial^+\Lambda_0$ with an energy cost equal to $4J$ and then it leaves $\Lambda$ by decreasing the energy of the system by $\Delta^m$.
% At this point, the droplet of pluses starts to grow, see Figure \ref{fig:ref_path_2}. A plus appears in $\partial^+\Lambda_0$ and the energy of the system increases by $\Delta^p$. Then, this spin is exchanged with the first and the the second spin zero in the frame and it is attached to the square of plus at distance two from the boundary. When the plus spin is attached to the droplet, the energy of the system decreases by $2J$, since a new bond is created between two particles with the same spin, instead the energy cost associated with the second exchange between the plus and the zero spin along the frame is zero. 
%
%The second plus appears in $\partial^+\Lambda_0$ with the same energy cost and we call this configuration $\eta$. We have 
\begin{align}\label{eq:saddle_grow}
H(\eta)-H(\sigma_{l,l})=6J+2(\lambda-h)-2(\lambda+h). 
\end{align}
Then, the path continues without never overcoming the energy of $\eta$.

\begin{figure}[t]
\begin{center}
    \includegraphics[scale=0.7]{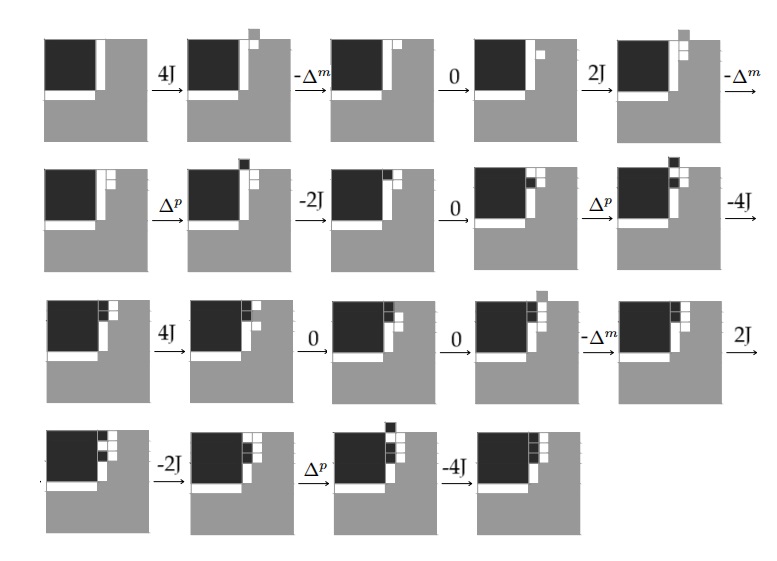}
    \caption{In this part of the reference path, a frame in the corner with side length $l$ starts to grow.}\label{fig:ref_path_2}
\end{center}
    \end{figure}
    %
%Indeed, when the second plus is attached to the droplet near the first protuberance (obtained by attaching the first plus), two new bonds are created and the energy of the system decreases by $4J$. Then, the procedure goes forward as follows: the minus close to the plus protuberance is exchange with zero spins until it leaves $\Lambda$, then all plus spins move along the frame to leave space for a new plus that enters in $\partial^+\Lambda_0$ and it attaches to the cluster of pluses. We iterate this mechanism along the considered side of the frame until $l-1$ pluses are attached to the cluster of pluses. 
% We note that, before the last plus gets in, we remove not only one minus but two, indeed in this case there are two minuses at distance smaller than $2$ from the cluster of pluses, see Figure \ref{fig:ref_path_3}. So the energy cost is twice $4J$ when one of the two minuses is detached and $-\Delta^m$ when it leaves $\Lambda$.
% Then, the plus get in $\partial^+\Lambda_0$ and it is attached from the cluster of pluses by forming a chopped corner frame $\sigma_{l,l+1}$. 
%
\begin{figure}[H]
\begin{center}
    \includegraphics[scale=0.6]{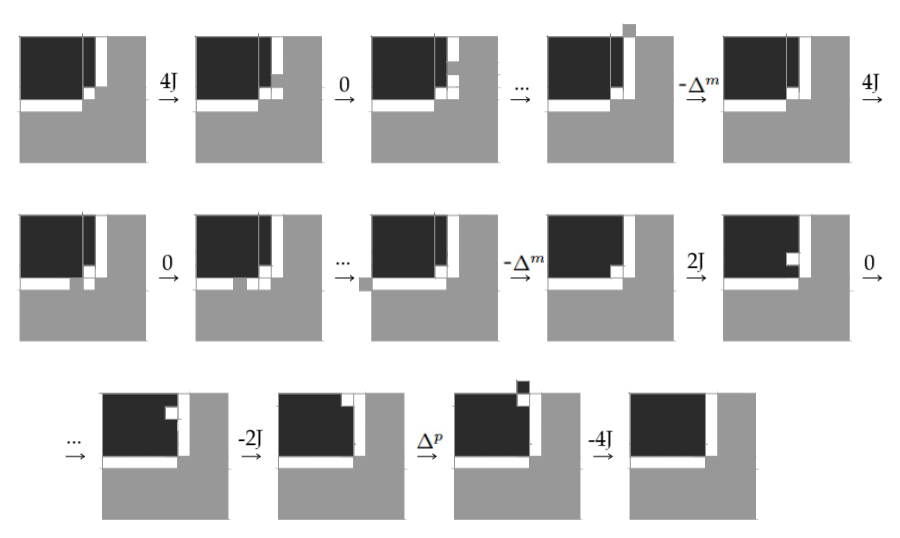}
    \caption{In this part of the reference path, a frame in the corner with side lengths $l$ and $l+1$ is created.}\label{fig:ref_path_3}
\end{center}
    \end{figure}
We repeat this mechanism of growing of the frame in the corner along the shortest side, until we reach the homogeneous phase $\puno$. 

In order to find the energy barrier to reach $\puno$ starting from $\muno$, we compute the value of the maximal energy of  $\sigma_{n,m}$
and we sum to it the value in \eqref{eq:saddle_grow}. The energy of $\sigma_{n,m}$ with $m,n \in \mathbb{N}$ respect to the energy of the homogeneous state $\muno$ is
\begin{equation}\label{eq:energy_chopped_corner_frame}
H(\sigma_{n,m})-H(\muno)=2J(n+m)+\lambda (n+m)-h(2nm+n+m).
\end{equation}
The maximal value of this function is attained for $n=m= \frac{2J+\lambda-h}{2h}$, however $n$ and $m$ are integer number, so with a simple computation we find that the chopped corner frame with maximal energy has size 
$l_c= \lfloor\frac{2J+\lambda-h}{2h} \rfloor+1$ 
and $l_c-1$. Thus, the height along the reference path is reached in a configuration $\sigma_s$ obtained as $\eta$, see as Figure \ref{fig:critical_configuration}, and its value is equal to
\begin{align}\label{eq:energy_barrier_RF}
H(\sigma_s)-H(\muno)&=2J(2l_c-1)+\lambda (2l_c-1)-h(2l_c^2-1) +6J+2(\lambda-h)-2(\lambda+h) \notag \\
&=4J(l_c+1)+\lambda (2l_c-1)-h(2l_c^2+3)
\end{align}
and then 
\begin{equation}\label{eq:upper_bound_gamma}
V_{\muno} \leq H(\sigma_s)-H(\muno).
\end{equation}

\subsubsection{Lower bound for $V_{\muno}$}\label{sec:lower}
\par\noindent
In order to find the lower bound for $V_\muno$, we use \cite[Lemma 4.14]{cirillo2024homogeneous}, that we report here for the convenience of the reader.
Let $n^+_c=l_c(l_c-1)$ and denote by $\mathcal{M}_{n^+_c} \subset \mathcal{X}$ (resp. $\mathcal{M}_{n^+_c+1} \subset \mathcal{X}$) the manifold with fixed number $n^+_c$ (resp. $n^+_c+1$) of pluses. We recall, the definition of the configuration $\sigma_c=\sigma_{l_c-1,l_c}$ provided in Figure \ref{fig:critical_configuration}. 
\begin{lemma}\label{lem:minmax_K}
 $H(\sigma_c)=\min_{\xi \in \mathcal{M}_{n^+_c}} H(\xi)$.
\end{lemma}
Note that this lemma was proven in \cite{cirillo2024homogeneous} for the Blume-Capel model, with the Hamiltonian \eqref{def:hamiltonian_BCK} that we use in this paper, with the Glauber dynamics. We remark that the result is valid also in this setting, since it does not depend on the dynamics but only on the structure of the Hamiltonian. 

Consider $\underline{\omega}$ from $\muno$ to $\puno$ and let $\sigma\in\mathcal{M}_{n^+_c}$ and $\eta\in\mathcal{M}_{n^+_c+1}$ be two consecutive configurations along this path. Then $\Phi(\underline{\omega})\geq H(\eta)=H(\sigma)+\Delta^p$. By Lemma \ref{lem:minmax_K}, we have $H(\sigma)\geq H(\sigma_c)$ and therefore $V_{\muno} \geq H(\sigma_c)+\Delta^p-H(\muno)$.

\subsection{Proof of Theorem~\ref{thm:Identification}}
\label{sec:Identification}
\par\noindent
By \eqref{mod007b}
it follows that $V^*<H(\sigma_c)+\Delta^p-H(\muno)$. Thus, 
by Theorems~\ref{thm:out_gammaBCK} and \ref{thm:vmuno} it follows that 
the stability level 
of $\muno$ is larger that the stability level of all other 
configurations differing from $\puno$.

Since, 
the stability level of $\muno$ 
is the maximal one for the 
configurations in $\mathcal{X}\setminus\{\muno\}$, 
by 
\cite[Theorem 2.4]{cirillo2013relaxation} 
we have that $\muno$ is the 
unique metastable state for the system. 

\subsection{Proof of Theorem~\ref{thm:transition_time_BCG}}
\par\noindent
Remark that $V_\muno$ 
is
the maximal stability level. Then,
by applying \cite[Theorem 4.1]{manzo2004essential} with $\eta_0=\muno$ 
we get that
\begin{displaymath}
    \lim_{\beta \to \infty} 
\mathbb{P}_{\muno}(e^{\beta(V_\muno-\epsilon)}< 
             \tau_{\puno}<e^{\beta(V_\muno+\epsilon)}) =1, 
\end{displaymath}
for any $\epsilon>0$.
Thus the theorem follows by noting that the event 
$\{e^{\beta(V_\muno-\epsilon)}< 
             \tau_{\puno}<e^{\beta(V_\muno+\epsilon)}\}$
is a subset of the event $\{e^{\beta(H(\sigma_c)+\Delta^p-H(\muno)-\epsilon)}< 
             \tau_{\puno}<e^{\beta(H(\sigma_s)-H(\muno)+\epsilon)}\}$.

\section{Proof of lemmas}
\label{s:pro-le}
\par\noindent
In this section we prove all the auxiliary lemmas related to the recurrence property and stated in Section \ref{sec:auxiliary}.

\begin{proof} [Proof of Lemma \ref{thm:grafo_pesato}]
Before starting the proof, we recall that a \emph{complete weighted directed graph} $G=(V,E)$ is a graph %with vertex set $V=\{v_1,...,v_n\}$ , 
such that a \emph{weight} $w{(u,v)}$ is associated to every edge $(u,v)$ of the graph and every vertex is connected with the others. Moreover, given a path $p$ on $G$, the sum of the weights along the path is called \emph{path-weight} and denoted by $w_p$. 
Given a path $p=(v_1,...,v_n)$, the \emph{maximal sub-path weight} is defined as $m_p=\max_{k \in \{1,...,n-1 \}} \sum_{i=1}^{k} w(v_i,v_{i+1})$.

Let $\sigma \in \mathcal{X}$ be a configuration as in the assumption. In particular, we suppose that $\sigma$ contains a zero-carpet. The proof of the other cases is analogous.

We let $x_0:=x$ and let $x_1,...,x_n$ be a sequence of sites of $X'$ such that $x_i$ is connected with $x_{i+1}$ for $i=0,...,n-1$.
Recall $\sigma(x)=r\in\{-1,+1\}$. 

We consider a weighted directed graph $G=(V,E)$ with self-loops such that each vertex is a non-ordered pair of spins and it is connected with all the others.
%For simplicity, we denote by $x_i$ the vertex $(a_i,\sigma(x_i),b_i)$. 
Each weight $w(v,u)$ is defined as reported in Figure \ref{figure:graph_carpet}.

We associate to every site $x_i$ of the sequence $ x_0,x_1,...,x_n $ a vertex in $V$ such that the nearest sites of $x_i$ different from $x_{i-1}$ and $x_{i+1}$ have spins equal to the vertex. We denote by $v_i$ the vertex associated with $x_i$. We recall that the sites $ x_0,x_1,...,x_n $ have two nearest spins equal to zero. 

We consider $x_i, x_j$ as above and we note that $H(\sigma^{(x,x_j)})-H(\sigma^{(x,x_i)})$ is equal to the weight $w(v_i,v_j)$;
for the sake of clarity the weight already reported in the left panel of Figure \ref{figure:graph_carpet} are reported also in \ref{tab:graph}.
%Roughly, the positive weight shows that the energy increases when a spin $r$ is exchanged with a zero in the zero-carpet due to the decreasing of the number of the bonds with equal spins or the increasing of the number of bonds with opposite spins. On the other hand, the negative weight shows that the energy decreases when the spin $r$ is exchanged with a zero due to the increasing of the number of the bonds with equal spins or the decreasing of the number of the bonds with opposite spins. 
%
\begin{table}\label{tab:graph}
\begin{center}
\begin{tikzpicture} [scale=0.6]
    \path (-3,0) edge (16,0);
    \path (0,1) edge (16,1);
    \path (0,2) edge (16,2);
    \path (0,3) edge (16,3);
    \path (0,4) edge (16,4);
    \path (-3,5) edge (16,5);
    \path (4,6) edge (16,6);
    
    \path (2,7) edge (16,7);
    \path (-3,5) edge (-3,0);
    
    \path (-3,5) edge (-3,7);
    \path (-3,7) edge (2,7);
    
    \path (0,0) edge (0,5);
    \path (4,0) edge (4,7);
    \path (6,0) edge (6,6);
    \path (8,0) edge (8,6);
    \path (12,0) edge (12,6);
    \path (14,0) edge (14,6);
    \path (16,0) edge (16,7);
    
    \node [scale=0.9] at (2,0.5) {$(s,s)$}; 
    \node [scale=0.9] at (2,1.5) {$(0,s)$};
    \node [scale=0.9] at (2,2.5) {$(0,0)$ or $(s,r)$};
    \node [scale=0.9] at (2,3.5) {$(0,r)$};
    \node [scale=0.9] at (2,4.5) {$(r,r)$};
    
    \node [scale=0.9] at (5,5.5) {$(r,r)$}; 
    \node [scale=0.9] at (7,5.5) {$(0,r)$};
    \node [scale=0.9] at (10,5.5) {$(0,0)$ or $(s,r)$};
    \node [scale=0.9] at (13,5.5) {$(0,s)$};
    \node [scale=0.9] at (15,5.5) {$(s,s)$};
    
    \node [scale=0.9] at (-1.5,4) {spins}; 
    \node [scale=0.9] at (-1.5,3) {of nearest};
    \node [scale=0.9] at (-1.5,2) {neighbours};
    \node [scale=0.9] at (-1.5,1) {of $x$};
    \node [scale=0.9] at (10,6.5) {spins of nearest neighbours of $y$};
    
    \node [scale=0.9] at (0.5,6) {$H(\sigma^y)-H(\sigma)$};
    
    \node [scale=0.9] at (5,4.5) {$0$};
    \node [scale=0.9] at (7,4.5) {$2J$};
    \node [scale=0.9] at (10,4.5) {$4J$};
    \node [scale=0.9] at (13,4.5) {$6J$};
    \node [scale=0.9] at (15,4.5) {$8J$};
    
    \node [scale=0.9] at (5,3.5) {$-2J$};
    \node [scale=0.9] at (7,3.5) {$0$};
    \node [scale=0.9] at (10,3.5) {$2J$};
    \node [scale=0.9] at (13,3.5) {$4J$};
    \node [scale=0.9] at (15,3.5) {$6J$};
    
    \node [scale=0.9] at (5,2.5) {$-4J$};
    \node [scale=0.9] at (7,2.5) {$-2J$};
    \node [scale=0.9] at (10,2.5) {$0$};
    \node [scale=0.9] at (13,2.5) {$2J$};
    \node [scale=0.9] at (15,2.5) {$4J$};
    
    \node [scale=0.9] at (5,1.5) {$-6J$};
    \node [scale=0.9] at (7,1.5) {$-4J$};
    \node [scale=0.9] at (10,1.5) {$-2J$};
    \node [scale=0.9] at (13,1.5) {$0$};
    \node [scale=0.9] at (15,1.5) {$2J$};
    
    \node [scale=0.9] at (5,0.5) {$-8J$};
    \node [scale=0.9] at (7,0.5) {$-6J$};
    \node [scale=0.9] at (10,0.5) {$-4J$};
    \node [scale=0.9] at (13,0.5) {$-2J$};
    \node [scale=0.9] at (15,0.5) {$0$};
    
\end{tikzpicture}
\end{center}
\caption{Energy difference for all the possible choices of the 
local configuration of the sites $x$ and $y$ zeros after the swap 
of the particle originally at site $x$. The two neighbouring spins not specified in the table are zeros. }
\label{tab:weighted}
\end{table}

We consider a path from $\sigma^{(x,x_i)}$ to $\sigma^{(x,x_j)}$ composed by the following pairwise communicating configurations 
$\underline{\omega}=(\sigma^{(x,x_i)}, \sigma^{(x,x_{i+1})}, ...,  \sigma^{(x,x_{j-1})}, \sigma^{(x,x_j)})$. 
In order to compute the height of this path, we consider the associated path $p_{\underline{\omega}}$ on the graph $G$ and we compute its maximal sub-path weight.

In the following, we explain in detail how to compute the maximal sub-path weight of a general path $p$ on $G$. 
First, we consider the path $p'$ obtained removing all the loops in $p$. We note that the path-weight of a loop is $0$, because the weight of the edge $(v,u)$ is defined as the difference of a real-valued function $f$ calculated in the two vertices $v$ and $u$, i.e. $w(v,u)=f(v)-f(u)$. 
Hence, it follows that $w(v,u)=-w(u,v)$ for each pair of vertices $v,u$.
Moreover, from Figure \ref{figure:graph_carpet}, the maximal sub-path weight of $p$ is smaller than or equal to $8J$.  
Thus, $\Phi_{\underline{\omega}}a=m_{p_{\underline{\omega}}} \leq 8J$.

\end{proof}
% \textbf{Example.}
% Let $p=(v,u,w,z,u,w,z,u,u,w,u,u,z)$ be the path in input. We obtain the path $p'=(v,u,w,u,z)$ by applying the algorithm one time. We note that $p'$ contains other loops, so we apply the algorithm again and we obtain the path $p''=(v,u,z)$ that is loop free. 
% \vspace{0.5cm}

\begin{proof}[Proof of Lemma \ref{lem:transport_zero}]
(i). 
Let $\sigma$ be a configuration as in the assumption and let $\eta=\sigma^{(x;s)}$. Let $s \in \{-1,+1 \}$. Let $x_1,...,x_{n} \in X$ and $x_0 \in \partial^+ \Lambda_0$ such that $|x_i-x_{i-1}|=1$ for $i=1,...,n$ and $|x_n-x|=1$. We construct a path $\underline{\omega}=(\sigma, \omega_0, \omega_1,..., \omega_n, \eta)$ in the following way:
\begin{align}
    \omega_0:=\sigma^{(x_0;s)}; \qquad
    \omega_i:=\omega_{i-1}^{(x_{i-1},x_{i})} \text{ for } i=1,...n; \qquad
    \omega_{n+1}:=\omega_n^{(x_n,x)}.
\end{align}

We observe that $\omega_{n+1}=\eta$. 
By \eqref{def:hamiltonian_BCK} and \eqref{def:cost_swap}, we have
\begin{align}
    H(\eta)-H(\sigma)& =[H(\eta)-H(\omega_0)]+[H(\omega_0)-H(\sigma)] \notag \\
    &\leq 2J\Big [ [ D_x(\omega_0)+D_{x_0}(\omega_0)]-[D_x(\eta)+D_{x_0}(\eta)] \Big] +\Delta_s=-2JD_{x}(\eta)+\Delta_s,
\end{align}
where the equality holds when $\sigma(x_0) = 0$.
We compute now the height between $\sigma$ and $\sigma^{(x;s)}$ along the path $\underline{\omega}$. 
We get the particle $s$ in $\partial^+\Lambda_0$ if $\sigma(x) \neq s$, and the energy increases at most by $\Delta_s$, i.e. $H(\omega_0) \leq H(\sigma)+\Delta_s$. 
By applying Lemma \ref{thm:grafo_pesato} to $\omega_0$, we have 
\begin{equation}
    \Phi((\omega_0,...,\omega_n)) \leq H(\omega_0)+8J \leq H(\sigma)+\Delta_s+8J.
\end{equation}
Thus, 
\begin{align}
    \Phi((\sigma,...,\eta)) &=\max \{ \Phi ((\omega_0,...,\omega_n)), \Phi ((\omega_n,\eta))\} \notag \\
    & \leq H(\sigma)+\Delta_s+8J+ 2J |D_{x}(\eta) | \leq H(\sigma)+18J.
\end{align}

(ii).
Let $\sigma$ be a configuration as in the assumption. We let $x_0=x$ and $\sigma(x)=s \neq 0$. 
Let $x_1,...,x_{n} \in X$ and $x_n \in \partial^+ \Lambda_0$ such that $|x_i-x_{i-1}|=1$ for $i=1,...,n$. We construct a path $\underline{\omega}=(\sigma, \omega_0, \omega_1,..., \omega_n, \eta)$ in the following way:
\begin{align}
    \omega_0:=\sigma; \qquad \omega_i:=\omega_{i-1}^{(x_{i-1},x_i)} \text{ for } i=1,...,n; \qquad \omega_{n+1}=\sigma^{(x;0)}.
\end{align} 
By \eqref{def:hamiltonian_BCK} and \eqref{def:cost_swap}, we obtain
\begin{align}
&    H(\omega_{n})-H(\sigma)=2JD_x(\sigma), \notag \\
& H(\omega_{n+1})-H(\omega_n)=-\Delta_s
\end{align}
Moreover, $\Phi((\omega_n, \omega_{n+1}))=H(\omega_n)$ since $H(\omega_{n+1})<H(\omega_n)$, and by applying Lemma \ref{thm:grafo_pesato} to $\omega_1$ we have that 
\begin{align}
    \Phi((\omega_1,..., \omega_n)) \leq H(\omega_1)+8J \leq H(\omega_0)+2J |D_{x_0,x_1}(\omega_0)| +8J,
\end{align}  
Thus, we have
\begin{align}
    \Phi((\omega_0, \omega_{n+1})) & \leq \max \{ \Phi ((\omega_0, \omega_1)), \Phi ((\omega_1,..., \omega_{n})) \} \notag \\
    & = \max \{H(\omega_0)+2J |D_{x_0,x_1}(\omega_0)|, H(\omega_0)+2J |D_{x_0,x_1}(\omega_0)| +8J \}\leq H(\omega_0)+14J,
\end{align}
where the last inequality follows from the fact that in the worst case $x_0$ has three nearest neighbours with spin equal to $s$ and the two nearest neighbours of $x_1$ not in $X$ has spin value $-s$. 
\end{proof}

\begin{proof}[Proof of Lemma \ref{lemma:transport_scarpet}]
    (i). Let $\sigma$ be a configuration as in the assumption and let $\eta=\sigma^{(y;0)}$. We suppose first $\sigma(y)=-s$. Set $y_0=y$ and let $y_1,...,y_{n} \in X$ be such that $|y_i-y_{i-1}|=1$ for $i=1,...,n$ and $y_n \in \partial^- \Lambda_0$. 
    Let $z \in \partial^+ \Lambda_0$ be at distance one from $X$ and, without loss of generality, we assume that $\sigma(z)=0$. We construct a path $\underline{\omega}$ from $\sigma$ to $\eta$ in the following way:
    \begin{itemize}
        \item[1.] the particle with spin $-s$ in $y$ moves through the s-carpet until $\partial^-\Lambda_0$;
        \item[2.] the $-s$ in $\partial^-\Lambda_0$ is exchanged with a zero in $\partial^+ \Lambda_0$ and then it leaves $\Lambda$;
        \item[3.] the zero in $\partial^-\Lambda_0$ moves along the s-carpet until it reaches the site $y$.
    \end{itemize}
    Formally, we define the path $\underline{\omega}=(\omega_0,\omega_1,\omega_n,\omega_{n+1}, \xi_0, \xi_1,...,\xi_n)$ such that
\begin{align}
    \omega_0:=\sigma; \qquad &
    \omega_i:=\omega_{i-1}^{(y_{i-1},y_{i})} \text{ for } i=1,...n; \qquad \,\,\,
    \omega_{n+1}:=\omega_n^{(y_n,z)}; \\
    \xi_0:=\omega_{n+1}^{(z;0)}; \qquad &
    \xi_i:=\xi_{i-1}^{(y_{n-i+1},y_{n-i})} \text{ for } i=1,...,n.
\end{align}
We note that $\xi_n= \eta$.
By \eqref{def:hamiltonian_BCK} and \eqref{def:cost_swap}, we obtain
\begin{align}
& H(\eta)-H(\sigma)=2JD_y(\sigma)-\Delta_{-s}
\end{align}

Next, we compute the height between $\sigma$ and $\eta$ along this path. 
\begin{align*}
    \Phi(\underline{\omega})= \max \{ &\Phi((\sigma,\omega_1)),\, \Phi((\omega_1,...,\omega_{n-1})), \, \Phi((\omega_{n-1}, \omega_n, \omega_{n+1})), \, \Phi((\omega_{n+1},\xi_0)), \Phi((\xi_0,\xi_1)), \notag \\
    & \Phi((\xi_1,...,\xi_{n-1})), \, \Phi((\xi_{n-1}, \eta))\}
\end{align*}
We note that $\Phi ((\sigma,\omega_1)) = \max \{ H(\sigma), H(\omega_1) \} \leq H(\sigma)+2J| D_{y_0,y_1}(\sigma)|$,
and by applying Lemma \ref{thm:grafo_pesato} to the configuration $\omega_1$ we have that 
\begin{align}
    \Phi((\omega_1,...,\omega_{n-1})) \leq H(\omega_1)+16J \leq H(\sigma)+2J| D_{y_0,y_1}(\sigma)| + 16J
\end{align}  
Then, $\Phi ((\sigma,\omega_1)) \leq \Phi((\omega_1,..., \omega_{n-1}))$.
Moreover, $\Phi((\omega_{n+1},\xi_0))=H(\omega_{n+1})$ since $H(\xi_0)<H(\omega_{n+1})$, then $\Phi((\omega_{n+1},\xi_0)) \leq \Phi((\omega_{n-1},\omega_n, \omega_{n+1}))=\max \{ H(\omega_{n-1}), H(\omega_n), H(\omega_{n+1}) \}$ and 
\begin{align}
    & H(\omega_{n+1})=H(\omega_n)+2JD_{y_n,z}(\omega_n) \notag \\
    & H(\omega_{n})=H(\omega_{n-1})+2JD_{y_{n-1},y_n}(\omega_{n-1}) \notag \\
    & H(\omega_{n-1}) \leq H(\omega_1)+16J
\end{align}
where the last inequality follows from Lemma \ref{thm:grafo_pesato}. Then,
\begin{align}
\Phi((\omega_{n-1},\omega_n,\omega_{n+1})) &\leq H(\omega_1)+2J|D_{y_n,z}(\omega_n)| +2J|D_{y_{n-1},y_n}(\omega_{n-1})| +16J \notag \\
& \leq H(\sigma)+2J| D_{y_0,y_1}(\sigma)| +2J|D_{y_n,z}(\omega_n)| +2J|D_{y_{n-1},y_n}(\omega_{n-1})| +32J
\end{align} 
Thus 
\begin{align}\label{eq:dis_Phi1}
\max \{ \Phi ((\sigma,\omega_1)), \Phi((\omega_1,..., \omega_{n-1})), \Phi((\omega_{n-1},\omega_n, \omega_{n+1})) \} \leq H(\sigma)+54J,
\end{align}
since $2J| D_{y_0,y_1}(\sigma)| \leq 10J$, $2J|D_{y_n,z}(\omega_n)| \leq 4J$ and $2J|D_{y_{n-1},y_n}(\omega_{n-1})| \leq 8J$.
 
Furthermore, $\Phi((\xi_0, \xi_1))= \max \{ H(\xi_0), H(\xi_1) \} \leq H(\xi_0)+2J| D_{y_{n},y_{n-1}}(\xi_0)|$, and by applying Lemma \ref{thm:grafo_pesato} to $\xi_1$ we have that 
\begin{align}
    \Phi((\xi_1,..., \xi_{n-1})) \leq H(\xi_1)+8J \leq H(\xi_0)+2J| D_{y_{n},y_{n-1}}(\xi_0)| +8J,
\end{align} 
then $\Phi((\xi_0, \xi_1))\leq \Phi((\xi_1,..., \xi_{n-1}))$.
Finally, 
\begin{align}
    \Phi ((\xi_{n-1},\xi_n))=\max \{ H(\xi_{n-1}), H(\xi_n) \} &\leq H(\xi_{n-1})+2J| D_{y_1,y_0}(\xi_{n-1})| \notag \\
    &\leq H(\xi_1)+8J+2J| D_{y_1,y_0}(\xi_{n-1})| \notag \\
    & \leq H(\xi_0)+2J| D_{y_{n},y_{n-1}}(\xi_0)| +8J +2J| D_{y_1,y_0}(\xi_{n-1})|
\end{align}
where the second inequality is obtained by using Lemma \ref{thm:grafo_pesato}. Then, 
\begin{align}\label{eq:dis_Phi2}
\max \{ \Phi((\xi_0, \xi_1)), \Phi((\xi_1,...,\xi_{n-1})), \Phi ((\xi_{n-1},\xi_n)) \} \leq H(\xi_0)+26J,
\end{align}
since $2J| D_{y_{n},y_{n-1}}(\xi_0)| \leq 8J$ and $2J| D_{y_1,y_0}(\xi_{n-1})| \leq 10J$.

Thus, by \eqref{eq:dis_Phi1} and \eqref{eq:dis_Phi2}, we obtain
\begin{align}
    \Phi((\sigma,..., \xi_{n})) & \leq \max \{ H(\sigma)+54J, \, H(\xi_0)+26J \}.
\end{align}
Moreover
\begin{align}
    H(\xi_0)-H(\sigma)& =[H(\xi_0)-H(\omega_{n+1})]+[H(\omega_{n+1})-H(\omega_{n})]+[H(\omega_{n})-H(\sigma)] \notag \\
    & =-\Delta_{-s}+2J| D_{y_n,z} (\omega_n)| +2J| D_{y_0,y_n} (\sigma)| \notag \\
    &\leq -\Delta_{-s} +14 J < 10J+(\lambda -s h) <11J
\end{align}
since $2J| D_{y_n,z} (\omega_n)| \leq 4J$, $2J| D_{y_0,y_n} (\sigma)| \leq 10 J$ and Condition \ref{mod007b}. Thus, we conclude
\begin{align}
    \Phi((\sigma,..., \omega_{n+1})) \leq H(\sigma)+54J.
\end{align}

In the case of $\sigma(y)=s$, we proceed in a similar way by constructing a path from $\sigma$ to $\eta$ composed by only the last part of the previous path. i.e. starting from $\sigma$ the zero in $\partial^-\Lambda_0$ moves along the s-carpet until it reaches the site $y$.

(ii).
Let $\sigma$ be a configuration as in the assumption, and let $\eta=\sigma^{(y;s)}$. Set $y=y_0$ and let $y_1,...,y_{n} \in X$ be such that $|y_i-y_{i-1}|=1$ for $i=1,...,n$ and $y_n \in \partial^- \Lambda_0$. 
Let $z \in \partial^+ \Lambda_0$ and we construct a path $\underline{\omega}$ from $\sigma$ to $\eta$ in the following way:
    \begin{itemize}
        \item[--] the particle with zero spin in $y$ crosses the s-carpet until $y_n \in \partial^-\Lambda_0$;
        \item[--] a particle with spin $s$ gets in $\partial^+ \Lambda_0$;
        \item[--] the particle moves from $z$ to $y_n$ by exchanging with the zero in $y_n$. 
    \end{itemize}
    Formally, we define the path $\underline{\omega}=(\omega_0,\omega_1,\omega_n,\omega_{n+1}, \xi_0, \xi_1,...,\xi_n)$ such that
\begin{align}
    & \omega_0:= \sigma; \qquad \qquad 
    \omega_i:=\omega_{i-1}^{(y_{i-1},y_{i})} \text{ for } i=1,...n; \\
    &\xi_1:=\omega_n^{(z;s)}; \qquad \,\,
    \xi_2:=\omega_{n+1}^{(z,y_n)}.
\end{align}
We note that $\xi_2= \eta$. By \eqref{def:hamiltonian_BCK} and \eqref{def:cost_swap}, we obtain
\begin{align}
& H(\eta)-H(\sigma)=[H(\eta)-H(\xi_1)]+[H(\xi_1)-H(\sigma)]=-2JD_y(\eta)+\Delta_{s}
\end{align}

Next, we compute the height between $\sigma$ and $\eta$ along this path. 
\begin{align*}
    \Phi(\underline{\omega})= \max \{ \Phi((\sigma,\omega_1)),\, \Phi((\omega_1,...,\omega_{n-1})), \, \Phi((\omega_{n-1}, ...,\eta)), \}
\end{align*}
We note that $\Phi((\sigma,\omega_1)) = \max \{ H(\sigma), H(\omega_1) \} \leq H(\sigma)+2J| D_{y_0,y_1}(\sigma)|$,
and by applying Lemma \ref{thm:grafo_pesato} to the configuration $\omega_1$ we have that 
\begin{align}\label{eq:Phi1_part2}
    \Phi((\omega_1,..., \omega_{n-1})) \leq H(\omega_1)+8J \leq H(\sigma)+2J| D_{y_0,y_1}(\sigma)| + 8J \leq  H(\sigma)+18J
\end{align}  
since $2J| D_{y_0,y_1}(\sigma)| \leq 10 J$. Then, $\Phi((\sigma,\omega_1)) \leq \Phi((\omega_1,..., \omega_{n-1}))$.

Moreover, $\Phi((\omega_{n-1},...,\eta))=\max \{ H(\omega_{n-1}), H(\omega_n), H(\xi_1), H(\eta) \}$ and 
\begin{align}
    & H(\eta)= H(\xi_1)+2JD_{z,y_n}(\xi_1) \leq H(\xi_1)+4J \notag \\
    & H(\xi_1)=H(\omega_n)+\Delta_s \notag \\
    & H(\omega_{n})=H(\omega_{n-1})+2JD_{y_{n-1},y_n}(\omega_{n-1}) \leq H(\omega_{n-1})+8J \notag \\
    & H(\omega_{n-1}) \leq H(\omega_1)+8J
\end{align}
where the last inequality follows from Lemma \ref{thm:grafo_pesato}. Then,
\begin{align}\label{eq:Phi2_part2}
\Phi((\omega_{n-1},...,\eta)) \leq H(\omega_1)+\Delta_s+20J \leq H(\omega_1) +25J \leq H(\sigma)+2J| D_{y_0,y_1}(\sigma)| +25 J \leq H(\sigma)+35J
\end{align} 
Thus, by \eqref{eq:Phi1_part2} and \eqref{eq:Phi2_part2}, we conclude $\Phi((\sigma,..., \eta)) \leq H(\sigma)+35J$.

\end{proof}

\begin{proof}[Proof of Lemma \ref{lem:bond_pm}]
Let $\sigma$ be a configuration as in the assumption. Suppose that $\sigma_{\partial^+ \Lambda_0}=0_{\partial^+ \Lambda_0}$, otherwise the stability level of $\sigma$ is zero, indeed when a plus or a minus leaves $\Lambda$ the energy decreases. 
Assume first that the bond $(+,-)$ is at distance one from $\partial^+ \Lambda_0$. We consider the configuration $\eta=\sigma^{(x;0)}$ where $x$ is the site of the bond at distance one from $\partial^+ \Lambda_0$. Suppose without loss of generality that $\sigma(x)=-1$, then with a direct computation we have
\begin{align}
    H(\eta)-H(\sigma) \leq -2J+(\lambda-h)<0.
\end{align}
We note that $\sigma$ and $\eta$ are not communicating configurations. Let $\xi=\sigma^{(x,x')}$ where $x' \in \partial^+ \Lambda_0$ is at distance one from $x$, then $\sigma \sim \xi \sim \eta$ and 
\begin{equation}
    \Phi(\sigma,\eta)-H(\sigma)= \max \{ H(\xi)-H(\sigma), 0\} \leq 2J,
\end{equation}
since $D_x(\sigma) \leq 1$, where $D_x(\sigma)$ is defined in \eqref{def:number_interfaces}.

In the following, assume that there exists a zero--carpet at distance one from $x$, where $x$ is one of the two sites of the bond. Suppose without loss of generality that $\sigma(x)=-1$.
We apply Lemma \ref{lem:transport_zero}-(ii) and we obtain 
\begin{equation}
    \Phi(\sigma^{(x;0)},\sigma)-H(\sigma) \leq 14 J \qquad \text{ and } \qquad H(\sigma^{(x;0)})-H(\sigma) \leq -2J+(\lambda-h),
\end{equation}
since this minus spin is at distance one from a plus spin and a zero spin of the carpet, then it has at most two nearest neighbors with the same value of the spin. This implies that $D_x(\sigma) \leq 1$, where $D_x(\sigma)$ is defined in \eqref{def:number_interfaces}.
\end{proof}

% \begin{lemma}\label{lem:bond_pm_different_phase}
% Let $\sigma$ be a configuration that contains a bond of type $(+,-)$. If there exists a strip of $s \in \{-1,+1\}$  with width at least 3 connecting the particle of type $-s$ of the bond with the boundary of $\Lambda$, then $V_\sigma<...$.
% \end{lemma}

\begin{proof}[Proof of Lemma \ref{lem:bond_pm_different_phase}]
Let $\sigma$ as in the assumption. Suppose that $\sigma_{\partial^+ \Lambda_0}=0_{\partial^+ \Lambda_0}$, otherwise $V_\sigma=0$, indeed when a plus or a minus leaves $\Lambda$ the energy decreases. 
Suppose without loss of generality that $s=-1$ and that the plus belonging to the bond $(+,-)$ is at distance one from the minus-carpet. We consider the configuration $\sigma^{(x;0)}$, where $x$ is the site of the bond with plus spin.  Assume first that this plus has at most two pluses as nearest neighbours, then by
using Lemma \ref{lemma:transport_scarpet}-(i), we have
\begin{equation}
   \Phi (\sigma, \sigma^{(x;0)})-H(\sigma) \leq 54 J \qquad \text{ and } \qquad H(\sigma^{(x;0)})-H(\sigma) \leq -2J+(\lambda+h).
\end{equation}
Now, we suppose that this plus spin has three pluses as nearest neighbours. 
We analyze the two columns (or rows) that contains the bond $(+,-)$ until we possibly find a bond different 
from $(+,-)$. 
In the case that 
$\sigma$ contains two columns composed by 
all bonds $(+,-)$ in $\Lambda_0$, we look at the bond of these two columns in $\partial^-\Lambda_0$ and we note that the plus of this bond has at most two nearest neighbours equal to plus, so we proceed as before. 
In the case that there exists a bond different 
from $(+,-)$ in the two columns, we distinguish two cases:
(i) if the pair of spins in the bond is of the type $(s,r)$ with $s \neq +1$ and $r \in \{-1,0,+1\}$, then the nearest bond $(+,-)$ contains a plus with at most two pluses nearest neighbours and we proceed as before;
(ii) if the pair of spins in the bond is of the type $(+,r)$, then the nearest bond $(+,-)$ contains a minus with at most two minus nearest neighbours and we conclude by applying Lemma \ref{lemma:transport_scarpet}-(i) to this minus.
\end{proof}

\begin{proof}[Proof of Lemma \ref{lem:p_distance_2_0carpet}]
Let $\sigma$ be a configuration as in the assumption. Suppose that $\sigma_{\partial^+ \Lambda_0}=0_{\partial^+ \Lambda_0}$, otherwise $V_\sigma=0$, indeed when a plus or a minus leaves $\Lambda$ the energy decreases. 
Assume first $\sigma(x_4) \neq -1$ and we consider the configuration $\eta=\sigma^{(x;+)}$. Then, we conclude by applying Lemma \ref{lem:transport_zero}-(i). 
On the other hand, we suppose $\sigma(x_4) =-1$ with at most two nearest neighbors equal to $-1$.
We apply Lemma \ref{lem:transport_zero}-(ii) to $x_4$ and we obtain a configuration $\xi=\sigma^{(x_4;0)}$ 
such that $\Phi(\sigma, \xi) \leq H(\sigma)+14J$ and $H(\xi) \leq H(\sigma)+(\lambda-h)$.
The configuration $\xi$ satisfies the assumption of Lemma \ref{lem:transport_zero}-(i), 
thus we define the configuration $\eta=\xi^{(x;+)}$ 
such that $\Phi(\xi, \eta) \leq H(\xi)+18J$ and $H(\eta) \leq H(\xi)-(\lambda+h)$. 
Then, $H(\eta)<H(\sigma)$ and finally
$\Phi(\sigma,\eta) \leq \max \{ \Phi(\sigma, \xi), \Phi(\xi, \eta) \} \leq 18 J.$
\end{proof}

\begin{proof}[Proof of Lemma \ref{lem:shrink_rectangle}]
Let $\sigma_0=\sigma$ be a configuration that contains a cluster of pluses as in the assumption. The proof for a cluster of minuses is similar.
Suppose that $\sigma_{\partial^+ \Lambda_0}=0_{\partial^+ \Lambda_0}$, otherwise $V_\sigma=0$, indeed when a plus or a minus leaves $\Lambda$ the energy decreases. 
We call $x_1,...,x_l$ the sites along the convex side of the cluster and let $x_1$ be the corner along this side at distance one from the zero-carpet. 
We define a path between $\sigma_0$ and the configuration $\sigma_l$ obtained from $\sigma_0$ by shrinking the cluster of pluses, formally: 
\begin{align}
   & \sigma_i=\sigma_{i-1}^{(x_i;0)} \text{ for } i=1,...,l.
\end{align}
We note that these configurations are not communicating, so we construct a path from $\sigma_{i-1}$ to $\sigma_i$ for every $i=1,...,l$ in the following way. For $i=1$, the plus in $x_1$ moves along the zero-carpet until $\partial^+ \Lambda_0$ and then it leaves $\Lambda$. By using Lemma \ref{lem:transport_zero}-(i), we have $H(\sigma_1)-H(\sigma_0)\leq \lambda+h$ and $\Phi (\sigma_0,\sigma_1) \leq H(\sigma_0)+14J$. For $i=2$, the plus in $x_2$ is at distance one from the zero-carpet in $\sigma_1$ obtained by adding the zero in $x_1$ to the zero-carpet in $\sigma_0$, then it moves along the zero-carpet until $\partial^+ \Lambda_0$ and then it leaves $\Lambda$. As before, by using Lemma \ref{lem:transport_zero}-(i), we have $H(\sigma_2)-H(\sigma_1) \leq \lambda+h$ and $\Phi (\sigma_1,\sigma_2) \leq H(\sigma_1)+14J$. Thus, 
\begin{align}
    \Phi(\sigma_0,\sigma_2)= \max \{\Phi (\sigma_0,\sigma_1), \Phi (\sigma_1,\sigma_2)\} \leq H(\sigma_1)+14J \leq H(\sigma_0)+(\lambda+h)+14J
\end{align}
We iterate this procedure for every $i=3,...,l-1$ and we obtain
$H(\sigma_{l-1})-H(\sigma_{l-2})\leq \lambda+h$ and $\Phi (\sigma_{l-2},\sigma_{l-1}) \leq H(\sigma_{l-2})+14J$. Thus, 
\begin{align}
    \Phi(\sigma_0,\sigma_{l-1}) &= \max \{ \Phi (\sigma_0,\sigma_1), \Phi (\sigma_1,\sigma_2), ..., \Phi (\sigma_{l-2},\sigma_{l-1})\} \notag \\
    & \leq H(\sigma_{l-2})+14J \leq H(\sigma_0)+(l-1)(\lambda+h)+14J.
\end{align}
When the last plus in $x_l$ is detached from the cluster of pluses and it leaves $\Lambda$, by using Lemma \ref{lem:transport_zero}-(i), we have $H(\sigma_l)-H(\sigma_{l-1})\leq -2J+\lambda+h$ and $\Phi (\sigma_{l-1},\sigma_l)\leq H(\sigma_{l-1})+14J \leq  \Phi (\sigma_{l-2},\sigma_{l-1})$. Then, we have
\begin{align}
    H(\sigma_l)\leq H(\sigma_0)-2J+l(\lambda+h) < H(\sigma_0),
\end{align}
where the last inequality follows from the assumption $l \leq \lfloor \frac{2J}{\lambda+h}\rfloor$, and
\begin{align}
    \Phi(\sigma_0,\sigma_{l}) \leq H(\sigma_0)+(l-1)(\lambda+h)+14J < H(\sigma_0)+16J.
\end{align}
\end{proof}

\begin{proof}[Proof of Lemma \ref{lem:grow_rectangle}]
Let $\sigma_0=\sigma$ be a configuration that contains a cluster of pluses as in the assumption. The proof for a cluster of minuses is similar.
Suppose that $\sigma_{\partial^+ \Lambda_0}=0_{\partial^+ \Lambda_0}$, otherwise $V_\sigma=0$, indeed when a plus or a minus leaves $\Lambda$ the energy decreases. 
If there is at least a plus spin at distance strictly smaller than two from the convex side, then we conclude by applying Lemma \ref{lem:transport_zero}-(ii). Otherwise, all sites at distance smaller than two are zeros and we define a path between $\sigma_0$ and the configuration $\sigma_l$ obtained from $\sigma_0$ by adding a strip of pluses with length $l$ to the cluster of pluses. More precisely, we can call $x_1,...,x_l$ the sites at distance one from the convex side that form a strip of zeros in such a way that the transition from $\sigma_0$ to $\sigma_l$ can be realized through the following sequence of configurations:
\begin{align}
   & \sigma_i=\sigma_{i-1}^{(x_i;+)} \text{ for } i=1,...,l.
\end{align}
We note that these configurations are not communicating, so we create a path from $\sigma_{i-1}$ to $\sigma_i$ for every $i=1,...,l$ in the following way. For $i=1$, a plus gets in $\partial^+\Lambda_0$ and it moves along the zero-carpet until $x_1$. By using Lemma \ref{lem:transport_zero}-(i), we have $H(\sigma_1)-H(\sigma_0)\leq 2J-(\lambda+h)$ and $\Phi (\sigma_0,\sigma_1)\leq H(\sigma_0)+18J$. For $i=2$, another plus gets in $\partial^+\Lambda_0$ and it moves along the zero-carpet until in $x_2$. By using Lemma \ref{lem:transport_zero}-(i), we have $H(\sigma_2)-H(\sigma_1)\leq 2J-(\lambda+h)$ and $\Phi (\sigma_1,\sigma_2)\leq H(\sigma_1)+18J$, note that the $2J$ in the estimate of the energy difference is not present if $x_1$ and $x_2$ are nearest neighbours.
Thus, 
\begin{align}
    \Phi(\sigma_0,\sigma_2)= \max \{\Phi (\sigma_0,\sigma_1), \Phi (\sigma_1,\sigma_2)\} \leq H(\sigma_1)+18J \leq H(\sigma_0)+22J-(\lambda+h)
\end{align}
We iterate this procedure for every $i=3,...,l-1$. We stress that from $i=3$ on each new plus spin will be accommodated at a site with at least two neighboring pluses, so that 
$H(\sigma_{i})-H(\sigma_{i-1})\leq -(\lambda+h)$ and $\Phi (\sigma_{i-1},\sigma_{i}) \leq H(\sigma_{i-1})+18J$. Thus,  
$\Phi(\sigma_{i-1},\sigma_{i}) \leq \Phi(\sigma_{i-2},\sigma_{i-1})$,
since $H(\sigma_{i})<H(\sigma_{i-1})$ for $i=1,...,l$. Then,
\begin{align}
    \Phi(\sigma_0,\sigma_{l}) &=\max \{\Phi (\sigma_0,\sigma_1), \Phi (\sigma_1,\sigma_2),..., \Phi (\sigma_{l-1},\sigma_l)\} \notag \\
    &= \max \{\Phi (\sigma_0,\sigma_1), \Phi (\sigma_1,\sigma_2)\}<H(\sigma_0)+22J.
\end{align}
Finally, we conclude
\begin{align}
    H(\sigma_l)-H(\sigma_0) \leq H(\sigma_0)+2J-l(\lambda+h) < H(\sigma_0),
\end{align}
since $l \geq \lfloor \frac{2J}{\lambda+h}\rfloor+1$.
\end{proof}

 \begin{proof}[Proof of Lemma \ref{lem:flag}]
 Let $\sigma$ be a configuration as in the assumption. Suppose that $\sigma_{\partial^+ \Lambda_0}=0_{\partial^+ \Lambda_0}$, otherwise $V_\sigma=0$, indeed when a plus or a minus leaves $\Lambda$ the energy decreases. 
 Let $l$ be the length of the strip of pluses. We distinguish two cases:
 \begin{itemize}
     \item[(i)] Consider the flag-shaped structure in the right panel of Figure \ref{fig:flag}. 
     \item[(ii)] Consider the flag-shaped structure in the left panel of Figure \ref{fig:flag}.
 \end{itemize}
 We start with the case (i). Assume first $l \geq \lfloor \frac{2J}{h}+\frac{\lambda-h}{h} \rfloor$+1.
 
 Let $x_1,...,x_l$ be the sites in the flag-shaped structured with minus spin and such that $|x_i-x_{i+1}|=1$ for every $i=1,...,l-1$.
 Let $y_1,...,y_l$ be the sites in the flag-shaped structured with zero spin and such that $|y_i-x_i|=1$ for every $i=1,...,l$.
 Let $z_1$ and $z_2$ be the two sites outside the flag-shaped structured such that $\sigma(z_1)=\sigma(z_2)=-1$, $|z_1-y_1|=1$ and $|z_2-y_l|=1$. 

First, we note that if $z_1$, $z_2$ have three nearest neighbours with minus spin, then there exists a bond $(+,-)$ and we conclude by applying Lemma \ref{lem:bond_pm_different_phase}. Then, assume that $z_1$ and $z_2$ have at most two nearest neighbours with minus spins.
Moreover, we suppose that $z_j$ and $x_i$ for every $j=1,2$ and every $i=1,...,l$ have only nearest neighbours with minus or zero spins, otherwise we conclude by applying Lemma \ref{lem:bond_pm_different_phase}.

We construct a path from $\sigma$ to $\sigma_{2l+2}$ where $\sigma_{2l+2}$ is such that
 \begin{align}
     \sigma_{2l+2}(k)= 
     \begin{cases}
         \sigma(k) & \qquad \text{ for each } k \neq x_i, y_i, z_j \text{ for every } i=1,...,l \text{ and every } j=1,2 \\
         0 & \qquad \text{ for each } k = x_i, z_j \text{ for every } i=1,...,l \text{ and every } j=1,2 \\
         +1 & \qquad \text{ for each } k = y_i \text{ for every } i=1,...,l 
     \end{cases}
 \end{align} 
Starting from $\sigma$, we apply Lemma \ref{lemma:transport_scarpet}-(i) and we obtain the configuration $\sigma_1=\sigma^{(z_1;0)}$ such that
$H(\sigma_1) \leq H(\sigma)+\lambda-h$ and $\Phi(\sigma,\sigma_1) \leq H(\sigma)+54J$. 
Then, we apply Lemma \ref{lemma:transport_scarpet}-(i) to $\sigma_1$ and we obtain the configuration $\sigma_2=\sigma_1^{(z_2;0)}$ such that
$H(\sigma_2) \leq H(\sigma_1)+\lambda-h$ and $\Phi(\sigma_1,\sigma_2) \leq H(\sigma_1)+54J$. Thus $\Phi((\sigma,\sigma_1,\sigma_2)) \leq H(\sigma)+54J+\lambda-h$.
We apply Lemma \ref{lemma:transport_scarpet}-(i) to $\sigma_2$ and we obtain the configuration $\sigma_3=\sigma_2^{(x_1;0)}$ such that
$H(\sigma_3) \leq H(\sigma_2)+2J+\lambda-h$, since $x_1$ may have three nearest neighbours equal to minus, and $\Phi(\sigma_2,\sigma_3) \leq H(\sigma_2)+54J$. Thus $\Phi((\sigma,...,\sigma_3)) \leq H(\sigma)+54J+2(\lambda-h)$.
Next, we define the configuration $\sigma_4=\sigma_3^{(y_1;+)}$ obtained by applying Lemma \ref{lem:bond_pm_different_phase}-(ii) to $\sigma_3$. We have $H(\sigma_4) \leq H(\sigma_3)+2J-(\lambda+h)$ and $\Phi(\sigma_3,\sigma_4) \leq H(\sigma_3)+35J$. Thus $\Phi((\sigma,...,\sigma_4)) \leq H(\sigma)+54J+2(\lambda-h)$, since by a direct computation we have $\Phi(\sigma_3,\sigma_4) - H(\sigma_3) \leq \Phi((\sigma,...,\sigma_3)) -H(\sigma)$.
Then, we apply Lemma \ref{lemma:transport_scarpet}-(i) to $\sigma_4$ and we obtain the configuration $\sigma_5=\sigma_4^{(x_2;0)}$ such that
$H(\sigma_5) \leq H(\sigma_4)+(\lambda-h)$, since $x_2$ may have at most two nearest neighbours equal to minus, and $$\Phi(\sigma_4,\sigma_5) \leq H(\sigma_4)+54J \leq H(\sigma_3)+56J-(\lambda+h) \leq H(\sigma)+58J+2(\lambda-2h).$$
Thus $\Phi((\sigma,...,\sigma_5)) \leq H(\sigma)+58J+2(\lambda-2h)$.
Next, we define the configuration $\sigma_6=\sigma_5^{(y_2;+)}$ obtained by applying Lemma \ref{lem:bond_pm_different_phase}-(ii) to $\sigma_5$. We have $H(\sigma_6) \leq H(\sigma_5)-(\lambda+h)$ and $\Phi(\sigma_5,\sigma_6) \leq H(\sigma_5)+35J$. Thus, $\Phi((\sigma,...,\sigma_6) \leq H(\sigma)+58J+2(\lambda-2h)$, since $\Phi(\sigma_5,\sigma_6) \leq \Phi((\sigma,\sigma_5))$.

The rest of the path is a two-step down-hill path. Indeed,
we iterate the two last steps for $i=3,...,l$ and we obtain
\begin{align}
    & H(\sigma_{2i+1}) \leq H(\sigma_i)+(\lambda-h) \label{eq:flag1} \\
    & H(\sigma_{2i+2}) \leq H(\sigma_{2i+1})-(\lambda+h) \label{eq:flag2}
\end{align}
and 
\begin{align}
    \Phi((\sigma, ..., \sigma_{2i+1}))=\Phi((\sigma, ..., \sigma_{2i+1})) 
    \leq H(\sigma)+58J+2(\lambda-2h)<58J
\end{align}
where the last inequality follows from the assumption $\lambda<2h$.

In order to conclude the proof, we show that $H(\sigma_{2l+2}) < H(\sigma)$. By \eqref{eq:flag1} and \eqref{eq:flag2}, we have 
 \begin{align}
     H(\sigma_{2l+2})& \leq H(\sigma_{2l+1})-(\lambda+h) \leq H(\sigma_{2l})+(\lambda-h)-(\lambda+h) \notag \\
     &\leq ...\leq H(\sigma_{4})+(l-1)(\lambda-h)-(l-1)(\lambda+h) \notag \\
     &\leq H(\sigma_3)+2J-l(\lambda+h)+(l-1)(\lambda-h) \notag \\
     &\leq H(\sigma_2)+4J-l(\lambda+h)+l(\lambda-h) \notag \\
     &\leq H(\sigma)+2(\lambda-h)+4J-2lh<H(\sigma),
 \end{align}
where the last inequality follows from $l \geq \lfloor \frac{2J}{h}+\frac{\lambda-h}{h} \rfloor+1$.

 Next, suppose $l \leq \lfloor \frac{2J}{h}+\frac{\lambda-h}{h} \rfloor$. Let $x_i, y_i, z_j$ be as in the previous case for $i=1,...,l$ and $j=1,2$. Let $v_1,...,v_l$ be the sites in the flag-shaped structured with plus spin and such that $|v_i-y_i|=1$ for every $i=1,...,l$. We consider a configuration $\eta$ such that
  \begin{align}
     \eta(k)= 
     \begin{cases}
         \sigma(k) & \qquad \text{ for each } k \neq y_i,v_i \text{ for every } i=1,...,l \\
         0 & \qquad \text{ for each } k = v_i \text{ for every } i=1,...,l \text{ and every } j=1,2 \\
         -1 & \qquad \text{ for each } k = y_i \text{ for every } i=1,...,l 
     \end{cases}
 \end{align} 
We construct a path from $\sigma$ to $\eta$ by replacing every plus in the flag-shaped structured with a zero and by exchanging every zero with a minus. The transport of each particle takes place through the minus-carpet as before. By arguing as above and recalling $l \leq \lfloor \frac{2J}{h}+\frac{\lambda-h}{h} \rfloor$, we obtain
$H(\eta)<H(\sigma)$ and $\Phi(\sigma, \eta)-H(\sigma) \leq 58J$.

 The cases (ii) is similar to case (i) considering the two cases $l \leq \lfloor \frac{2J}{h}+\frac{\lambda-h}{2h} \rfloor$ and $l \geq \lfloor \frac{2J}{h}+\frac{\lambda-h}{2h} \rfloor$+1.

 \end{proof}

\begin{proof}[Proof of Lemma \ref{lemma:path_0_+}]
To prove the result, we construct a path $\underline{\omega}$ from $\zero$ to $\puno$ as a sequence of configurations from $\zero$ to $\puno$ with increasing clusters \emph{as close as possible to a quasi-square}, see Figure \ref{fig:path_+_0}. 
\begin{figure}[t]
\begin{center}
    \includegraphics[scale=0.6]{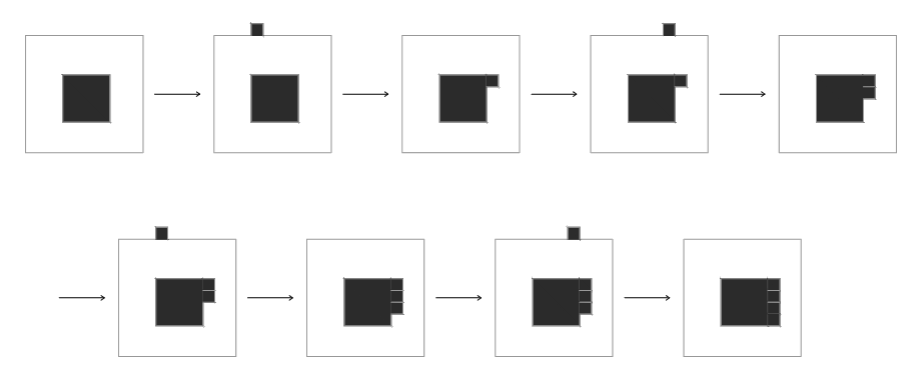}
    \caption{In the figure, the white part represents the region with zero spins, the dark gray region is the cluster of pluses. We observe that the configurations in this figure belongs to the path $\underline{\omega}$ but they are not communicating, indeed the free plus in the boundary must move in the box before to reach the cluster. }\label{fig:path_+_0}
\end{center}
    \end{figure}
We construct a path in which at each step a particle of plus gets in $\Lambda$ and then it is attached to the cluster of pluses constructed before. We observe that every time that a particle gets in $\Lambda$ the energy cost is $\Delta^p$. The first two pluses  get in $\Lambda$ %with an energy cost equal to $\Delta^p$ 
and they possibly moves in the box until they are attached. When they are attached, a bond between plus spins is created and the energy decreases by $2J$. Then, another plus gets in $\Lambda$ and it moves towards the cluster of two particles. First the energy increases by $\Delta^p$ and then, when the plus reaches the cluster, it decreases by $2J$. The fourth plus get in $\Lambda$ with the same cost, but when it is attached to the cluster we obtain a square $2 \times 2$ and two bonds between pluses are created, so the energy decreases by $4J$. 
Next, another two particles get in $\Lambda$ and they are attached clockwise to the square, so we obtain a quasi-square $2 \times 3$ and so on. We note that when a particle is attached to the cluster the energy of the system decreases by at least $2J$. In particular, when we attach the first particle along a side of the quasi-square the energy decreases by $2J$, while the energy decreases by $4J$ when we attach the other particles along the same side. 
Next, we iterate this procedure by getting in $\Lambda$, moving and attaching a plus. In this way, we create squares and quasi-squares of pluses consequentially, until the cluster invades all the space $\Lambda_0$. In the following we compute the height of this path. First of all, we compute the energy of a configuration $\eta$ that contains a rectangle of pluses with side lengths $m$ and $n$ and only zero spins outside.
\begin{align}\label{eq:difference_energy_rectangle}
    & H(\eta)=H(\zero)+2J(m+n)-(\lambda+h)mn
\end{align}
where $2(m+n)$ is the number of bonds $(0,+)$ and $mn$ is the number of pluses in $\eta$. The equation \eqref{eq:difference_energy_rectangle} attains the maximum in $(m,n)=\Big(\frac{2J}{\lambda+h},\frac{2J}{\lambda+h}\Big)$, that corresponds to a configuration with a square of pluses with side length $\tilde n=\lfloor\frac{2J}{\lambda+h}\rfloor +1$, called $\eta_{\tilde n}$. Starting from $\eta_{\tilde n}$ to reach the configuration $\eta_{\tilde n+1}$ that contains a quasi-square with side length $\tilde n$ and $\tilde n +1$, the energy cost is equal to $6J-2(\lambda+h)$ and it is given by the first three steps with a non-null energy cost. Indeed, this value is obtained by the sum between the positive cost $\Delta^p$ of the entrance of the first plus, the negative cost $-2J$ due to the junction of this plus to the cluster and the positive cost $\Delta^p$ of the entrance of the second plus $\Delta^p$. The rest of the path is a two-step downhill path indeed, a particle moves in $\Lambda$ by increasing the energy by $\Delta^p$ and then it is attached to the cluster by decreasing the energy by $4J>\Delta^p$, see Figure \ref{fig:path_+_0}. Thus, recalling that $H(\zero)=0$, using the value of $\tilde n$ and the assumption $2h>\lambda>h$, we have
\begin{align}
& \Phi(\zero,\puno)-H(\zero) \leq \Phi(\eta_{\tilde n},\eta_{\tilde n+1})-H(\zero) =H(\eta_{\tilde n})+6J-2(\lambda+h)-H(\zero) \notag \\
&=4J\tilde n-(\lambda+h)\tilde n^2+6J-2(\lambda+h) = \frac{4J^2}{\lambda+h} +6J-2(\lambda+h)=V^*.
\end{align}
\end{proof}

\section{Conclusions}
\label{s:con}
\par\noindent
We have studied the metastable behavior of the Blume--Capel model 
with magnetic field smaller than the chemical potential
and we have been able to prove that, for this choice of the 
parameters, the minus homogeneous state is the unique 
metastable state. 
Moreover, we have studied in detail the transition 
from the metastable to the stable homogeneous plus 
state and we have provided an estimate of the 
exit time.

As we have explained above, see also the 
introductory section, the solution of the variational problems 
involved in the study of metastability is particularly 
difficult, because of the conservative 
nature of the Kawasaki dynamics and the 
interplay among the energy costs of 
different types of interfaces
(the energy cost of the minus--plus lattice bonds 
is higher than that of the zero--plus 
and zero--minus ones). 

The main problem is that, in order to minimize
the energy cost of the interacting structure of pluses and minuses, 
it is necessary to separate them by means of a thin layer of zeros. 
Thus, the transition from the metastable minus state 
to the stable plus one must happen via the formation and 
growth of a droplet of pluses surrounded by a layer of zeros.
This, associated with the swap character of the dynamics, 
is a huge problem, since, during the growth, it is not 
possible to create the necessary zeros and pluses where they are needed, 
but it is necessary to transport 
them from the boundary to these particular lattice sites
through a completely arbitrary mixture of the three 
spin species.

We have solved this problem by introducing the idea of carpet, 
namely, a nearest neighbor connected set of lattice sites with 
constant spin value, and using these structures to transport spins 
along the lattice. The key lemma on which our study is based 
is Lemma~\ref{thm:grafo_pesato} which allows us to control
this transport mechanism and to estimate the involved 
energy costs. 

For the sake of clearness, we have stated and proven this lemma 
in the context of the Blume--Capel model, but it is useful to remark that 
it is possible to state it in a more general setup, namely for multi--state spin systems. 
The generalized version of the lemma opens the way for the study of 
the metastable behavior of more general multi--state systems.

We thus wrap up this paper by stating the general version of the lemma.
Let us, then, consider the generalized Potts model with 
spin variables taking value in $\{1,...,q\}$, with $q \in \mathbb{N}$. 
Given the lattice $\Lambda$ with periodic boundary conditions, 
the configuration space is $\mathcal{X}= \{1,...,q\}^\Lambda$ 
and the Hamiltonian function is given by

\begin{align}\label{eq:H_generalize_Potts}
    H(\sigma)  =
 \sum_{\newatop{x,y \in \Lambda:}{|x-y|=1}} 
\!\!
    J(\sigma(x), \sigma(y))
 +
\sum_{x\in\Lambda}
h(\sigma(x)),
\end{align}
where $\sigma\in\mathcal{X}$ is a configuration,
$J:\{1,\dots,q\}\times\{1,\dots,q\}\to\mathbb{R}$ gives the energy cost of the 
bond with spins $\sigma(x)$ and $\sigma(y)$, and $h:\{1,\dots,q\}\to\mathbb{R}$ is the 
magnetic field. 

In the following we misuse the notation introduced in the paper for the 
Blume--Capel model and, mutatis mutandis, apply it to the generalized Potts
model.

\begin{lemma}[Weighted graph for the energy cost in $s$--carpet]
\label{thm:grafo_pesato_canale_Potts}
Let $s \in \{1,...,q\}$.
Consider a configuration 
$\sigma \in \mathcal{X}$ such that there exists 
a $s$-carpet $X$ of $\sigma$ and two nearest neighboring sites
$x \in \Lambda\setminus X$ and $x'\in X$ such that 
$\sigma(x)=r \neq s$.
Assume, also, that 
there exists $y \neq x'$ a nearest neighbor of $x$ such 
that $\sigma(y)=s$. 
Let $X'$ be the subset of $X$ obtained by collecting $x'$ together 
with all the sites of $X$ having at least two 
neighboring sites in $X$.
Then the following holds:
\begin{enumerate}
\item
for any $v\in X'\cup\{x\}$,
let
$\sigma^v=\sigma^{(x,v)}$  and recall $\sigma^{(x,x)}$ is equal to $\sigma$.
Then, for any pair of neighboring 
sites $v,w\in X'\cup\{x\}$ 
\begin{equation}
H(\sigma^v)-H(\sigma^w)
= [D_v(\sigma^v)+D_w(\sigma^v)]-[D_v(\sigma^w)+D_w(\sigma^w)] 
\end{equation}
where 
\begin{equation}
    D_x(\eta):=
    \left\{
    \begin{array}{cl}
{\displaystyle 
   \sum_{\substack{y \in \Lambda:\\ |x-y|=1}} 
J(\eta(x),\eta(y)) 
}
& 
\textup { if } x\in\Lambda_0, 
\\
 0 &\textup { if } x \in \partial^+\Lambda_0.
\end{array}
\right.
\end{equation}
\item
For any $v\in X'$ 
\begin{equation}
    \Phi(\sigma, \sigma^v)-H(\sigma) \leq \max_{v\in X'} \Big ( [D_v(\sigma^v)-D_v(\sigma)]+[D_x(\sigma^v)-D_x(\sigma)] \Big )
.
\end{equation}
\end{enumerate}
\end{lemma}

\begin{figure}[t]
\centering
\raisebox{-0.5\height}{\includegraphics[width=0.45\textwidth]{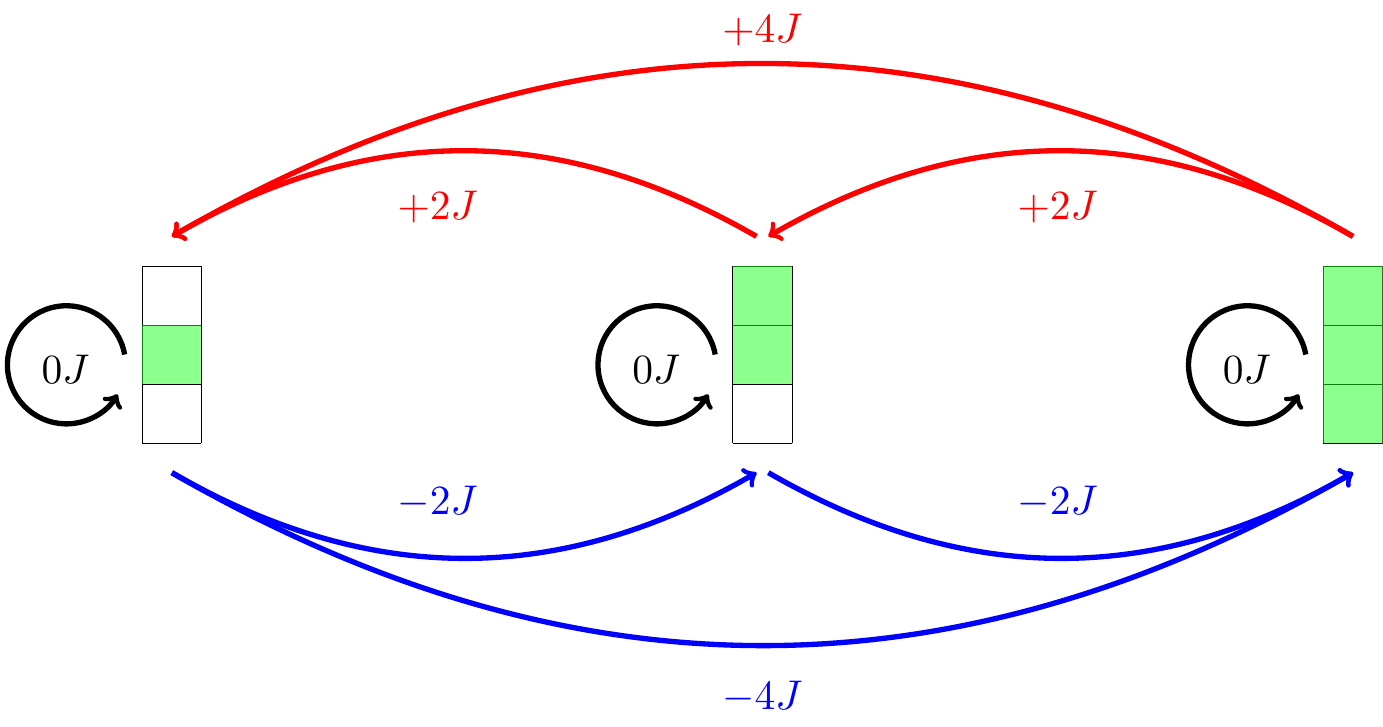}}
\caption{In each vertex of the graph 
the three squares represent three sites such that the one at the 
center is the site $x$ and it is nearest neighbor of the other two.
The local configuration is reported with the colors 
green and white representing, respectively, 
the spin values $r$, and any other spin different from $r$. 
The other two sites neighboring the middle square and not reported in the picture 
are occupied by a fixed $s\neq r$ spin, indeed, in this example, we supposed that the configuration contains a $s$-carpet. 
The vertex at the center of the graph corresponds to any 
of the two reported configurations.
Any transition 
from $\sigma^v$ to $\sigma^w$ for 
$v,w\in X'\cup\{x\}$ 
is realized via the swap of the particle
between two sites with local configuration represented by 
one of the vertices of the graph.
For each possible swap the energy difference is reported in the picture.
}
\label{figure:graph_carpet_Potts}
\end{figure}

In order to illustrate an interesting application of the 
Lemma~\ref{thm:grafo_pesato_canale_Potts} we consider the 
case of the ferromagnetic standard $q$--state Potts model and compute explicitely
the energy difference $H(\sigma^v)-H(\sigma^w)$ and the communication height $\Phi(\sigma, \sigma^v)-H(\sigma)$. In this case the interaction term is a negative constant
for equal spins and zero otherwise, i.e.,
\begin{align}\label{eq:H_standard_Potts}
        H(\sigma)  =-J
 \sum_{\newatop{x,y \in \Lambda:}{|x-y|=1}}
\!\!
    \textbf{1}_{\{\sigma(x)=\sigma(y)\}}
  +
\sum_{x\in\Lambda}
h(\sigma(x))
\end{align}
with $J>0$ and $\textbf{1}$ is the characteristic function.
Thus in this case, we have
\begin{equation}
    H(\sigma^v)-H(\sigma^w) \in \{ -4J, -2J, 0, 2J, 4J \} 
    \;\textup{ and }\;
     \Phi(\sigma, \sigma^v)-H(\sigma) \leq 4J,
\end{equation}
see Figure \ref{figure:graph_carpet_Potts} for details.

\renewcommand{\thesubsection}{\Alph{section}.\arabic{subsection}}
\renewcommand{\theequation}{\Alph{section}.\arabic{equation}}
\renewcommand{\thefigure}{\Alph{section}.\arabic{figure}}

%Per i ringraziamenti
\begin{acknowledgments}
ENMC thanks the Institute of Mathematics of the Utrecht University
for the warm hospitality.
ENMC thanks the PRIN 2022 project
``Mathematical modelling of heterogeneous systems"
(code 2022MKB7MM, CUP B53D23009360006).
VJ thanks GNAMPA.

\end{acknowledgments}

\paragraph{Data availability}
Data sharing not applicable to this article as no datasets were generated or analysed during the current study.

\paragraph{Conflict of interest}
The authors declare that they have no conflict of interest.
\bibliographystyle{plain}
\bibliography{cjs-bck_zero}
\end{document}